\documentclass[11pt]{article}
\usepackage{fullpage,amssymb}
\usepackage{caption}
\usepackage{subcaption}
\usepackage{graphicx}
\usepackage{amsmath}
\usepackage{float}
\usepackage{listings}
\usepackage{xcolor}
\usepackage{amsthm}
\usepackage{algorithmic}
\usepackage{algorithm}
\usepackage[colorlinks,
linkcolor=blue,
anchorcolor=blue,
citecolor=blue]{hyperref}
\usepackage{natbib}
\usepackage{comment}
\usepackage{bbm}
\usepackage{multirow}
\usepackage{array}
\usepackage[shortlabels]{enumitem}
\usepackage{romannum}
\AtBeginDocument{\pagenumbering{arabic}}
\usepackage{multirow}
\usepackage{stmaryrd}
\usepackage{amsmath}

\setcitestyle{authoryear,round}

\newtheorem{theorem}{Theorem}
\newtheorem{lemma}{Lemma}
\newtheorem{definition}{Definition}

\newtheorem{corollary}{Corollary}
\newtheorem{assumption}{Assumption}
\newtheorem{remark}{Remark}

\newtheorem{example}{Example}

\def\l{\left}
\def\r{\right}
\def\wh{\widehat}

\def\EE{{\mathbb E}}

\def\II{{\mathbb I}}

\def\PP{{\mathbb P}}

\def\RR{{\mathbb R}}
\def\SS{{\mathbb S}}

\def\eps{\varepsilon}

\renewcommand{\ldots}{\cdots}

\begin{document}
\title{Sharp Concentration Inequalities: Phase Transition and Mixing of Orlicz Tails with Variance \thanks{Yinan Shen is Assistant Professor, RTPC (in fact postdoc) in Department of Mathematics, University of Southern California, Los Angeles, CA 90089 (E-mail: \textit{yinanshe@usc.edu}; ORCID: 0000-0001-9146-9549). 
		Jinchi Lv is Kenneth King Stonier Chair in Business Administration and Professor, Data Sciences and Operations Department, Marshall School of Business, University of Southern California, Los Angeles, CA 90089 (E-mail: \textit{jinchilv@marshall.usc.edu}; ORCID: 0000-0002-5881-9591).}}
\author{Yinan Shen and Jinchi Lv\\
University of Southern California\\
}

\maketitle



\begin{abstract}
In this work, we investigate how to develop sharp concentration inequalities for sub-Weibull random variables, including sub-Gaussian and sub-exponential distributions. Although the random variables may not be sub-Guassian, the tail probability around the origin behaves as if they were sub-Gaussian, and the tail probability decays align with the Orlicz $\Psi_\alpha$-tail elsewhere. Specifically, for independent and identically distributed (i.i.d.) $\{X_i\}_{i=1}^n$ with finite Orlicz norm $\|X\|_{\Psi_\alpha}$, our theory unveils that there is an interesting phase transition at $\alpha = 2$ in that $\PP\l(\l|\sum_{i=1}^n X_i \r| \geq t\r)$ with $t > 0$ is upper bounded by $2\exp\l(-C\max\l\{\frac{t^2}{n\|X\|_{\Psi_{\alpha}}^2},\frac{t^{\alpha}}{ n^{\alpha-1} \|X\|_{\Psi_{\alpha}}^{\alpha}}\r\}\r)$ for $\alpha\geq 2$, and by $2\exp\l(-C\min\l\{\frac{t^2}{n\|X\|_{\Psi_{\alpha}}^2},\frac{t^{\alpha}}{ n^{\alpha-1} \|X\|_{\Psi_{\alpha}}^{\alpha}}\r\}\r)$ for $1\leq \alpha\leq 2$ with some positive constant $C$. In many scenarios, it is often necessary to distinguish the standard deviation from the Orlicz norm when the latter can exceed the former greatly. To accommodate this, we build a new theoretical analysis framework, and our sharp, flexible concentration inequalities involve the variance and a mixing of Orlicz $\Psi_\alpha$-tails through the min and max functions. Our theory yields new, improved concentration inequalities even for the cases of sub-Gaussian and sub-exponential distributions with $\alpha = 2$ and $1$, respectively. We further demonstrate our theory on martingales, random vectors, random matrices, and covariance matrix estimation. These sharp concentration inequalities can empower more precise non-asymptotic analyses across different statistical and machine learning applications.
\end{abstract}



\section{Introduction} \label{new.Sec.intro}

Concentration inequalities are central to modern statistics and machine learning, especially in problems where tail probabilities determine finite-sample performance, such as matrix analysis \citep{minsker2017some,koltchinskii2016perturbation,adamczak2011restricted},  decision-making and inference \citep{hao2019bootstrapping,khamaru2025near,lin2025semiparametric}, and robust statistics \citep{minsker2018sub,depersin2022robust,ma2024high}. The goal of this paper is to study concentration inequalities for real-valued sub-Weibull random variables, i.e., random variables $X$ whose tails decay at an exponential--Weibull rate
$$\PP\l(|X|\geq t\r)\leq 2\exp\l(-\frac{t^{\alpha}}{K}\r),$$
where $K, \alpha>0$ are constants, and $t > 0$. The case of $\alpha=2$ corresponds to sub-Gaussian distributions, while the case of $\alpha=1$ corresponds to sub-exponential distributions. Throughout the paper, we work with random variables having finite Orlicz $\Psi_\alpha$-norm, defined as follows.
\begin{definition}[Orlicz $\|\cdot\|_{\Psi_\alpha}$-norm]
    For a given random variable $X$ and $\alpha\geq 1$, we define the Orlicz $\|\cdot\|_{\Psi_\alpha}$-norm as $$\|X\|_{\Psi_{\alpha}}:=\inf_{u>0}\l\{\EE\l\{\exp(|X|/u)^{\alpha} \r\}\leq 2\r\}.$$
\end{definition}

A large literature has contributed to developing concentration theory for sub-Weibull random variables; see, e.g., \cite{bennett1962probability}, \cite{talagrand1989isoperimetry}, \cite{talagrand1994supremum}, \cite{latala1997estimation}, \cite{boucheron2003concentration}, \cite{koltchinskii2011neumann}, \cite{adamczak2011restricted}, \cite{van2013bernstein}, \cite{ledoux2013probability}, \cite{rio2013extensions}, \cite{minsker2017some}, \cite{vershynin2018high}, \cite{hao2019bootstrapping}, \cite{zhang2022sharper}, \cite{kuchibhotla2022moving}, \cite{jeong2022sub}, and references therein. Yet for the fundamental tail probability
\begin{align*}
    \PP\l(\l|\sum_{i=1}^n X_i\r| \geq t\r)
\end{align*}
for random variables $X_i$'s and $t > 0$, the existing results in the literature still leave an important gap. On the one hand, for $\alpha>2$, existing concentration inequalities either require information \textit{stronger} than $\|X\|_{\Psi_\alpha}<\infty$ or \textit{fail} to deliver a sharp $\Psi_\alpha$ large-deviation tail. On the other hand, when the Orlicz norm $\|X\|_{\Psi_\alpha}$ and the standard deviation $\sigma_X$ are \textit{not} of the same order, the literature does \textit{not} simultaneously capture the variance-dominated small-deviation regime and the correct $\Psi_\alpha$-tail for large deviations.

Specifically, for $\alpha>2$, \cite{ledoux2013probability} and \cite{talagrand1989isoperimetry} proved the elegant inequality
\begin{align}
    \l\|\sum_{i=1}^n X_i\r\|_{\Psi_{\alpha}}\leq K_{\alpha}\l(\l\|\sum_{i=1}^nX_i \r\|_1+ \l\| \l(\l\|X_i \r\|_{\Psi_{s}}\r)\r\|_{\beta,\infty}\r),
    \label{eq:talagrand:norm}
\end{align}
where $\alpha<s<\infty$. However, $\l\|X_i \r\|_{\Psi_{s}}$ does not need to be finite: it is possible to have $\l\|X_i \r\|_{\Psi_{s}}=\infty$ while still having $\|X\|_{\Psi_\alpha}<\infty$. In that case, 
\eqref{eq:talagrand:norm} becomes trivial for bounding the $\Psi_\alpha$-norm of $\sum_{i=1}^n X_i$. For $\alpha\in[1,2]$, the same line of work established that 
\begin{align}
    \l\|\sum_{i=1}^n X_i\r\|_{\Psi_{\alpha}}\leq K_{\alpha}\l(\l\|\sum_{i=1}^nX_i \r\|_1+ \l( \sum_{i=1}^n\l\|X_i \r\|_{\Psi_{\alpha}}^\beta\r)^{\frac{1}{\beta}}\r),
    \label{eq:talagrand:norm2}
\end{align}
which has a particularly clean form, but still does not separate the role of variance from that of the Orlicz norm and as we will demonstrate later, Orlicz norm is not always a sharp characterization of tail probability. From a different perspective, the foundational work of \cite{koltchinskii2011neumann} showed that for $\alpha\geq 1$,
\begin{align}
    \PP\l(\l|\sum_{i=1}^n X_i\r|\geq t\r)\leq 2\exp\l(-C\min\l\{\frac{t^2}{n\sigma_X^2},\frac{t}{\|X\|_{\Psi_\alpha}\log^{\frac{1}{\alpha}}\l(\frac{2\|X\|_{\Psi_\alpha}}{ \sigma_X}\r)}\r\}\r),
    \label{eq:koltchinskii:conc}
\end{align}
which is sharp for sufficiently small deviations since it preserves the variance term. However, it yields only a sub-exponential tail for $\sum_{i=1}^n X_i$ regardless of $\alpha\geq 1$.

At the same time, the crude triangle inequality
$$\l\|\sum_{i=1}^n X_i\r\|_{\Psi_\alpha}\leq n\|X\|_{\Psi_\alpha}<\infty$$
shows that $\sum_{i=1}^n X_i$ still has finite $\Psi_\alpha$-norm under the assumption of $\|X\|_{\Psi_\alpha}<\infty$. Although this bound is far from sharp, it strongly suggests that a \textit{sharper} concentration theory should exist. This motivates the following questions:
\begin{center}
    \textit{For the tail probability $\PP\l(\l|\sum_{i=1}^n X_i\r|\geq t\r)$, can one obtain a sharp $\Psi_\alpha$-tail under only $\|X_i\|_{\Psi_\alpha}<\infty$ while simultaneously retaining variance-controlled concentration for sufficiently small $t$? More broadly, what is the corresponding concentration theory when $X_1,\ldots,X_n$ are dependent?}
\end{center}

In this paper, we aim to answer these questions and identify a sharp phase transition at $\alpha=2$ in univariate sub-Weibull concentration. The central novelty of our work is that sums of $\Psi_\alpha$ random variables display local sub-Gaussian behavior around the origin even when the summands themselves are not sub-Gaussian, while their large-deviation behavior retains the correct $\Psi_\alpha$-tail. This yields a sharp, density-free concentration theory in the regime of $\alpha\geq 2$, and Section~\ref{sec:univariate_var} further develops a variance-sensitive theory that remains statistically optimal when $\|X\|_{\Psi_\alpha}$ and $\sigma_X$ are not comparable. The corollary below illustrates the main phenomenon in the independent and identically distributed (i.i.d.) setting.

\begin{corollary}[Concentration for i.i.d. univariate] 
Assume that $X_1,\ldots,X_n$ are i.i.d. mean-zero real-valued random variables, and satisfy $\|X\|_{\Psi_{\alpha}}<\infty$ for some $\alpha\geq 1$. Then we have that 
\begin{equation}
    \begin{split}
        \PP\l(\l|\sum_{i=1}^n X_i \r| \geq t\r)&\leq 2\exp\l(-C\max\l\{\frac{t^2}{n\|X\|_{\Psi_{\alpha}}^2},\frac{t^{\alpha}}{ n^{\alpha-1} \|X\|_{\Psi_{\alpha}}^{\alpha}}\r\}\r) \ \ \text{ for } \alpha\geq 2,\\
      \PP\l(\l|\sum_{i=1}^n X_i \r| \geq t\r)&\leq 2\exp\l(-C\min\l\{\frac{t^2}{n\|X\|_{\Psi_{\alpha}}^2},\frac{t^{\alpha}}{ n^{\alpha-1} \|X\|_{\Psi_{\alpha}}^{\alpha}}\r\}\r) \ \ \text{ for } 1\leq \alpha\leq 2,
      \label{eq2:intro}
    \end{split}
\end{equation}
where $C>0$ is some constant that does not depend on $\alpha,t,n,X$.
  \label{cor:iid_conc_univariate}
\end{corollary}

Corollary~\ref{cor:iid_conc_univariate} above is a direct consequence of Theorem~\ref{thm:conc_univariates} (see Section \ref{sec:univariate}) and already shows why the new theory differs qualitatively from the existing literature. When $\alpha\geq 2$, the decisive feature of 
\eqref{eq2:intro} is the appearance of $\textit{max}$, rather than the familiar $\textit{min}$ in \cite{boucheron2003concentration,kuchibhotla2022moving,zhang2022sharper}; this is precisely the phase transition at $\alpha=2$ and it yields a strictly sharper tail in the large-deviation regime. When $1\leq \alpha\leq 2$, the bound recovers the correct order when $\|X\|_{\Psi_\alpha}\asymp\sigma_X$. In both regimes, the sum is locally sub-Gaussian: for all $\alpha\geq 1$ and $t\leq n\|X\|_{\Psi_\alpha}$,
\begin{align*}
    \PP\l(\l|\sum_{i=1}^n X_i\r|\geq t\r)\leq 2\exp\l(-C\frac{t^2}{n\|X\|_{\Psi_\alpha}^2}\r),
\end{align*}
while for $t\geq n\|X\|_{\Psi_\alpha}$,
\begin{align*}
    \PP\l(\l|\sum_{i=1}^n X_i\r|\geq t\r)\leq 2\exp\l(-C\frac{t^\alpha}{n^{\alpha-1}\|X\|_{\Psi_\alpha}^\alpha}\r),
\end{align*}
which is consistent with the fact that $\sum_{i=1}^n X_i$ has finite $\Psi_\alpha$-norm and sharpens substantially the crude triangle inequality. At the moment level, we prove that 
\begin{align*}
    &\EE\l\{\l|\frac{1}{\sqrt{n}}\sum_{i=1}^n X_i\r|^p\r\}\leq C_1^p p^{\frac{p}{2}}\|X\|_{\Psi_\alpha}^{p}
    +C_1^{\frac{p}{\alpha}}p^{\frac{p}{\alpha}}n^{\frac{p}{2}-\frac{p}{\alpha}}\|X\|_{\Psi_{\alpha}}^p\cdot\exp(-cn) \ \ \text{ when } \alpha\geq 1,\\
    &\EE\l\{\l|\frac{1}{\sqrt{n}}\sum_{i=1}^n X_i\r|^p\r\}\leq C^p\min\l\{p^{\frac{p}{2}},p^{\frac{p}{\alpha}}n^{\frac{p}{2}-\frac{p}{\alpha}}\r\}\cdot\|X\|_{\Psi_{\alpha}}^p \ \ \text{ when } \alpha\geq 2,
\end{align*}
which is unimprovable up to universal constants. These moment bounds improve the results in \cite{kuchibhotla2022moving,latala1997estimation} in two distinct regimes: 1) when $\alpha>2$ and 2) when $\alpha\in[1,2]$ with $1\ll n\lesssim p$. We defer the detailed discussion of concentration, moments, and $\Psi_\alpha$-norms for heterogeneous univariate summands to Section~\ref{sec:univariate}.

Many important distributions satisfy that $\sigma_X\ll\|X\|_{\Psi_\alpha}$, so separating the variance from the Orlicz norm is essential rather than superficial. A basic example is the Bernoulli random variable with success probability close to zero. To handle this regime, we introduce Definition~\ref{assm:sigma_L} (see Section \ref{sec:univariate_var}), a general moment framework that extends the sub-Gaussian characterizations in \cite{van2013bernstein} and \cite{alquier2013sparse}. This framework interacts naturally with the class of random variables having finite $\Psi_\alpha$-norm. In particular, given i.i.d. $X_1,\ldots,X_n$ with $\|X\|_{\Psi_\alpha}<\infty$, we show that there are infinitely many choices of $(\sigma,L)$ satisfying Definition~\ref{assm:sigma_L}; for suitable choices, we can obtain that for all $\alpha\geq 1$, 
\begin{align*}
    &\PP\l(\l|\sum_{i=1}^n X_i\r|\geq t\r)\leq 2\exp\l(-C_1\frac{t^2}{n\sigma_X^2}\r) \ \ \text{ for } t\leq cn\frac{\sigma_X^2}{\|X\|_{\Psi_\alpha}}\l(\log\l(\frac{2\|X\|_{\Psi_\alpha}}{\sigma_X}\r)\r)^{-\frac{1}{\alpha}},\\
    &\PP\l(\l|\sum_{i=1}^n X_i\r|\geq t\r)\leq 2\exp\l(-C_1\frac{t^\alpha}{n^{\alpha-1}\|X\|_{\Psi_\alpha}^\alpha}\r) \ \ \text{ for } t\geq n\|X\|_{\Psi_\alpha},
\end{align*}
and for $t$ between these two scales, the tail becomes an interpolation of $\Psi_1$- and $\Psi_2$-tails that connects the two endpoints. In particular, when $\alpha=1$, the second endpoint becomes $t\geq n\sigma_X$. We defer the details to Theorem~\ref{thm:conc_univariate_SigmaL} and Corollary~\ref{cor:sigmaL_inf} (see Section \ref{sec:univariate_var}). These results preserve the sharp variance scaling of \cite{koltchinskii2011neumann} for small deviations while recovering the missing sharp $\Psi_\alpha$-tail for large deviations. In this sense, our results sharpen the works of \cite{talagrand1994supremum}, \cite{ledoux2013probability}, \cite{kuchibhotla2022moving}, and complete the picture initiated by \cite{koltchinskii2011neumann}.

Our framework also extends beyond independent scalar sums. For martingales, we study the distribution of the limit $\lim_{n\to\infty}\sum_{k=1}^n a_kX_k$, where $\{X_k\}$ is a martingale. Our goal is different from that of \cite{rio2013extensions}, which controlled the limiting distribution under an assumed moment generating function bound. In contrast, we derive the relevant moment generating function behavior and convergence from moment conditions or from a finite conditional $\Psi_\alpha$-norm. For random vectors, we identify a different two-phase transition, now at $\alpha=4$, separating the regimes of $2\leq \alpha\leq 4$ and $\alpha\geq 4$. The resulting concentration behavior exhibits a nontrivial interplay among the decaying $\Psi_2$-, $\Psi_4$-, $\Psi_{\frac{\alpha}{2}}$-, and $\Psi_\alpha$-tails, together with a delicate interaction between the variance and $\Psi_\alpha$-norm. These results sharpen and extend Theorem~3.1.1 in the classical work of \cite{vershynin2018high} and the recent work of \cite{jeong2022sub}. The state-of-the-art \cite{jeong2022sub} proved a sharp bound for $X\in\mathbb{R}^d$ with i.i.d. components and $\operatorname{var}(X_i)=1$, $K:=\|X_i\|_{\Psi_2}\geq 1$, $\l\|\|X\|-\sqrt{d}\r\|_{\Psi_2}\leq CK\sqrt{\log K}$, which implies that 
 \begin{align*}
     \PP\l(\l|\|X\|-\sqrt{d} \r| \geq t\r)\leq 2\exp\l(-C\frac{t^2}{K^2\log K}\r).
 \end{align*}
We prove the \textit{sharper} tail probability
\begin{align*}
    \PP\l(\l|\|X\| -\sqrt{d}\sigma_X\r| \geq s \r)\leq \begin{cases}
        2\exp\l(-\frac{cs^2}{\|X\|_{\Psi_2}^2\log\l(\frac{2\|X\|_{\Psi_2}}{\sigma_X}\r)}\r)& \text{ for } 0 < s\leq \tau_1,\\
        2\exp\l(-\frac{cs^4}{d\sigma_X\|X\|_{\Psi_\alpha}^3}\r)& \text{ for } \tau_1\leq s\leq \tau_2,\\
        2\exp\l(-\frac{cs^2}{\|X\|_{\Psi_2}^2}\r)& \text{ for } s\geq \tau_2,
    \end{cases}
\end{align*}
where $\tau_1:=\sqrt{d}\sqrt{\frac{\sigma_X\|X\|_{\Psi_2}}{\log \l(\frac{2\|X\|_{\Psi_2}}{\sigma_X}\r)}}$ and $\tau_2:=\sqrt{d}\sqrt{\sigma_X\|X\|_{\Psi_2}}$. The same framework also applies to eigenvalue analysis for random matrices and covariance matrix estimation; see Section~\ref{sec:app}.

\subsection{Related works} \label{new.Sec.relalite}

The discussion above already identifies the major gap in the literature. We now position our results more systematically relative to the existing works. An early foundational work is \cite{talagrand1994supremum}, which studied concentrations of symmetric random variables with density $c_{\alpha}\exp(-|x|^\alpha)$ and obtained the moment generating function bound
\begin{equation}
    \begin{split}
        &\EE\exp(\lambda X)\leq \exp(C_\alpha \lambda^2 \EE\exp(|X|/C)) \ \ \text{ for all } \lambda\leq 1,\\
    &\EE\exp(\lambda X)\leq \exp(\lambda^\beta/\beta+\log(\EE\exp(|X|^\alpha/C\alpha) )) \ \ \text{ for all } \lambda>0.
    \end{split}
    \label{eq:talagrand:momG}
\end{equation}
Here, $\beta$ is the conjugate of $\alpha$, i.e., $\frac{1}{\alpha}+\frac{1}{\beta}=1$. For the symmetric density model considered in \cite{talagrand1994supremum}, the bound is sharp, but for general random variables---especially asymmetric ones and 
for sufficiently small $\lambda$---
\eqref{eq:talagrand:momG} is not sharp or sufficient. The 
generalizations of \cite{ledoux2013probability} and \cite{talagrand1989isoperimetry} led to 
\eqref{eq:talagrand:norm} and \eqref{eq:talagrand:norm2}. As discussed above, however, \eqref{eq:talagrand:norm} may be vacuous when $\alpha>2$, while \eqref{eq:talagrand:norm2} does not separate the variance from the Orlicz norm.

A second line of work includes \cite{boucheron2003concentration}, \cite{adamczak2011restricted}, \cite{kuchibhotla2022moving}, and \cite{zhang2022sharper}. These papers proved that for all $\alpha\in[1,\infty)$,
  \begin{align}
      \PP\l(\l|\sum_{i=1}^n X_i \r| \geq C\sqrt{ n}\|X\|_{\Psi_{\alpha}}\sqrt{t} + Cn^{\frac{1}{\beta}}\|X\|_{\Psi_\alpha}t^{\frac{1}{\alpha}}\r)\leq 2\exp\l(-t\r),
      \label{eq3:intro}
  \end{align}
which yields a sub-Gaussian component along with a $\Psi_\alpha$ component. However, for $\alpha\geq 2$ this is still a standard $\min$-type bound and therefore, misses the sharper $\max$-type behavior established in Theorem~\ref{thm:conc_univariates}. Moreover, \cite{latala1997estimation}, \cite{kuchibhotla2022moving}, and \cite{zhang2022sharper} established or employed the moment inequality
\begin{align*}
    \EE\l\{ \l|\frac{1}{\sqrt{n}}\sum_{i=1}^n X_i \r|^p\r\}\leq C_1^{\frac{p}{2}}p^{\frac{p}{2}}\|X\|_{\Psi_{\alpha}}^p+C_1^{\frac{p}{\alpha}}p^{\frac{p}{\alpha}}n^{\frac{p}{\beta}-\frac{p}{\alpha}}\|X\|_{\Psi_{\alpha}}^p \ \ \text{ for } \alpha\geq 1,
\end{align*}
which is governed by the maximum of $C_1^{\frac{p}{2}}p^{\frac{p}{2}}\|X\|_{\Psi_{\alpha}}$ and $C_1^{\frac{p}{\alpha}}p^{\frac{p}{\alpha}}n^{\frac{p}{\beta}-\frac{p}{\alpha}}\|X\|_{\Psi_{\alpha}}$. This line of work does not distinguish $\sigma_X$ from $\|X\|_{\Psi_\alpha}$. Consequently, 
\eqref{eq3:intro} is sharp for $\alpha\in[1,2]$ when $\sigma_X\asymp\|X\|_{\Psi_\alpha}$, but is not sharp for $\alpha>2$ and is not variance-sensitive when $\sigma_X\ll\|X\|_{\Psi_\alpha}$.

A third line of work, represented by \cite{koltchinskii2011neumann} and \cite{adamczak2008tail}, emphasizes the variance-controlled concentration around the origin. In addition to \cite{koltchinskii2011neumann}, \cite{adamczak2008tail} developed a related concentration inequality at $\alpha=1$ that is also 
adaptive to variance near the origin, but incurs an additional $\log(n)$ factor for sufficiently large deviations. The approaches in \cite{koltchinskii2011neumann} and \cite{adamczak2008tail} are complementary, and either can be sharper depending on the specific regime. Our contribution is to \textit{unify} sharp local variance behavior with the correct global $\Psi_\alpha$-tail in a single framework.
 
\begin{table}[t]
\centering
\begin{tabular}{ c|| c  }
\hline
\hline
 $\log(\frac{1}{2}\PP\l(|\sum_{k=1}^n X_k|\geq t\r))$ & $\alpha\geq 2$  \\ 
 \cite{ledoux2013probability} & $-\l(\frac{t}{K_\alpha\|\sum_k X_i\|_1+  \|( \|X_k\|_{\Psi_s}) \|_{\beta,\infty}}\r)^\alpha$ \\  
 \cite{boucheron2003concentration} & $-C_\alpha\min\l\{\frac{t^2}{n\|X\|_{\Psi_\alpha}^2},\frac{t^\alpha}{n^{\alpha-1}\|X\|_{\Psi_\alpha}^\alpha}\r\}+C_\alpha'$ \\
 \cite{koltchinskii2011neumann} & $-C\min\l\{\frac{t^2}{n\sigma_X^2},\frac{t}{\|X\|_{\Psi_\alpha}\log\l(\frac{2\|X\|_{\Psi_\alpha}}{\sigma_X}\r)}\r\}$ \\
 \cite{kuchibhotla2022moving}&$-C\min\l\{\frac{t^2}{n\|X\|_{\Psi_\alpha}^2},\frac{t^\alpha}{n^{\alpha-1}\|X\|_{\Psi_\alpha}^\alpha}\r\}$\\
 This work (Theorem~\ref{thm:conc_univariates}) &$-C\max\l\{\frac{t^2}{n\|X\|_{\Psi_\alpha}^2},\frac{t^\alpha}{n^{\alpha-1}\|X\|_{\Psi_\alpha}^\alpha}\r\}$ \\
 This work (distinguishing $\|X\|_{\Psi_\alpha}$ and $\sigma_X$)& Section~\ref{sec:univariate_var}\\
 \hline
 \hline
\end{tabular}
 \caption{Concentration inequalities when $\alpha\geq 2$. Here, $K_\alpha,C_\alpha,C_\alpha'$ represent positive constants that may depend on $\alpha$, following the notation in the original works, while positive constant $C$ does not depend on $\alpha,X,n$. }
\label{table1}
\end{table}

To appreciate these sharper bounds, Tables~\ref{table1} and~\ref{table2} summarize representative concentration inequalities for univariate random variables. The comparison makes the contribution of this paper transparent. When $\alpha>2$, the available bounds are either potentially vacuous or retain a $\min$-type tail; in contrast, Theorem~\ref{thm:conc_univariates} yields the sharp $\max$-type behavior in Table~\ref{table1}. When $\alpha\in[1,2]$, Theorem~\ref{thm:conc_univariates} matches the best known order when $\sigma_X\asymp\|X\|_{\Psi_\alpha}$, while the inequalities in Section~\ref{sec:univariate_var} sharpen the literature whenever $\sigma_X$ needs to be separated from $\|X\|_{\Psi_\alpha}$. 
Although our technical arguments are self-contained and do not rely directly on existing results, the cited works provide a rich source of elegant ideas that motivate our analysis. For concentration inequalities of bounded random variables, we refer interested readers to \cite{ahlswede2002strong}, \cite{recht2011simpler}, \cite{gross2010quantum}, and \cite{gross2011recovering}.

\begin{table}[t]
\centering
\begin{tabular}{ c|| c }
\hline
\hline
 $\log(\frac{1}{2}\PP\l(|\sum_{k=1}^n X_k|\geq t\r))$  &  $1\leq \alpha\leq 2$ \\ 
 \cite{ledoux2013probability} & $-\l(\frac{t}{K_\alpha\|\sum_k X_i\|_1+  n^{1-\frac{1}{\alpha}}\|X\|_{\Psi_\alpha} }\r)^{\alpha}$ \\  
 \cite{boucheron2003concentration} & $-C_\alpha\min\l\{\frac{t^2}{n\|X\|_{\Psi_\alpha}^2},\frac{t^\alpha}{n^{\alpha-1}\|X\|_{\Psi_\alpha}^\alpha}\r\}+C_\alpha'$ \\
 \cite{koltchinskii2011neumann} & $-C\min\l\{\frac{t^2}{n\sigma_X^2},\frac{t}{\|X\|_{\Psi_\alpha}\log\l(\frac{2\|X\|_{\Psi_\alpha}}{\sigma_X}\r)}\r\}$ \\
 \cite{kuchibhotla2022moving}&$-C\min\l\{\frac{t^2}{n\|X\|_{\Psi_\alpha}^2},\frac{t^\alpha}{n^{\alpha-1}\|X\|_{\Psi_\alpha}^\alpha}\r\}$\\
 This work (Theorem~\ref{thm:conc_univariates}) &$-C\min\l\{\frac{t^2}{n\|X\|_{\Psi_\alpha}^2},\frac{t^\alpha}{n^{\alpha-1}\|X\|_{\Psi_\alpha}^\alpha}\r\}$\\
 This work (distinguishing $\|X\|_{\Psi_\alpha}$ and $\sigma_X$)&Section~\ref{sec:univariate_var}\\
 \hline
 \hline
\end{tabular}
\caption{Concentration inequalities when $1\leq \alpha\leq 2$. Here, $K_\alpha,C_\alpha,C_\alpha'$ represent positive constants that may depend on $\alpha$, following the notation in the original works, while positive constant $C$ does not depend on $\alpha,X,n$. When $\sigma_X\asymp \|X\|_{\Psi_\alpha}$, our result aligns with \cite{kuchibhotla2022moving}; when $\sigma_X\ll\|X\|_{\Psi_\alpha}$, Section~\ref{sec:univariate_var} is sharper than the existing literature.}
\label{table2}
\end{table}

\subsection{Our contributions} \label{new.Sec.contriorga}

We summarize the major contributions in the same order as the paper. First, Lemmas~\ref{lem:momGeneratingFunct} and \ref{lem:momGeneratingFunct_SigmaL} establish new moment generating function bounds that improve and generalize \cite{talagrand1994supremum}. These lemmas are the analytical core of the paper and explain the phase transition at $\alpha=2$ through the regularity of the moment generating function. Building on them, Section~\ref{sec:univariate} proves refined concentration, moment, and Orlicz norm bounds for sums of independent random variables with $\|X\|_{\Psi_\alpha}<\infty$. For sufficiently small deviation $t$, the tail probability $\PP(|\sum X_i|\geq t)$ is sub-Gaussian, whereas for large enough $t$, it has the correct $\Psi_\alpha$-tail. In the regime of $\|X\|_{\Psi_\alpha}\asymp \sigma_X$, these bounds are optimal. To the best of our knowledge, this is the \textit{first} work to prove a sharp, density-free concentration inequality with a genuine $\Psi_\alpha$-tail for $\alpha\geq 2$ under the only assumption of $\|X\|_{\Psi_\alpha}<\infty$, thereby improving the state of the art in \cite{kuchibhotla2022moving}. Figure~\ref{fig:OrliczNorm} visualizes such improvement.

\begin{figure}[t]
    \centering
    \begin{subfigure}[b]{0.3\textwidth}
    \centering
        \includegraphics[width=\textwidth]{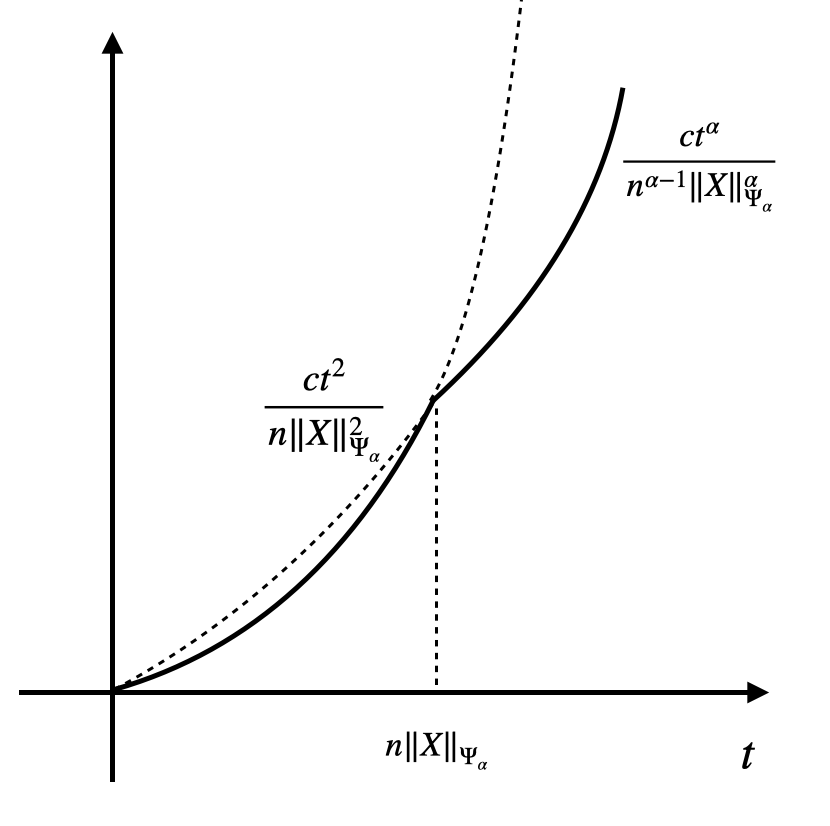}
        \caption{$\alpha\in[1,2]$}
        \label{fig:kuchi_alpha<2}
    \end{subfigure}
\hfill
    \begin{subfigure}[b]{0.3\textwidth}
    \centering
        \includegraphics[width=\textwidth]{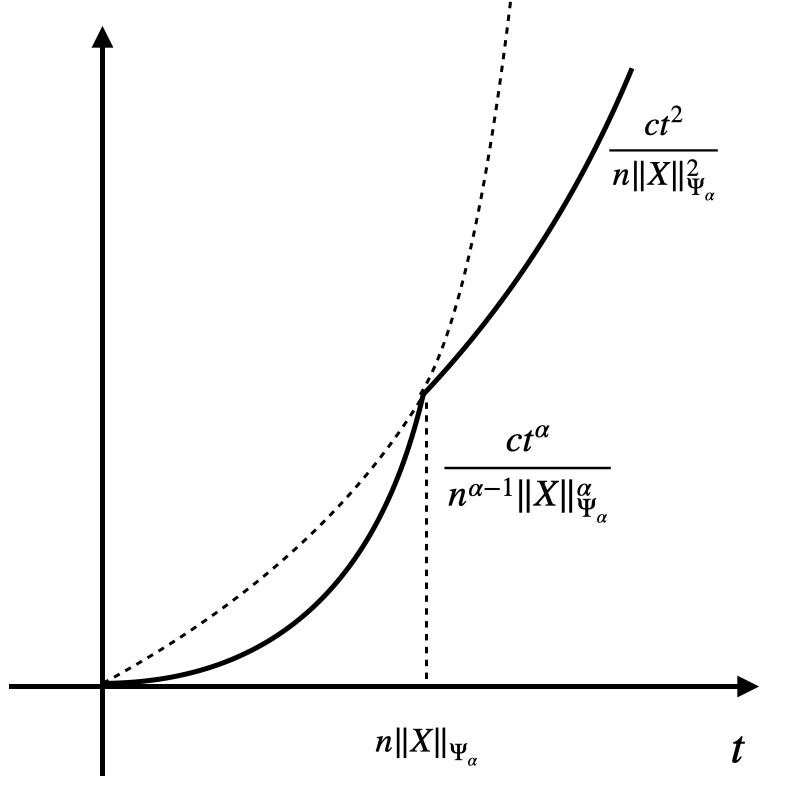}
        \caption{$\alpha\geq 2$}
        \label{fig:kuchi_alpha>2}
    \end{subfigure}
\hfill
    \begin{subfigure}[b]{0.33\textwidth}
    \centering
        \includegraphics[width=\textwidth]{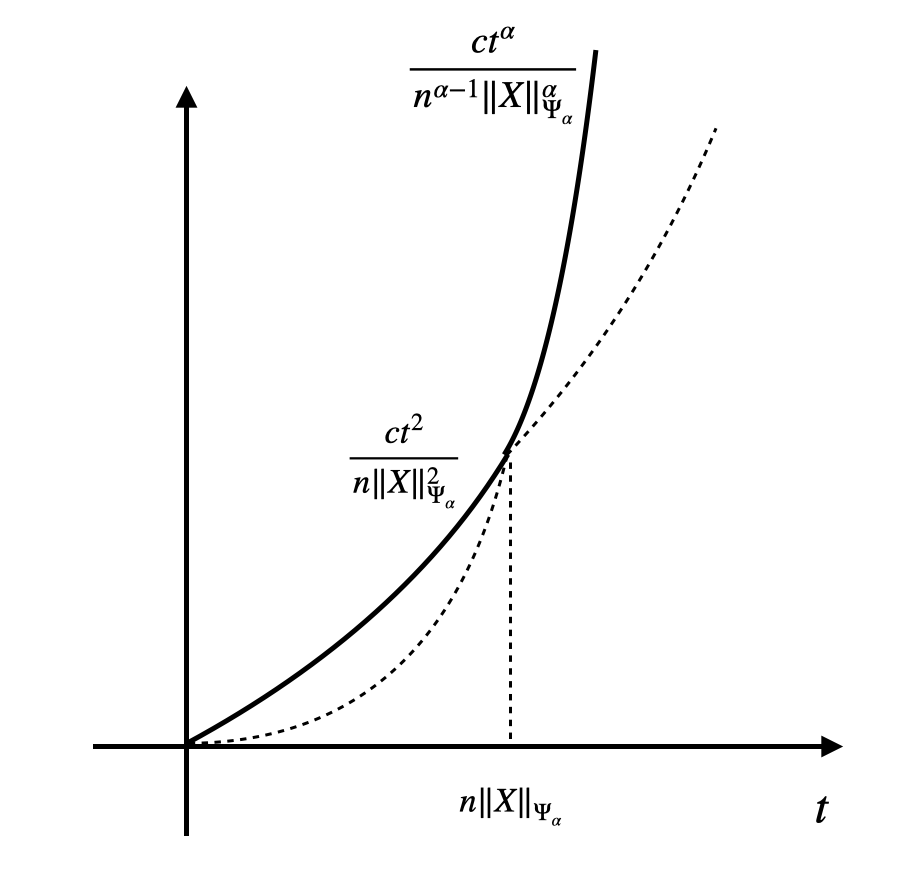}
        \caption{$\alpha\geq 2$}
        \label{fig:alpha>2}
    \end{subfigure}
    \caption{The $y$-axis represents the bound on $-\log(\frac{1}{2}\PP(|\sum_{i=1}^n X_i|\geq t))$. Figures~\ref{fig:kuchi_alpha<2} and \ref{fig:kuchi_alpha>2} are from \cite{kuchibhotla2022moving}. Figure~\ref{fig:alpha>2} corresponds to Theorem~\ref{thm:conc_univariates} or Theorem~\ref{thm:conc_univariate_SigmaL} when $\sigma_X$ does not need to be distinguished from $\|X\|_{\Psi_\alpha}$. The bound in Figure~\ref{fig:alpha>2} \textit{improves} that in Figure~\ref{fig:kuchi_alpha>2}. }
    \label{fig:OrliczNorm}
\end{figure}

Second, Section~\ref{sec:univariate_var} develops the variance-sensitive framework based on Definition~\ref{assm:sigma_L}. When $\alpha=2$, such framework reduces to conditions considered previously in \cite{van2013bernstein} and \cite{alquier2013sparse}; for general $\alpha$, it yields a \textit{new} interpolation between the variance and $\Psi_\alpha$-tails. Our main results, Theorem~\ref{thm:conc_univariate_SigmaL} and Corollary~\ref{cor:sigmaL_inf}, simultaneously produce sub-Gaussian tails depending only on variance for sufficiently small $t$, and rate-optimal $\Psi_\alpha$-tails for large $t$, even though the random variables themselves may not be sub-Gaussian. This combination \textit{cannot} be recovered by simply combining previous inequalities. In particular, for sub-exponential random variables, we prove that
\begin{multline*}
    \PP\l(\l|\sum_{i=1}^n X_i\r|\geq t\r) \leq 2\exp\l(-\max\l\{\min\l\{\frac{t^2}{n\sigma_X^2},\frac{t}{\|X\|_{\Psi_1}\log\l(\frac{2\|X\|_{\Psi_2}}{\sigma_X}\r)} \r\},\r.\r.\\\l.\l. \min\l\{\frac{t}{\|X\|_{\Psi_1}}, \frac{t^2}{n\sigma_X\|X\|_{\Psi_1}}\r\}\r\}\r),
\end{multline*}
which \textit{cannot} be obtained from the existing literature and improves, e.g., the Bernstein-type bounds such as those discussed in \cite{vershynin2018high}. More generally, for all  $\alpha\geq 1$, Corollary~\ref{cor:sigmaL_inf} yields sub-Gaussian tails of form
$$\exp\l(-ct^2/\sum_{i=1}^n \operatorname{var}(X_i)\r)$$
for sufficiently small $t$, while for large $t$, the tail probability becomes
$$\exp\l(-ct^\alpha/(\sum_{i=1}^n \|X_i\|_{\Psi_\alpha}^\beta)^{\frac{\alpha}{\beta}}\r),$$
which is rate-optimal. These results sharpen \cite{talagrand1989isoperimetry,talagrand1994supremum,ledoux2013probability,boucheron2003concentration,kuchibhotla2022moving} and complete \cite{koltchinskii2011neumann}. Figure~\ref{fig:OrliczNorm_sigmanorm} compares the corresponding tails.

\begin{figure}[t]
    \centering
    \begin{subfigure}[b]{0.3\textwidth}
    \centering
        \includegraphics[width=\textwidth]{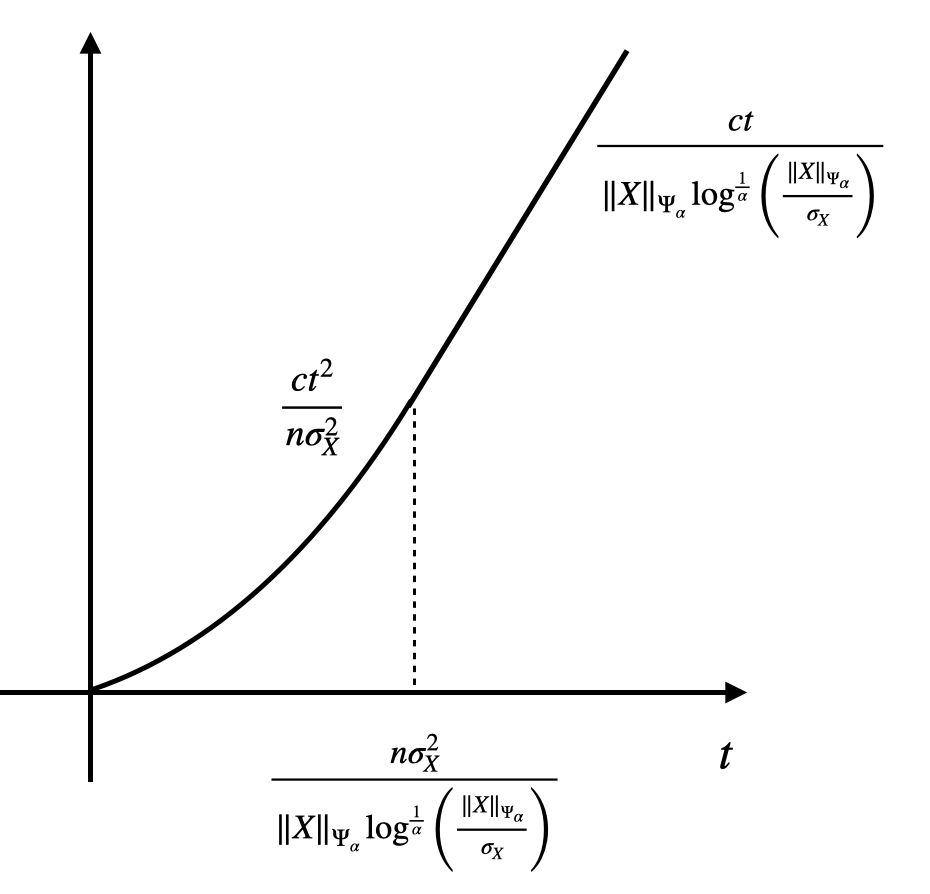}
        \caption{$\alpha\geq 1$}
        \label{fig:koltchinskii}
    \end{subfigure}
\hfill
    \begin{subfigure}[b]{0.3\textwidth}
    \centering
        \includegraphics[width=\textwidth]{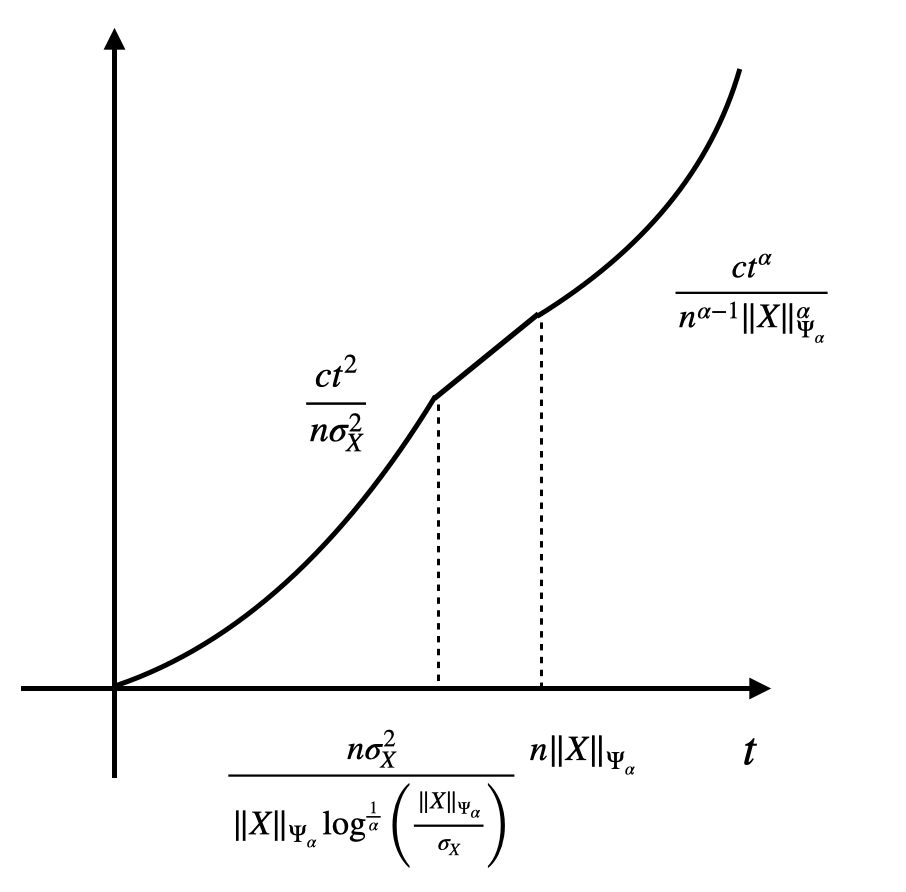}
        \caption{$\alpha\in[1,2]$}
        \label{fig:sigmaL_alpha<2}
    \end{subfigure}
\hfill
    \begin{subfigure}[b]{0.3\textwidth}
    \centering
        \includegraphics[width=\textwidth]{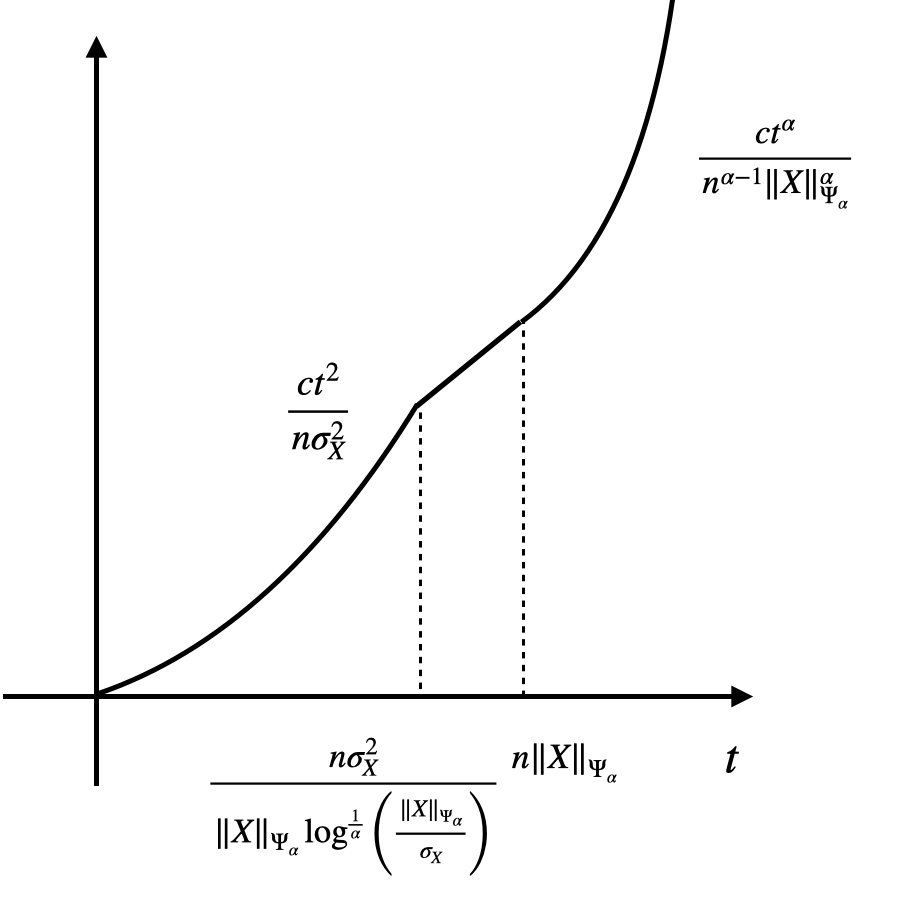}
        \caption{$\alpha\geq 2$}
        \label{fig:sigmaL_alpha>2}
    \end{subfigure}
    \caption{The $y$-axis represents the bound on $-\log(\frac{1}{2}\PP(|\sum_{i=1}^n X_i|\geq t))$. Figure~\ref{fig:koltchinskii} plots the bound given by \cite{koltchinskii2011neumann}, whose tail probability is sharp when $t\leq \frac{n\sigma_X^2}{\|X\|_{\Psi_\alpha}\log^{\frac{1}{\alpha}}\l(\frac{\|X\|_{\Psi_\alpha}}{\sigma_X}\r)}$. Figures~\ref{fig:sigmaL_alpha<2} and \ref{fig:sigmaL_alpha>2} correspond to Theorem~\ref{thm:conc_univariate_SigmaL} and Corollary~\ref{cor:norm_sigmaL}. The bounds in Figures~\ref{fig:sigmaL_alpha<2} and \ref{fig:sigmaL_alpha>2} \textit{improve} those in Figures~\ref{fig:kuchi_alpha<2}, \ref{fig:kuchi_alpha>2}, and \ref{fig:koltchinskii}. }
    \label{fig:OrliczNorm_sigmanorm}
\end{figure}

Third, Section~\ref{sec:app} extends the framework to martingales for dependent data, random vectors, random matrices, and covariance matrix estimation. These extensions show that our new theoretical framework is not limited to an isolated scalar setting: it is flexible enough to improve existing vector and matrix norm bounds, including those of \cite{vershynin2018high} and \cite{jeong2022sub}, even in some i.i.d. sub-Gaussian settings when $\sigma_X\ll\|X\|_{\Psi_\alpha}$. The resulting sharper tails are relevant to statistical and machine learning applications such as low-rank matrix recovery \citep{koltchinskii2011neumann}, adaptively collected data \citep{lin2025semiparametric,khamaru2025near}, and tensor learning \citep{zhang2018tensor,zhou2025deflated,abdalla2026dimension}. Section \ref{Sec.disc} discusses further implications and possible extensions of the theory. All proofs of the main results and additional technical details are included in the Supplementary Material.


\textit{Notation}. Throughout the paper, for any random variable $X$ and number $k\geq 0$, we define $\|X\|_k:=\l(\EE|X|^k\r)^{\frac{1}{k}}$, and $\operatorname{var}(X)$ denotes the variance of $X$. For a vector $X\in\RR^d$, denote by $\|X\|$ its Euclidean norm, and for a matrix $X$, denote by $\|X\|$ its operator norm. Universal constants are written as $C,C_1,C_2,c,c_1,\ldots$. In addition, $[a]$ denotes the largest integer not exceeding $a$, and for a nonnegative integer $k$, we define $0!=1$ and $k!=k\times (k-1)\times \cdots\times 1$.

\section{Concentration inequalities when Orlicz norm is proportional to standard deviation} 
\label{sec:univariate}

To illustrate our main ideas, we start with concentration inequalities for independent univariate sub-Weibull random variables, when it is \textit{not necessary} to emphasize the difference between the Orlicz norm $\|X\|_{\Psi_\alpha}$ and the standard deviation $\sigma_X:=\sqrt{\operatorname{var}(X)}$. The lemma below upper bounds the moment generating function for general sub-Weibull random variables with $\alpha> 1$.

\begin{lemma}[Moment generating function]
    Let $X$ be a mean-zero random variable with $\|X\|_{\Psi_{\alpha}}<\infty$ for some $\alpha>1$, and $\beta$ the conjugate of $\alpha$ with $\frac{1}{\alpha}+\frac{1}{\beta}=1$. Then there exist some positive constants $C,C_1,C_2,C_3,C_4,C_5$ that do not depend on $\alpha$, $\lambda$, and $X$ such that for any $\lambda\geq 0$,
    \begin{enumerate}[(1)]
        \item for $\alpha\geq 2$, it holds that 
    \begin{align*}
        \EE\l\{\exp(\lambda X)\r\}\leq \exp\l(C_1\min\l\{\lambda^2\|X\|_{\Psi_{\alpha}}^2,\lambda^{\beta}\|X\|_{\Psi_{\alpha}}^{\beta}\r\}\r);
    \end{align*}
    \item for $\alpha\in(1,2]$, it holds that 
    \begin{align*}
        &\EE\l\{\exp(\lambda X)\r\}\leq \exp(C_2\beta\lambda^2\|X\|_{\Psi_\alpha}^2) \ \ \text{ when } \lambda\leq 1/(C\|X\|_{\Psi_\alpha}),\\
        &\EE\l\{\exp(\lambda X)\r\}\leq \exp(C_3^\beta\beta \lambda^\beta \|X\|_{\Psi_\alpha}^{\beta}) \ \ \text{ when } \lambda\geq 1/(C\|X\|_{\Psi_\alpha});
    \end{align*}
    and 
    further it holds for any $\tau\in(0,1)$ that 
    \begin{align*}
        &\EE\l\{\exp(\lambda X)\r\}\leq \exp\l(C_4\frac{1}{1-\tau}\lambda^2\|X\|_{\Psi_\alpha}^2\r) \ \ \text{ when } \lambda\leq \tau/(C\|X\|_{\Psi_\alpha}),\\
        &\EE\l\{\exp(\lambda X)\r\}\leq \exp\l(C_5^\beta\frac{\tau^{-[\beta]-1}}{1-\tau} \lambda^\beta \|X\|_{\Psi_\alpha}^{\beta}\r) \ \ \text{ when } \lambda\geq \tau/(C\|X\|_{\Psi_\alpha}).
    \end{align*}
    \end{enumerate}
    \label{lem:momGeneratingFunct}
\end{lemma}

Lemma~\ref{lem:momGeneratingFunct} above shows that for $\alpha>1$, regardless of $\alpha\geq 2$ or $\alpha\in (1,2]$, when $\lambda$ is sufficiently small, the moment generating function can be upper bounded with $\exp(O(\lambda^2\|X\|_{\Psi_\alpha}^2))$; when $\lambda$ is sufficiently large, the bound for the moment generating function becomes $\exp(O(\lambda^\beta \|X\|_{\Psi_{\alpha}}^{\beta}))$. Lemma~\ref{lem:momGeneratingFunct} \textit{improves and completes} the moment generating function bound in \cite{talagrand1994supremum} or 
\eqref{eq:talagrand:momG}. The following lemma provides the lower bound on the moment generating function for $\lambda$ with sufficiently small values. 

\begin{lemma}[Lower bound on moment generating function]
    Let $X$ be a mean-zero random variable with $\|X\|_{\Psi_{\alpha}}<\infty$ for some $\alpha\geq 1$, and $\sigma_X^2:=\operatorname{var}(X)$. Then we have that for all $\lambda\leq \frac{1}{\|X\|_{\Psi_\alpha}}\l(\log\l(\frac{2\|X\|_{\Psi_\alpha}}{\sigma_X}\r)\r)^{-\frac{1}{\alpha}}$, 
    \begin{align*}
        \EE\l\{\exp(\lambda X)\r\}\geq \exp\l(\lambda^2\sigma_X^2/8\r).
    \end{align*}
    \label{lem:momGeneratingFunct_lower}
\end{lemma}

Consequently, combining Lemma~\ref{lem:momGeneratingFunct} and Lemma~\ref{lem:momGeneratingFunct_lower} verifies that for sufficiently small $\lambda$, it holds that 
$$\log(\EE(\lambda X))\asymp \lambda^2,$$ 
which is quadratic in $\lambda$. In fact, we will demonstrate in Example~\ref{exp:1} and Theorem~\ref{thm:conc_univariate_SigmaL} in Section \ref{sec:univariate_var} later that $\log(\EE(\lambda X))\asymp \lambda^2\sigma_X^2$ holds in this range. It implies that regardless of values of $\alpha$, for small enough $\lambda$, the moment generating function of $X$ behaves as if the random variable were sub-Gaussian with the Orlicz norm 
proportional to the standard deviation. An application of Lemma~\ref{lem:momGeneratingFunct} yields the following concentration inequalities.

\begin{theorem}[Concentration inequalities]
    Let $X_1,\ldots,X_n$ be independent mean-zero random variables with $\|X_i\|_{\Psi_{\alpha}}<\infty$, $a_1,\ldots,a_n$ any $n$ scalars, and $\beta$ satisfy $\frac{1}{\alpha}+\frac{1}{\beta}=1$. Then there exist some constants $C_1,C_2>0$ that do not depend on $\alpha$, $t$, and $X$ such that for $\alpha\geq 2$, 
    \begin{align*}
        \PP\l(\l|\sum_{i=1}^n a_iX_i \r| \geq t\r)\leq2\exp\l(-C_1\max\l\{\frac{t^2}{\sum_{i=1}^n a_i^2\|X_i\|_{\Psi_{\alpha}}^2},\frac{t^{\alpha}}{(\sum_{i=1}^n |a_i|^{\beta} \|X_i\|_{\Psi_{\alpha}}^{\beta})^{\frac{\alpha}{\beta}}}\r\}\r),
    \end{align*}
    and for $1\leq\alpha\leq2$, 
    \begin{align*}
        \PP\l(\l|\sum_{i=1}^n a_iX_i \r| \geq t\r)\leq 2\exp\l(-C_2\min\l\{\frac{t^2}{\sum_{i=1}^n a_i^2\|X_i\|_{\Psi_{\alpha}}^2},\frac{t^{\alpha}}{(\sum_{i=1}^n |a_i|^{\beta} \|X_i\|_{\Psi_{\alpha}}^{\beta})^{\frac{\alpha}{\beta}}}\r\}\r),
    \end{align*}
    for all $t\geq 0$. 
    \label{thm:conc_univariates}
\end{theorem}

Theorem~\ref{thm:conc_univariates} above establishes the concentration inequalities for sub-Weibull random variables. It unveils an interesting phase transition at $\alpha=2$. Regardless of $\alpha\in[2,\infty)$ or $\alpha\in[1,2]$, for sufficiently small $t$, the tail probability of $\sum_{i=1}^n a_iX_i$ behaves as if $\{X_i\}_{i=1}^n$ were sub-Gaussian.  When $t$ is large enough, the tail probability presented in Theorem~\ref{thm:conc_univariates} enjoys a $\Psi_\alpha$ decaying tail. Indeed, the triangle inequality with respect to the $\Psi_\alpha$-norm leads to 
$$\l\|\sum_{i=1}^n a_iX_i \r\|_{\Psi_{\alpha}}\leq \sum_{i=1}^n |a_i|\l\|X_i \r\|_{\Psi_{\alpha}},$$ 
which entails that $\sum_{i=1}^n a_iX_i $ has a finite $\Psi_\alpha$-norm. This is consistent with Theorem~\ref{thm:conc_univariates}, while Theorem~\ref{thm:conc_univariates} \textit{sharpens} the triangle inequality. In addition, when $\alpha=1$, we have $\beta=\infty$ and $(\sum_{i=1}^n |a_i|^{\beta} \|X_i\|_{\Psi_{1}}^{\beta})^{\frac{1}{\beta}}=\max_{i=1,\ldots,n}|a_i|\|X_i\|_{\Psi_1}$.
  
The following theorem provides the moment inequalities for $\sum_{i = 1}^n a_iX_i$.

\begin{theorem}[Moment inequalities]
     Let $X_1,\ldots,X_n$ be independent mean-zero random variables with $\|X_i\|_{\Psi_{\alpha}}<\infty$ for some $\alpha\geq 1$,  $a_1,\ldots,a_n$ any $n$ scalars, and $p\geq 1$. 
     Then we have that for all $\alpha\geq 1$, 
     \begin{multline*}
        \EE\l\{ \l|\sum_{i=1}^n a_iX_i \r|^p\r\}\leq  C_1^pp^{\frac{p}{2}}\l(\sum_{i=1}^n a_i^2\|X_i\|_{\Psi_\alpha}^2 \r)^{\frac{p}{2}}\\
        +C_1^{p}p^{\frac{p}{\alpha}}\l(\sum_{i=1}^n a_i^\beta \|X_i\|_{\Psi_{\alpha}}^\beta\r)^{\frac{p}{\beta}}\exp\l(-C\l(\frac{(\sum_{i=1}^{n} |a_i|^{\beta} \|X\|_{\Psi_{\alpha}}^{\beta})^{\frac{1}{\beta}}}{(\sum_{i=1}^n a_i^2\|X_{i}\|_{\Psi_{\alpha}}^2)^{\frac{1}{2}}}\r)^{\frac{2\alpha}{\alpha-2}}\r),
    \end{multline*}
    and additionally for all $\alpha\geq 2$, 
         \begin{align*}
             \EE\l\{ \l|\sum_{i=1}^n a_iX_i \r|^p\r\}\leq  C_2^p\min\l\{p^{\frac{p}{2}}\l(\sum_{i=1}^n a_i^2\|X_i\|_{\Psi_{\alpha}}^2 \r)^{\frac{p}{2}}, p^{\frac{p}{\alpha}}\l(\sum_{i=1}^n |a_i|^\beta \|X_i\|_{\Psi_{\alpha}}^\beta\r)^{\frac{p}{\beta}} \r\},
         \end{align*}
         where constants $C,C_1,C_2>0$ do not depend on $\alpha$, $p$, and $\{X_i\}_{i=1}^n$.
     \label{cor:moment_univariate}
\end{theorem}

Theorem~\ref{cor:moment_univariate} above provides upper bounds for the $p$th moment of $\sum_{i=1}^n a_iX_i$. Specifically, in the context of i.i.d. $X_i$'s with equal weights, Theorem~\ref{cor:moment_univariate} implies that for all $\alpha\geq 1$ and all $p\geq 1$, 
\begin{align*}
    \EE\l\{\l|\frac{1}{\sqrt{n}}\sum_{i=1}^n X_i\r|^p\r\}\leq C_1^p p^{\frac{p}{2}}\|X\|_{\Psi_\alpha}^{p}
    +C_1^{p}p^{\frac{p}{\alpha}}n^{\frac{p}{2}-\frac{p}{\alpha}}\|X\|_{\Psi_{\alpha}}^p\cdot\exp(-cn).
\end{align*}
Hence, for each fixed $p$, when $\sigma_X\asymp\|X\|_{\Psi_\alpha}$, the right-hand side of the expression above is dominated by $C_1^pp^{\frac{p}{2}}\|X\|_{\Psi_\alpha}^p\asymp C_1^pp^{\frac{p}{2}}\sigma_X^p$ as $n\to\infty$, which is consistent with the central limit theorem. Meanwhile, for $p$ varying with $n$ and the scenario of $\alpha\leq 2$, when $p\leq n$ the upper bound is dominated by $p^{\frac{p}{2}}$, whereas when $p\geq n$ the upper bound is dominated by $p^{\frac{p}{\alpha}}$. Moreover, for the scenario of $\alpha\geq 2$, Theorem~\ref{cor:moment_univariate} entails the following upper bound
\begin{align*}
    \EE\l\{ \l|\frac{1}{\sqrt{n}}\sum_{i=1}^n X_i \r|^p\r\}\leq C^p\min\l\{p^{\frac{p}{2}},p^{\frac{p}{\alpha}}n^{\frac{p}{2}-\frac{p}{\alpha}}\r\}\cdot\|X\|_{\Psi_{\alpha}}^p,
\end{align*}
where $X_i$'s are i.i.d. The above inequality shows that for $\alpha\geq 2$, when $p\leq n$ it is bounded by $C_1^\frac{p}{2}p^{\frac{p}{2}} $, while when $n\leq p$ it is bounded by $C_1^{\frac{p}{\alpha}}p^{\frac{p}{\alpha}}n^{\frac{p}{2}-\frac{p}{\alpha}} $.   

A combination of Theorem~\ref{thm:conc_univariates} and Theorem~\ref{cor:moment_univariate} further yields the following bound for the Orlicz norm.

\begin{corollary}[Bound on Orlicz norm]
    Let $X_1,\ldots,X_n$ be independent mean-zero random variables with $\|X_i\|_{\Psi_{\alpha}}<\infty$, and $a=(a_1,\ldots,a_n)$ any $n$-dimensional vector. Then we have that for $\alpha\geq 2$, 
    \begin{align*}
        \l\|\sum_{i=1}^n a_i X_i\r\|_{\Psi_{\alpha}}\leq C_1\l(\sum_{i=1}^n |a_i|^{\beta}\|X_i\|_{\Psi_{\alpha}}^{\beta}\r)^{\frac{1}{\beta}},\quad \l\|\sum_{i=1}^n a_i X_i\r\|_{\Psi_{2}}\leq C_1\l(\sum_{i=1}^n a_i^2\|X_i\|_{\Psi_{\alpha}}^{2}\r)^{\frac{1}{2}},
    \end{align*}
    and for $\alpha\in[1,2]$, 
    \begin{align*}
    \l\| \sum_{i=1}^n a_iX_i\r\|_{\Psi_\alpha}\leq C_2 \l(\sum_{i=1}^n a_i^{2}\|X_i\|_{\Psi_{\alpha}}^{2}\r)^{\frac{1}{2}},
\end{align*}
where $C_1,C_2>0$ are constants that do not depend on $\alpha$, $a$, and $X$.
    \label{cor:norm_univariate}
\end{corollary}

Corollary~\ref{cor:norm_univariate} above gives the upper bound for the $\Psi_\alpha$- and $\Psi_2$-norms of $\sum_{i=1}^n a_i X_i$, where the bound is expressed explicitly by coefficients $a_i$ and $\|X_i\|_{\Psi_\alpha}$. Similarly, a phase transition at $\alpha=2$ is observed. Generally, Corollary~\ref{cor:norm_univariate} is \textit{not} improvable, and when it is necessary to distinguish $\|X\|_{\Psi_\alpha}$ and $\sigma_X$, a \textit{sharper and more delicate} bound is given in Corollary~\ref{cor:norm_sigmaL} in Section~\ref{sec:univariate_var} later. Taken together, Theorem~\ref{thm:conc_univariates}, Theorem~\ref{cor:moment_univariate}, and Corollary~\ref{cor:norm_univariate} \textit{improve and complete} the univariate concentration inequalities in the works including \cite{ledoux2013probability}, \cite{boucheron2003concentration}, and \cite{kuchibhotla2022moving}, among others.

\begin{remark}[Characterization of concentration]
    A natural question is what the appropriate characterization of $\sum_{i=1}^n a_iX_i$ is. Indeed, we provide the tail probability in Theorem~\ref{thm:conc_univariates}, the bound of moments in Theorem~\ref{cor:moment_univariate}, and its $\Psi_\alpha$-norm in Corollary~\ref{cor:norm_univariate}. The tail probability in Theorem~\ref{thm:conc_univariates} presents a delicate \textit{interplay} between $\Psi_2$- and $\Psi_\alpha$-tails.  We emphasize that it is \textit{not} possible to obtain the delicate tail probability in Theorem~\ref{thm:conc_univariates} \textit{simply} based on the bounds of the $\Psi_\alpha$-norm or moments.
\end{remark}

\section{Concentration inequalities when Orlicz norm can exceed standard deviation greatly} 
\label{sec:univariate_var}

In this section, we further investigate the concentration inequalities of sub-Weibull random variables when the Orlicz norm may \textit{not} be proportional to the standard deviation, in which the tail probability bounds established in Section \ref{sec:univariate} earlier may \textit{no longer} be tight. 
Indeed, it holds that $\sigma_X:=\sqrt{\operatorname{var}(X)}\leq \sqrt 2\|X\|_{\Psi_\alpha}$ for any random variable $X$. However, these two quantities (i.e., the Orlicz norm and standard deviation) may \textit{not} have the same scale in 
general, e.g., for Bernoulli distributions; that is, the Orlicz norm can exceed the standard deviation greatly. This requires us to distinguish the standard deviation $\sigma_X$ from the Orlicz norm $\|X\|_{\Psi_\alpha}$. Specifically, we now focus on random variables that are characterized by moments, where the moments are determined jointly by two positive quantities $\sigma$ and $L$ as specified in the definition below.

\begin{definition}
    There exist two positive constants $\sigma$ and $L$ such that for all integers $k\geq 2$ and some $\alpha\geq 1$, it holds that 
    \begin{align*}
        \EE|X|^k\leq k^{\frac{k}{\alpha}}\sigma^2L^{k-2}.
    \end{align*}\label{assm:sigma_L}
\end{definition}

Definition~\ref{assm:sigma_L} above gives a delicate characterization of the distribution and has been prevalent in the literature; see, e.g., \cite{van2013bernstein} and \cite{alquier2013sparse} for the condition under $\alpha=2$. Indeed, Definition~\ref{assm:sigma_L} is related to the Orlicz $\Psi_\alpha$-norm; see the remark below.

\begin{remark}
    If a random variable $X$ satisfies Definition~\ref{assm:sigma_L} with some $(\sigma,L)$, its Orlicz $\Psi_\alpha$-norm can be bounded as $\|X\|_{\Psi_\alpha}\leq C\max\{\sigma,L\}$. On the other hand, if a random variable $X$ admits $\|X\|_{\Psi_\alpha}<\infty$, it satisfies Definition~\ref{assm:sigma_L} with the choice of $(\sigma,L) = (\|X\|_{\Psi_\alpha},C\|X\|_{\Psi_\alpha})$. 
    \label{rmk:def_SigmaL}
\end{remark}

We emphasize that $\sigma\ll L$ can indeed occur in many scenarios. As a concrete example, consider a Bernoulli random variable $X\sim\operatorname{Ber}(p)$. Its centered moments scale with its variance, although its Orlicz norm remains bounded by an absolute constant, i.e., $$\EE|X-\EE X|^k=p(1-p)\l\{p^{k-1}+(1-p)^{k-1}\r\}.$$
It shows that $\operatorname{Ber}(p)$ satisfies Definition~\ref{assm:sigma_L} with $\sigma=\sqrt{p(1-p)}$ and $L=1$. More generally, for a random variable $X$ with a finite $\Psi_\alpha$-norm, we provide two examples of admissible pairs $(\sigma,L)$ defined through $\sigma_X:=(\EE X^2)^{\frac{1}{2}}$ and $\|X\|_{\Psi_\alpha}$.

\begin{example}
    Assume that $X$ is a mean-zero random variable with $\|X\|_{\Psi_\alpha}<\infty$. Denote by $\sigma_X:=(\EE X^2)^{\frac{1}{2}}$. Then $X$ satisfies Definition~\ref{assm:sigma_L} with
    \begin{align*}
        \sigma:=\sigma_X,\quad L:=\|X\|_{\Psi_\alpha}\log^{\frac{1}{\alpha}}\l(\frac{2\|X\|_{\Psi_\alpha}}{\sigma_X}\r).
    \end{align*}
    This characterization is \textit{sharp} with the selection of $\sigma$ in light of Lemma~\ref{lem:momGeneratingFunct_lower} and the fact that $\sigma_X^2=\EE X^2\leq 2^{2/\alpha}\sigma^2$. The proof is nontrivial and presented in Section~\ref{sec:proof:univariate_var} of the Supplementary Material. Essentially, the proof exploits the truncation technique used in \cite{ahlswede2002strong}, \cite{recht2011simpler}, \cite{gross2010quantum}, \cite{gross2011recovering}, and \cite{koltchinskii2011neumann}.
    \label{exp:1}
\end{example}

\begin{example}
    Assume that $X$ is a mean-zero random variable with $\|X\|_{\Psi_\alpha}<\infty$. Denote by $\sigma_X:=(\EE X^2)^{\frac{1}{2}}$. Then $X$ satisfies Definition~\ref{assm:sigma_L} with
    \begin{align*}
        \sigma:=\sqrt{\sigma_X \|X\|_{\Psi_\alpha}},\quad L:=C\|X\|_{\Psi_\alpha}.
    \end{align*}
    This characterization is \textit{sharp} with the value of $L$. Indeed, the definition of the sub-Weilbull random variable guarantees that $\|X\|_k\leq Ck^{\frac{1}{\alpha}}\|X\|_{\Psi_\alpha}$, which entails that $L\geq C\|X\|_{\Psi_\alpha}$ for some constant $C>0$. The proof is also included in Section~\ref{sec:proof:univariate_var}.
    \label{exp:2}
\end{example}

The lemma below provides the bound on the moment generating function for random variables satisfying Definition~\ref{assm:sigma_L}, where a \textit{delicate interplay} between $\sigma$ and $L$ is observed. Due to the bounds given in Examples~\ref{exp:1} and \ref{exp:2} above, where generally each random variable $X$ with $\|X\|_{\Psi_\alpha}<\infty$ satisfies Definition~\ref{assm:sigma_L} for nontrivial $\sigma\leq L$, in what follows we assume that $\sigma\leq L$.

\begin{lemma}[Moment generating function]
Assume that random variable $X$ has mean zero and satisfies Definition~\ref{assm:sigma_L} with some $\alpha\geq 1$ and $(\sigma,L)$. Then we have the following bounds for the moment generating function of $X$, where $c,C_1,C_2,C_3>0$ are constants that do not depend on $\alpha,X,\sigma,L$.
    \begin{enumerate}
        \item When $\alpha\geq 2$, it holds that 
        \begin{align*}
            &\EE\exp(\lambda X)\leq \exp\l(C_1\lambda^2\sigma^2\r) \ \ \text{ for all } \lambda\leq\frac{c}{L},\\
            & \EE\exp(\lambda X)\leq \exp \l(C_2\min\l\{\lambda^\beta L^\beta, \lambda^2 L^2 \r\}\r) \ \ \text{ for all } \lambda\geq 0.
        \end{align*}
         \item When $1<\alpha\leq 2$, it holds that 
         \begin{align*}
            &\EE\exp(\lambda X)\leq \exp\l(C_1\lambda^2\sigma^2\r) \ \ \text{ for all } \lambda\leq\frac{c}{L},\\
            & \EE\exp(\lambda X)\leq \exp\l(C_3^\beta \frac{\tau^{-[\beta]-1}}{1-\tau}\lambda^\beta L^\beta\r) \ \ \text{ for all } \lambda\geq \frac{c\tau}{L},
        \end{align*}
        where $\tau>0$ is any constant in $(0,1)$.
        \item When $\alpha=1$, it holds that  $\EE\exp(\lambda X)\leq \exp\l(C_1\lambda^2\sigma^2\r) $ for all $\lambda\leq\frac{c}{L} $. 
    \end{enumerate}
    \label{lem:momGeneratingFunct_SigmaL}
\end{lemma}

Lemma~\ref{lem:momGeneratingFunct_SigmaL} above bounds the moment generating function of $X$ under different ranges of $\alpha$. When $\lambda$ is sufficiently small, the upper bound depends \text{only} on $\lambda^2\sigma^2$, which is interestingly \textit{independent of} $L$. Specifically, combining Example~\ref{exp:1}, Lemma~\ref{lem:momGeneratingFunct_lower}, and  Lemma~\ref{lem:momGeneratingFunct_SigmaL} together shows that for $\lambda\leq \frac{1}{\|X\|_{\Psi_\alpha}}\l(\log\l(\frac{2\|X\|_{\Psi_\alpha}}{\sigma_X}\r)\r)^{-\frac{1}{\alpha}}$, $$\log\l(\EE(\lambda X) \r)\asymp\lambda^2\sigma_X^2.$$ When $\lambda$ is large enough, the upper bound behaves \textit{differently}, which coincides with the fact that $X$ has a finite $\Psi_\alpha$-norm.

The following concentration inequality follows from Lemma~\ref{lem:momGeneratingFunct_SigmaL}.
\begin{theorem}
    Assume that $X_1,\ldots,X_n$ are independent with mean zero, and $X_i$ satisfies Definition~\ref{assm:sigma_L} with $(\sigma_i,L_i)$. Then we have that for $\alpha\geq 2$, 
    \begin{multline*}
        \PP\l(\l|\sum_{i=1}^n a_iX_i \r|\geq t\r)
        \leq 2\exp\l(-C\max\l\{\frac{t^\alpha}{\l(\sum_{i=1}^n |a_i|^\beta L_i^\beta\r)^{\frac{\alpha}{\beta}} } , \frac{t^2}{\sum_{i=1}^n a_i^2 L_i^2 },\r.\r.\\\l.\l. \min\l\{ \frac{t^2}{\sum_{i=1}^n a_i^2\sigma_i^2}, \frac{t}{\max_{i=1,\ldots,n}|a_i|L_i}\r\} \r\}\r),
    \end{multline*}
    and for $\alpha\in[1,2]$, 
    \begin{multline*}
        \PP\l(\l|\sum_{i=1}^n a_iX_i\r|\geq t\r)
        \leq 2\exp\l(-C\max\l\{\min\l\{\frac{t^2}{\sum_{i=1}^n a_i^2\sigma_i^2},\frac{t}{\max_{i=1,\ldots,n}|a_i|L_i}\r\},\r.\r.\\\l.\l.\min\l\{\frac{t^2}{\sum_{i=1}^n a_i^2\sigma_i^2},\frac{t^\alpha}{\l(\sum_{i=1}^n |a_i|^{\beta}L_i^{\beta}\r)^{\frac{\alpha}{\beta}}}\r\}\r\}\r).
    \end{multline*}
    \label{thm:conc_univariate_SigmaL}
\end{theorem}

Theorem~\ref{thm:conc_univariate_SigmaL} above establishes the concentration inequality under Definition~\ref{assm:sigma_L}. Interestingly, the tail probability presents a \textit{mixing} of Orlicz $\Psi_2$-, $\Psi_1$-, and $\Psi_\alpha$-tails. For sufficiently small $t$, the tail probability is the exponential of $-C\frac{t^2}{\sum_{i=1}^n a_i^2\sigma_i^2}$, which depends only on $\{\sigma_i\}$. For $t$ with intermediate values, it admits a sub-exponential tail. For large enough $t$, the tail presents a $\Psi_\alpha$ decay. The corollary below proves the $\Psi_\alpha$-norm and moment bounds under the framework of Definition~\ref{assm:sigma_L}.
\begin{corollary}[Bounds on $\Psi_\alpha$-norm and moments]
    Assume that $X_1,\ldots,X_n$ are independent with mean zero, and $X_i$ satisfies Lemma~\ref{lem:momGeneratingFunct} with $(\sigma_i,L_i)$. Then we have that for $\alpha\in[1,2]$, 
    \begin{align*}
        \l\|\sum_{i=1}^n a_iX_i \r\|_{\Psi_\alpha}&\leq C\l(\sum_{i=1}^n a_i^2\sigma_{i}^2\r)^{\frac{1}{2}}+C\l(\sum_{i=1}^n |a_i|^\beta L_i^{\beta}\r)^ {\frac{1}{\beta}},
    \end{align*}
    and for $\alpha\geq 2$, 
    \begin{align*}
        \l\|\sum_{i=1}^n a_i X_i\r\|_{\Psi_{\alpha}}\leq C_1\l(\sum_{i=1}^n |a_i|^{\beta}L_i^{\beta}\r)^{\frac{1}{\beta}},\quad \l\|\sum_{i=1}^n a_i X_i\r\|_{\Psi_{2}}\leq C_1\l(\sum_{i=1}^n a_i^2L_i^{2}\r)^{\frac{1}{2}}.
    \end{align*}
    Moreover, for $\alpha\geq 1$, it holds for the $p$th moment with $p\geq 1$ that 
    \begin{multline*}
        \EE \l|\sum_{i=1}^n a_iX_i \r|^p\leq  C^p p^{\frac{p}{2}} \l(\sum_{i=1}^n a_i^2\sigma_i^2\r)^{\frac{p}{2}}\\
        +C^p p^{\frac{p}{\alpha}}\l(\sum_{i=1}^n |a_i|^\beta L_i^\beta\r)^{\frac{p}{\beta}}\exp\l(-C\l(\frac{(\sum_{i=1}^n |a_i|^\beta L_i^\beta )^{\frac{1}{\beta}} }{( \sum_{i=1}^n a_i^2\sigma_i^2)^{\frac{1}{2}}}\r)^{\frac{2\alpha}{\alpha-2}}\r).
    \end{multline*}
    \label{cor:norm_sigmaL}
\end{corollary}

Indeed, when $\alpha\in[1,2]$, we have $\beta\geq 2$ and $\l(\sum_{i=1}^n |a_i|^\beta L_i^{\beta}\r)^ {\frac{1}{\beta}}\leq \l(\sum_{i=1}^n |a_i|^2L_i^{2}\r)^ {\frac{1}{2}}$, which show that Corollary~\ref{cor:norm_sigmaL} above \textit{improves} over Theorem~\ref{cor:moment_univariate} and Corollary~\ref{cor:norm_univariate} in Section \ref{sec:univariate} earlier when the Orlicz norm exceeds the standard deviation greatly. We remark that $(\sigma_i,L_i)$ can be expressed as functions of $\sigma_X$ and $\|X\|_{\Psi_\alpha}$, and the forms are generally \textit{not} unique, as illustrated in the following three examples.

\begin{example}[Connections between Theorems~\ref{thm:conc_univariates} and \ref{thm:conc_univariate_SigmaL}]
    Theorem~\ref{thm:conc_univariate_SigmaL} is consistent with Theorem~\ref{thm:conc_univariates}. Indeed, substituting the choice of $(\sigma_i,L_i)=(\|X\|_{\Psi_\alpha},\|X\|_{\Psi_\alpha})$ (presented in Remark~\ref{rmk:def_SigmaL}) into Theorem~\ref{thm:conc_univariate_SigmaL} yields Theorem~\ref{thm:conc_univariates}, where we employ
\begin{align*}
    \l(\sum_{i=1}^n a_i^2\|X_i\|_{\Psi_\alpha}^2\r)^{\frac{1}{2}}\geq \l(\sum_{i=1}^n |a_i|^\gamma\|X_i\|_{\Psi_\alpha}^\gamma\r)^{\frac{1}{\gamma}}\geq \max_{i=1,\ldots,n}|a_i|\|X_{i}\|_{\Psi_\alpha} \ \  \text{ for all } \gamma\geq 2.
\end{align*}
Thus, Theorem~\ref{thm:conc_univariate_SigmaL} generalizes 
Theorem~\ref{thm:conc_univariates}.
\end{example}

\begin{example}
    Let us continue with Example~\ref{exp:1}. Substituting the choice of $(\sigma_i,L_i)$ given by 
$$\sigma_{i}:=\sigma_{X_i},\quad L_i:=\|X_i\|_{\Psi_\alpha}\log^{\frac{1}{\alpha}}\l(\frac{2\|X_i\|_{\Psi_\alpha}}{\sigma_{X_i}}\r)$$ 
into Theorem~\ref{thm:conc_univariate_SigmaL} leads to the following concentration inequalities. For ease of presentation, let us assume that $X_i$'s are i.i.d. Then we have that for $\alpha\geq 2$, 
\begin{multline}
      \PP\l(\l|\sum_{i=1}^n X_i \r|\geq t\r)
    \leq 2\exp\l(-C\max\l\{\frac{t^\alpha}{n^{\alpha-1}  \|X\|_{\Psi_\alpha}^\alpha \log\l(\frac{2\|X\|_{\Psi_\alpha}}{\sigma_{X}}\r) } ,\r.\r.\\\l.\l.\min\l\{ \frac{t^2}{n\sigma_{X}^2}, \frac{t}{\|X\|_{\Psi_\alpha}\log^{\frac{1}{\alpha}}\l(\frac{2\|X\|_{\Psi_\alpha}}{\sigma_{X}}\r) }\r\} \r\}\r),
        \label{eq1:univariate_var}
    \end{multline}
    and for $1\leq\alpha\leq2$, 
\begin{multline}
    \PP\l(\l|\sum_{i=1}^n X_i \r|\geq t\r)
        \leq 2\exp\l(-C\min\l\{\frac{t^2}{n\sigma_{X}^2} ,\r.\r.\\\l.\l.\max\l\{ \frac{t^\alpha}{n^{\alpha-1} \|X\|_{\Psi_\alpha}^\alpha \log\l(\frac{2\|X\|_{\Psi_\alpha}}{\sigma_{X}}\r)}, \frac{t}{\|X\|_{\Psi_\alpha}\log^{\frac{1}{\alpha}}\l(\frac{2\|X\|_{\Psi_\alpha}}{\sigma_{X}}\r) }\r\} \r\}\r).
        \label{eq2:univariate_var}
    \end{multline}
    By \eqref{eq1:univariate_var} and \eqref{eq2:univariate_var}, the tail probability is $\exp(-Ct^2/\sum_{i=1}^n a_i^2\sigma_{X_i}^2)$ for $t\leq \frac{n\sigma_{X}^2}{\|X\|_{\Psi_\alpha}\log^{\frac{1}{\alpha}}\l(\frac{2\|X\|_{\Psi_\alpha}}{\sigma_{X}}\r) }$. The tail probability bound for this range is \textit{not} improvable, which arises from the \textit{sharp} value of $\sigma_i$ and has also been proved in the foundational work of \cite{koltchinskii2011neumann}. On the other hand, an additional $\log(\|X\|_{\Psi_\alpha}/\sigma_X)$ appears for sufficiently large $t$ compared to Theorem~\ref{thm:conc_univariates}.  
    \label{exp:4}
\end{example}

\begin{example}
     Let us continue with Example~\ref{exp:2}. Assume that $X_i$'s are i.i.d. for simplicity. Then by the choice of $(\sigma_i,L_i)$ in Example~\ref{exp:2}, we have that for $\alpha\geq2$,
    \begin{align}
        \PP\l(\l|\sum_{i=1}^n X_i \r|\geq t\r)
        \leq 2\exp\l(-C\max\l\{\frac{t^\alpha}{n^{\alpha-1}  \|X\|_{\Psi_\alpha}^\alpha  } ,\min\l\{ \frac{t^2}{n\sigma_{X}\|X\|_{\Psi_\alpha}}, \frac{t}{\|X\|_{\Psi_\alpha} }\r\} \r\}\r),
        \label{eq3:univariate_var}
    \end{align}
    and for $1\leq\alpha\leq2$, 
    \begin{align}
        \PP\l(\l|\sum_{i=1}^n X_i \r|\geq t\r)
        \leq 2\exp\l(-C\min\l\{\frac{t^2}{n \sigma_{X}\|X\|_{\Psi_\alpha}} ,\max\l\{ \frac{t^\alpha}{n^{\alpha-1} \|X\|_{\Psi_\alpha}^\alpha }, \frac{t}{\|X\|_{\Psi_\alpha} }\r\} \r\}\r).
        \label{eq4:univariate_var}
    \end{align}
    The tail probability in 
    \eqref{eq3:univariate_var} and \eqref{eq4:univariate_var} above is \textit{sharp} when $t$ is large enough, which coincides with Theorem~\ref{thm:conc_univariates}. Indeed, Example~\ref{exp:2} gives the \textit{sharp} value for $L=C\|X\|_{\Psi_\alpha}$, which yields the \textit{sharp} tail probability when $t\geq n\|X\|_{\Psi_\alpha}$; for this range, it improves over Example~\ref{exp:4}. However, for sufficiently small $t$, the tail probability is $\exp\l(-\frac{Ct^2}{\sum_{i=1}^n a_i^2\sigma_{X_i}\|X_i\|_{\Psi_\alpha}}\r) $, which is \textit{weaker} than the corresponding 
    one in Example~\ref{exp:4}. 
\end{example}

     In order to obtain the \textit{sharpest} tail probability represented as a function of variance and $\Psi_\alpha$-norm, one strategy is to consider the minimum tail probability over all admissible pairs of $(\sigma,L)$. To this end, we have the following corollary.
     
\begin{corollary}
    Assume that $X_1,\ldots,X_n$ are independent with mean zero, and $X_i$ satisfies  $\|X_i\|_{\Psi_\alpha}<\infty$. Denote by $\mathcal{D}_i:=\{(\sigma_i,L_i):\, \EE|X_i|^k\leq k^{\frac{k}{\alpha}}\sigma_i^2L_i^2 \ \text{ for all } k\geq 2\}$. Then we have that for $\alpha\geq 2$, 
    \begin{multline*}
        \PP\l(\l|\sum_{i=1}^n a_iX_i \r|\geq t\r)
        \leq 2\inf_{(\sigma_i,L_i)\in\mathcal{D}_i} \exp\l(-C\max\l\{\frac{t^\alpha}{\l(\sum_{i=1}^n |a_i|^\beta L_i^\beta\r)^{\frac{\alpha}{\beta}} } , \frac{t^2}{\sum_{i=1}^n a_i^2 L_i^2 },\r.\r.\\\l.\l. \min\l\{ \frac{t^2}{\sum_{i=1}^n a_i^2\sigma_i^2}, \frac{t}{\max_{i=1,\ldots,n}|a_i|L_i}\r\} \r\}\r),
    \end{multline*}
    and for $\alpha\in[1,2]$, 
    \begin{multline*}
        \PP\l(\l|\sum_{i=1}^n a_iX_i\r|\geq t\r)
        \leq 2\inf_{(\sigma_i,L_i)\in\mathcal{D}_i}\exp\l(-C\min\l\{\frac{t^2}{\displaystyle\sum_{i=1}^n a_i^2\sigma_i^2},\r.\r.\\\l.\l.\max\l\{\frac{t}{\displaystyle \max_{i=1,\ldots,n} |a_i|L_i},\frac{t^\alpha}{\displaystyle\l(\sum_{i=1}^n |a_i|^{\beta}L_i^{\beta}\r)^{\frac{\alpha}{\beta}}}\r\}\r\}\r).
    \end{multline*}
    \label{cor:sigmaL_inf}
\end{corollary}

An application of Corollary~\ref{cor:sigmaL_inf} above yields the following concentrations of sub-Gaussian and sub-exponential random variables that are \textit{new} to the literature.

\begin{example}[sub-exponential random variables]
    For i.i.d. sub-exponential $X_1,\ldots,X_n$, combining Corollary~\ref{cor:sigmaL_inf} and 
    \eqref{eq1:univariate_var}--\eqref{eq4:univariate_var} gives that 
\begin{align*}
    &\PP\l(\l|\sum_{i=1}^n X_i\r|\geq t\r)\leq 2\exp\l(-\frac{t^2}{n\sigma_X^2}\r) \ \ \text{ for all } t\leq\frac{n\sigma_X^2}{\|X\|_{\Psi_1}\log\l(\frac{2\|X\|_{\Psi_1}}{\sigma_X}\r)},\\
    &\PP\l(\l|\sum_{i=1}^n X_i\r|\geq t\r)\leq 2\exp\l(-\frac{t}{\|X\|_{\Psi_1}}\r) \ \ \text{ for all } t\geq n\sigma_X.
\end{align*}
For $t$ between the above two end points, the tail probability involves a function with a \textit{mixing} of $\exp(-t/( \|X\|_{\Psi_{1}}\log(\|X\|_{\Psi_1}/\sigma_X)))$ and $\exp(-ct^2/(n\sigma_X\|X\|_{\Psi_2}))$. More specifically, we have that 
\begin{multline*}
    \PP\l(\l|\sum_{i=1}^n X_i\r|\geq t\r) \leq 2\exp\l(-C\max\l\{\min\l\{\frac{t^2}{n\sigma_X^2},\frac{t}{\|X\|_{\Psi_1}\log\l(\frac{2\|X\|_{\Psi_2}}{\sigma_X}\r)} \r\},\r.\r.\\ \l.\l. \min\l\{\frac{t}{\|X\|_{\Psi_1}}, \frac{t^2}{n\sigma_X\|X\|_{\Psi_1}}\r\}\r\}\r).
\end{multline*}
Notably, the above concentration inequality \textit{cannot} be obtained by combining the works of \cite{koltchinskii2011neumann} and \cite{vershynin2018high}.
\end{example}

\begin{example}[sub-Gaussian random variables]
    For i.i.d. sub-Gaussian $X_1,\ldots,X_n$, an application of Corollary~\ref{cor:sigmaL_inf} and 
    \eqref{eq1:univariate_var}--\eqref{eq4:univariate_var} yields that 
\begin{align*}
    &\PP\l(\l|\sum_{i=1}^n X_i\r|\geq t\r)\leq 2\exp\l(-\frac{t^2}{n\sigma_X^2}\r) \ \ \text{ for all } t\leq\frac{n\sigma_X^2}{\|X\|_{\Psi_2}\log^{\frac{1}{2}}\l(\frac{2\|X\|_{\Psi_2}}{\sigma_X}\r)},\\
    &\PP\l(\l|\sum_{i=1}^n X_i\r|\geq t\r)\leq 2\exp\l(-\frac{t^2}{n\|X\|_{\Psi_2}^2}\r) \ \ \text{ for all } t\geq n\|X\|_{\Psi_2}.
\end{align*}
For $t$ between the above two end points, the tail probability involves a function with a \textit{mixing} of $\exp(-t/\|X\|_{\Psi_2})$, $ \exp(-t/( \|X\|_{\Psi_{2}}\log(2\|X\|_{\Psi_2}/\sigma_X)))$, and $\exp(-ct^2/(n\sigma_X\|X\|_{\Psi_2}))$. Specifically, we have that 
\begin{multline*}
    \PP\l(\l|\sum_{i=1}^n X_i\r|\geq t\r)\leq2\exp\l(-C\max\l\{\frac{t^2}{n\|X\|_{\Psi_2}^2},\min\l\{\frac{t^2}{n\sigma_X^2},\frac{t}{\|X\|_{\Psi_2}\log\l(\frac{2\|X\|_{\Psi_2}}{\sigma_X}\r)}\r\},\r.\r. \\ \l.\l. \min\l\{\frac{t^2}{n\sigma_X\|X\|_{\Psi_2}},\frac{t}{\|X\|_{\Psi_2}}\r\}\r\} \r).
\end{multline*}
Again, the above concentration inequality \textit{cannot} be derived by combining the previous works.
\end{example}

\section{Applications} \label{sec:app}

In this section, we present four applications of our new concentration theory. Section~\ref{sec:dependent} explores the convergence of martingales, where we \textit{no longer} assume that $X_i$'s are independent, and Section~\ref{sec:vector} considers the norm of a random vector. Section~\ref{new.Sec.randmat} studies the operator norm of a random matrix, and Section~\ref{new.Sec.covmatest} focuses on covariance matrix estimation.

\subsection{Martingales} \label{sec:dependent}

Here, we study the concentration behavior for \textit{dependent} data. Let $\{\mathcal{F}_n\}_{n\in\mathbb{N}}$ be an increasing filtration and $\{X_n\}_{n\in\mathbb{N}}$ a sequence of real-valued random variables adapted to $\{\mathcal{F}_n\}_{n\in\mathbb{N}}$. Specifically, we will focus on dependent random variables that satisfy the conditional moment condition below.

\begin{assumption}
    There exist two sequences of constants $\{\sigma_n\}$ and $\{L_n\}$ such that for all integers $k\geq 2$ and some $\alpha\geq 1$, it holds that 
    \begin{align*}
        \EE\l\{|X_n|^k|\mathcal{F}_{n-1} \r\}\leq k^{\frac{k}{\alpha}}\sigma_n^2L_n^{k-2} \quad \text{a.s.}
    \end{align*}
    \label{assm:dependent}
\end{assumption}

 Assumption~\ref{assm:dependent} above generalizes the independent-scenario condition presented in Definition~\ref{assm:sigma_L} earlier. In particular, when $\alpha=2$, Assumption~\ref{assm:dependent} reduces to the characterization used in \cite{alquier2013sparse}. Additionally, Assumption~\ref{assm:dependent} is related to the conditional Orlicz norm \citep{shen2026sgd}; see the following remark.

 \begin{remark}[Conditional $\Psi_\alpha$-norm]
     We inherit the notation and definition of the conditional Orlicz norm $\|\cdot|\mathcal{F}_n\|_{\Psi_\alpha}$ from Appendix A of \cite{shen2026sgd}. If $\|X_n|\mathcal{F}_{n-1}\|_{\Psi_\alpha}\leq K_n$ for some constant $K_n$, $X_n$ satisfies Assumption~\ref{assm:dependent} with $$\sigma_n=K_n,\quad L_n=CK_n.$$ On the other hand, if $X_n$ satisfies Assumption~\ref{assm:dependent} with $(\sigma_n,L_n)$, it holds that  $\|X_n|\mathcal{F}_n\|_{\Psi_\alpha}\leq C\max\{\sigma_n,L_n\}$. However, in some scenarios, it is necessary to express $\sigma_n$ as a function of the conditional variance and $\Psi_\alpha$-norm, and distinguish the variance from the squared norm. Assume that $X_n$ satisfies $\EE\{X_n|\mathcal{F}_{n-1}\}=0$ and denote by $\sigma_{X_n}^2:=\EE\{X_n^2|\mathcal{F}_{n-1}\}$. Then in parallel to Examples~\ref{exp:1} and \ref{exp:2}, we have two pairs of admissible quantities for $(\sigma_n,L_n)$ given by 
     \begin{align*}
         \sigma_n^2:=\sigma_{X_n}^2,\quad L_n:=K_n\log^{\frac{1}{\alpha}}\l(\frac{2K_n}{\sigma_{X_n}}\r)
     \end{align*}
     and
     \begin{align*}
         \sigma_n^2:=\sigma_{X_n}K_n,\quad L_n:=CK_n.
     \end{align*}
     The above two choices of $(\sigma_n,L_n)$ remain valid for all martingale differences that admit finite conditional $\Psi_\alpha$-norms.
 \end{remark}
 
 Based on the above remark, without loss of generality, let us assume that $\sigma_n\leq L_n$. In particular, we are interested in the limit of the following martingale
 \begin{align*}
     M_n:=\sum_{k=1}^n a_k \bigg\{X_k-\EE\bigg\{X_k|\mathcal{F}_k\bigg\}\bigg\}.
 \end{align*}
Specifically, we will study the convergence of $M_n$ under the conditions
  \begin{align}
      \sum_{k=1}^{\infty}a_k^2\sigma_k^2<\infty,\quad \sum_{k=1}^{\infty} |a_k|^\beta L_k^{\beta}<\infty,
      \label{eq1:dependent}
  \end{align}
  which are necessary in general. Indeed, by Rosenthal's inequality and the martingale convergence theorem, $M_n$ converges almost surely and in $L^\beta$ to $M_\infty$ given by 
\begin{align*}
    M_{\infty}:=\sum_{k=1}^{\infty} a_k \bigg\{X_k-\EE\bigg\{X_k|\mathcal{F}_k\bigg\}\bigg\}.
\end{align*} 
On the other hand, \cite{rio2013extensions} studied the convergence from a \textit{different} perspective. Specifically, \cite{rio2013extensions} made direct assumptions on the moment generating function; in contrast, we assume \textit{only} the finite conditional $\Psi_\alpha$-norm or conditional moments. We emphasize that for any $\gamma>\beta$, it holds that 
  $$ \l(\sum_{k=1}^{\infty} |a_k|^\gamma L_k^{\gamma}\r)^{\frac{1}{\gamma}}\leq \l(\sum_{k=1}^{\infty} |a_k|^\beta L_k^{\beta}\r)^{\frac{1}{\beta}}<\infty. $$ Hence, under 
  \eqref{eq1:dependent}, we have $\max_{k=1,\ldots}|a_k|L_k<\infty$. However, for $\theta\in(0,\alpha)$, when \eqref{eq1:dependent} is satisfied, $\sum_{k=1}^{\infty} |a_k|^\theta L_k^{\theta} $ can go to $\infty$. We are now ready to present the convergence of $M_n$.

 \begin{theorem}
Assume that $\{X_k\}$ is a martingale difference adapted to $\{\mathcal{F}_n\}_{n\in\mathbb{N}}$ and satisfies Assumption~\ref{assm:dependent}. Let $\{a_k\}$ be any sequence of coefficients  satisfying 
     \eqref{eq1:dependent}. Then we have that for $\alpha\geq 2$, 
        \begin{multline*}
        \PP\l(|M_\infty|\geq t\r)\leq 2\exp\l(-C\max\l\{\frac{t^\alpha}{\l(\sum_{k=1}^{\infty} |a_k|^\beta L_k^{\beta} \r)^{\frac{\alpha}{\beta}} }, \frac{t^2}{\sum_{k=1}^{\infty} a_k^2L_k^2},\r.\r.\\ \l.\l.\min\l\{\frac{t^2}{\sum_{k=1}^{\infty} a_k^2\sigma_k^2}, \frac{t}{\max_{k=1,\ldots} |a_k|L_k}\r\}\r\}\r),
    \end{multline*}
     and for $\alpha\in[1,2]$, 
     \begin{multline*}
        \PP\l(\l|M_\infty \r|\geq t\r)\\
        \leq  2\exp\l(-C\min\l\{\frac{t^2}{\sum_{k=1}^{\infty} a_k^2\sigma_k^2}, \max\l\{\frac{t}{\max_{k=1,\ldots}|a_k|L_k},\frac{t^\alpha}{\l( \sum_{k=1}^\infty |a_k|^\beta L_k^\beta\r)^{\frac{\alpha}{\beta}} }\r\}\r\} \r).
    \end{multline*}
     \label{thm:dependent}
 \end{theorem}
  It is noteworthy that when $\alpha\leq 2$, 
 \eqref{eq1:dependent} does guarantee that $\sum_{k=1}^{\infty}a_k^2L_k^2<\infty$. 
 Theorem~\ref{thm:dependent} above implies the following bound for the $\Psi_\alpha$-norm of $M_\infty$.

 \begin{corollary}[$\Psi_\alpha$-norm]
     Assume that the same conditions as in Theorem~\ref{thm:dependent} are satisfied. Then we have that for $\alpha\geq 2$, 
    \begin{align*}
        \l\|M_\infty\r\|_{\Psi_{\alpha}}\leq C_1\l(\sum_{i=1}^\infty |a_i|^{\beta}L_i^{\beta}\r)^{\frac{1}{\beta}},\quad \l\|M_\infty\r\|_{\Psi_{2}}\leq C_1\l(\sum_{i=1}^\infty a_i^2L_i^{2}\r)^{\frac{1}{2}},
    \end{align*}
    and for $\alpha\in[1,2]$, 
    \begin{align*}
    \l\| M_\infty\r\|_{\Psi_\alpha}\leq C_2 \l(\sum_{i=1}^\infty a_i^{2}\sigma_i^{2}\r)^{\frac{1}{2}}+C_2\l(\sum_{i=1}^\infty |a_i|^\beta L_i^{\beta}\r)^{\frac{1}{\beta}},
\end{align*}
where $C_1,C_2>0$ are constants that do not depend on $\alpha$ and $\{\sigma_n,L_n\}$.
 \end{corollary}

\subsection{Random vectors} \label{sec:vector}
  
In this section, we provide an application to the norm of random vectors of independent components. The bound of random vector norm can be traced back to the classical book of \cite{vershynin2018high}, where Theorem 3.1.1 therein gives the norm for random vectors with sub-Gaussian components. \cite{jeong2022sub} later improved it with a sharper tail probability. Here, we aim to improve and generalize the corresponding results in \cite{vershynin2018high} and \cite{jeong2022sub}. Theorem~\ref{thm:vec_an} below gives the concentration inequality for the Euclidean norm of a random vector, which is \textit{sharp} when $\sigma_X\asymp K$. Then Theorem~\ref{thm:vec:Sigma_L} later distinguishes $\sigma_X$ from the $\Psi_\alpha$-norm, which further generalizes Theorem~\ref{thm:vec_an}.

\begin{theorem}[Heterogeneous components]
    Assume that $X=(X_1,\ldots,X_d)\in\mathbb{R}^d$ is a mean-zero random vector with independent components, namely $\mathbb{E}X=0$, and $\|X_i\|_{\Psi_\alpha}\leq K_i<\infty$ for some $\alpha\geq 2$. Denote by $\sigma_{X_i}:=\|X_i\|_2$. Then we have that for $\alpha\geq 4$, 
    \begin{multline*}
        \PP\l(\l|\|X\| -\sqrt{\EE\|X\|^2}\r| \geq s\r)
        \leq 2\exp \l(-C\max\l\{\frac{s^4}{\sum_{i=1}^d K_i^4}, \frac{\sum_{i=1}^d \sigma_{X_i}^2}{\sum_{i=1}^d K_i^4}s^2,\r.\r.\\ \l.\l.\frac{s^\alpha}{\l(\sum_{i=1}^d K_i^{\frac{2\alpha}{\alpha-2}}\r)^{\frac{\alpha-2}{2}}}, \frac{s^{\frac{\alpha}{2}}\l(\sum_{i=1}^d \sigma_{X_i}^2\r)^{\frac{\alpha}{4}} }{\l(\sum_{i=1}^d K_i^{\frac{2\alpha}{\alpha-2}}\r)^{\frac{\alpha-2}{2}}}\r\}\r),
    \end{multline*}
    and for $\alpha\in[2,4]$, 
    \begin{multline*}
        \PP\l(\l|\|X\| -\sqrt{\EE\|X\|^2}\r| \geq s\r)
        \leq 2\exp \l(-C\min\l\{\max\l\{\frac{s^4}{\sum_{i=1}^d K_i^4}, \frac{\sum_{i=1}^d \sigma_{X_i}^2}{\sum_{i=1}^d K_i^4}s^2\r\}, \r.\r.\\ \l.\l.\max\l\{\frac{s^\alpha}{\l(\sum_{i=1}^d K_i^{\frac{2\alpha}{\alpha-2}}\r)^{\frac{\alpha-2}{2}}}, \frac{s^{\frac{\alpha}{2}}\l(\sum_{i=1}^d \sigma_{X_i}^2\r)^{\frac{\alpha}{4}} }{\l(\sum_{i=1}^d K_i^{\frac{2\alpha}{\alpha-2}}\r)^{\frac{\alpha-2}{2}}}\r\}\r\}\r).
    \end{multline*}
    \label{thm:vec_an}
\end{theorem}
  For sufficiently large $s$, the tail probability established in Theorem \ref{thm:vec_an} above is dominated by the $\Psi_\alpha$ one under both scenarios. When applying Theorem~\ref{thm:vec_an} to the identical distributions, the concentration inequality can be simplified as follows.
\begin{corollary}[Isotropic]
    Assume that $X=(X_1,\ldots,X_d)\in\mathbb{R}^d$ is a mean-zero random vector with independent and identically distributed (i.i.d.) components, namely $\mathbb{E}X=0$, and $\|X_i\|_{\Psi_\alpha}\leq K<\infty$ for some $\alpha\geq 2$. Denote by $\sigma_X:=\|X_i\|_2$ with $\sigma_X\leq K$. Then we have that for $\alpha\geq 4$, 
    \begin{align*}
        \PP\l( \l|\frac{1}{\sqrt{d}}\|X\|-\sigma_X \r|\geq s\r)\leq 2\exp\l(-Cd\max\l\{\frac{s^4}{K^4},\frac{\sigma_X^2s^2}{K^4},\frac{s^\alpha}{K^\alpha} \r\}\r),
    \end{align*}
    and for $\alpha\in[2,4]$, 
    \begin{align*}
        \PP\l(\l|\frac{1}{\sqrt{d}}\|X\|-\sigma_X \r|\geq s \r)\leq 2\exp\l(-Cd\min\l\{ \max\l\{\frac{s^4}{K^4},\frac{\sigma_X^2 s^2}{K^4}\r\},\frac{s^\alpha}{K^\alpha} \r\}\r).
    \end{align*}
    \label{thm:vec}
\end{corollary}

Corollary~\ref{thm:vec} above bounds the norm of random vector $X$ with i.i.d. components. In general, regardless of $\alpha\geq 4$ or $2\leq \alpha\leq 4$, for $s\leq \sigma_X$, the tail probability is $2\exp\l(-Cd\sigma_X^2 s^2/K^4\r)$; for $s\in[\sigma_X,K]$, the tail probability is $2\exp\l(-Cd s^4/K^4\r)$; and for $s\geq K$, the tail probability is $2\exp\l(-Cd s^\alpha/K^\alpha\r)$. When $\sigma_X\asymp K$, the bound is \textit{sharp} and the intermediate phase can be eliminated.

The theorem below provides a delicate bound of random vector norm under the characterization of moments.

\begin{theorem}
    Assume that $X=(X_1,\ldots,X_d)\in\mathbb{R}^d$ is a mean-zero random vector with independent components, namely $\mathbb{E}X=0$, and $X_i^2-\EE X_i^2$ satisfies Definition~\ref{assm:sigma_L} for some $\alpha\geq 1$ and $(\sigma,L)$. Denote by $\sigma_X^2\geq \operatorname{var}(X_i)$. Then we have that for $\alpha\geq 2$, 
    \begin{multline*}
        \PP\l(\l|\|X\|-\sqrt{\EE\|X\|^2} \r|\geq s \r)\leq 2\exp\l(-C\max\l\{\frac{s^{2\alpha}}{d^{\alpha-1}L^\alpha}, \frac{s^4}{dL^2}, \frac{s^2\sigma_{X}^2}{L^2},\r.\r. \\ \l.\l.\min\l\{ \max\l\{\frac{s^4}{d \sigma^2},\frac{\sigma_X^2}{\sigma^2}s^2\r\},\max\l\{\frac{s^2}{ L},\frac{s\sqrt{d} \sigma_X}{L}\r\}\r\}\r\}\r),
    \end{multline*}
    and for $1\leq \alpha\leq 2$, 
    \begin{multline*}
        \PP\l(\l|\|X\|-\sqrt{d}\sigma_X \r|\geq s \r)\leq 2\exp\l(-C\min\l\{\max\l\{\frac{s^4}{d \sigma^2}, \frac{s^2\sigma_{X}^2}{ \sigma^2}\r\}
        ,\r.\r.\\\l.\l. \min\l\{ \max\l\{\frac{s^{2\alpha}}{d^{\alpha-1} L^\alpha},\frac{d^{1-\frac{\alpha}{2}} \sigma_{X}^\alpha}{ L^\alpha}s^\alpha\r\},\max\l\{\frac{s^2}{L},\frac{s\sqrt{d }\sigma_X}{ L}\r\}\r\}\r\}\r).
    \end{multline*}
    \label{thm:vec:Sigma_L}
\end{theorem}
The proof of Theorem~\ref{thm:vec:Sigma_L} is in fact established for \textit{heterogeneous} components. For the clarity of presentation, in the main text here, we present only the case when the components are identically distributed. The example below provides the admissible values of $(\sigma,L)$ for $X_i^2-\EE X_i^2$ when $\|X_i\|_{\Psi_\alpha}<\infty$.
\begin{example}[Examples of $(\sigma,L)$]
    Denote by $\|X_{i}\|_{2,2}^2:=\operatorname{var}(X_i^2-\EE X_i^2)=\EE X_i^4-(\EE X_i^2)^2\leq \|X_i\|_4^4$. Then if $X_i$ is mean zero with $\sigma_{X_i}=\sqrt{\operatorname{var}(X_i)}$ and $\|X_i\|_{\Psi_\alpha}<\infty$, quantity $X_i^2-\EE X_i^2$ satisfies Definition~\ref{assm:sigma_L} with $\frac{\alpha}{2}$, i.e., $\|X_i^2-\EE X_i^2\|_{\Psi_{\frac{\alpha}{2}}}\leq  2\|X_i\|_{\Psi_\alpha}<\infty$. In view of Example~\ref{exp:1}, it holds that 
    \begin{align*}
        \EE |X_i^2-\EE X_i^2|^2&=\|X_i\|_{2,2}^2\leq \|X_i\|_4^4\leq C\sigma_{X_i}^2\|X_i\|_{\Psi_\alpha}^2 \log\l(\frac{\|X_i\|_{\Psi_\alpha}}{\sigma_{X_i}}\r),\\
        \EE |X_i^2-\EE X_i^2|^k&\leq C^k\EE|X_i|^{2k}\leq C^k k^{\frac{2k}{\alpha}} \sigma_{X_i}^2\|X_i\|_{\Psi_\alpha}^{2k-2} \log^{k-1}\l(\frac{\|X_i\|_{\Psi_\alpha}}{\sigma_{X_i}}\r),
    \end{align*}
    which indicate that $X^2-\EE X^2$ satisfies Definition~\ref{assm:sigma_L} with $\frac{\alpha}{2}$ for 
    \begin{align}
        \sigma:= C\sigma_{X_i}\|X_i\|_{\Psi_\alpha}\log^{\frac{1}{2}}\l(\frac{\|X_i\|_{\Psi_\alpha}}{\sigma_{X_i}}\r),\quad L:=C\|X_i\|_{\Psi_{ \alpha }}^2\log\l(\frac{2\|X_i\|_{\Psi_{\alpha}}}{\sigma_{X_i}}\r).
        \label{eq1:vec}
    \end{align}
    On the other hand, it is noteworthy that the values of $(\sigma,L)$ are \textit{not} unique. An application of Example~\ref{exp:2} leads to 
    \begin{align*}
        \EE |X_i^2-\EE X_i^2|^2&=\|X\|_{2,2}^2\leq \|X_i\|_4^4\leq C\sigma_{X_i}\|X_i\|_{\Psi_\alpha}^3,\\
        \EE |X_i^2-\EE X_i^2|^k&\leq C^k\EE|X_i|^{2k}\leq C^k k^{\frac{2k}{\alpha}} \sigma_{X_i}\|X_i\|_{\Psi_\alpha}^{2k-1},
    \end{align*}
    which entail that $X_i^2-\EE X_i^2$ satisfies Definition~\ref{assm:sigma_L} with $\frac{\alpha}{2}$ for
    \begin{align}
        \sigma:=C\sqrt{\sigma_{X_i} \|X_i\|_{\Psi_\alpha}}\|X\|_{\Psi_\alpha},\quad L:= C\|X_i\|_{\Psi_\alpha}^2.
        \label{eq2:vec}
    \end{align}
    \label{exp:8}
\end{example}

\begin{remark}[Comparisons to existing works under $\alpha=2$]
    Here, we consider the scenario when random vector $X$ has i.i.d. mean-zero components, with $\sigma_X^2=1$ and $\|X\|_{\Psi_2}\leq K$. Theorem 3.1.1 in 
    \cite{vershynin2018high} gives the upper bound for $\|X\|$
\begin{align*}
    \PP\l(\l|\|X\| -\sqrt{d \sigma_{X}^2}\r| \geq s\r)
        \leq 2\exp \l(-C\frac{s^2}{K^4}\r),
\end{align*}
which was \textit{conjectured not} sharp in \cite{vershynin2018high}. \cite{jeong2022sub} improved it to 
\begin{align*}
    \PP\l(\l|\|X\| -\sqrt{d \sigma_{X}^2}\r| \geq s\r)
        \leq 2\exp \l(-C\frac{s^2}{K^2\log (K)}\r).
\end{align*}
On the other hand, when substituting 
\eqref{eq1:vec} into Theorem~\ref{thm:vec:Sigma_L} above, we can obtain that 
\begin{multline*}
     \PP\l(\l|\|X\| -\sqrt{d}\sigma_X\r| \geq s\r)
        \\
        \leq 2\exp\l(-C\min\l\{\max\l\{\frac{s^4}{d \sigma_X^2\|X\|_{\Psi_2}^2\log\l(\frac{2\|X\|_{\Psi_2}}{\sigma_X}\r)}, \frac{s^2}{ \|X\|_{\Psi_2}^2\log\l(\frac{2\|X\|_{\Psi_2}}{\sigma_X}\r)}\r\}
        ,\r.\r.\\\l.\l. \max\l\{\frac{s^2}{\|X\|_{\Psi_2}^2\log\l(\frac{2\|X\|_{\Psi_2}}{\sigma_X}\r)},\frac{s\sqrt{d }\sigma_X}{ \|X\|_{\Psi_2}^2\log\l(\frac{2\|X\|_{\Psi_2}}{\sigma_X}\r)}\r\}\r\}\r)\\
        = 2\exp\l(-C\frac{s^2}{\|X\|_{\Psi_2}^2\log\l(\frac{2\|X\|_{\Psi_2}}{\sigma_X}\r)}\r).
\end{multline*}
When applying 
\eqref{eq2:vec} to Theorem~\ref{thm:vec:Sigma_L}, we can deduce that 
\begin{multline*}
     \PP\l(\l|\|X\| -\sqrt{d}\sigma_X\r| \geq s\r)
        \leq 2\exp\l(-C\min\l\{\max\l\{\frac{s^4}{d \sigma_X\|X\|_{\Psi_2}^3}, \frac{s^2\sigma_X}{ \|X\|_{\Psi_2}^3}\r\}
        ,\r.\r.\\\l.\l. \max\l\{\frac{s^2}{\|X\|_{\Psi_2}^2},\frac{s\sqrt{d }\sigma_X}{ \|X\|_{\Psi_2}^2}\r\}\r\}\r).
\end{multline*}
Consequently, the two inequalities above together entail that 
\begin{multline*}
    \PP\l(\l|\|X\| -\sqrt{d}\sigma_X\r| \geq s \r)\\
    \leq \begin{cases}
        2\exp\l(-\frac{cs^2}{\|X\|_{\Psi_2}^2\log\l(\frac{2\|X\|_{\Psi_2}}{\sigma_X}\r)}\r)& \text{ if } s\leq \sqrt{d}\sqrt{\frac{\sigma_X\|X\|_{\Psi_2}}{\log \l(\frac{\|X\|_{\Psi_2}}{\sigma_X}\r)}},\\
        2\exp\l(-\frac{cs^4}{d\sigma_X\|X\|_{\Psi_\alpha}^3}\r)& \text{ if } \sqrt{d}\sqrt{\frac{\sigma_X\|X\|_{\Psi_2}}{\log \l(\frac{\|X\|_{\Psi_2}}{\sigma_X}\r)}}\leq s\leq \sqrt{d}\sqrt{\sigma_X\|X\|_{\Psi_2}},\\
        2\exp\l(-\frac{cs^2}{\|X\|_{\Psi_2}^2}\r)&\text{ otherwise},
    \end{cases} 
\end{multline*}
which improves Theorem~3.1.1 of 
\cite{vershynin2018high} and the result in \cite{jeong2022sub}. Additionally, our tail probability bound can be written as 
\begin{multline*}
    \PP\l(\l|\|X\| -\sqrt{d}\sigma_X\r| \geq s \r)\\
    \leq2\exp\l(-C\max\l\{\frac{s^2}{\|X\|_{\Psi_2}^2\log\l(\frac{2\|X\|_{\Psi_2}}{\sigma_X}\r)},\min\l\{\frac{s^4}{d\sigma_X\|X\|_{\Psi_\alpha}^3}, \frac{s^2}{\|X\|_{\Psi_2}^2} \r\} \r\}\r).
\end{multline*}
\end{remark}

\subsection{Random matrices} \label{new.Sec.randmat}

In this section, we investigate the largest and smallest nonzero \textit{singular values} of a random matrix. The study of the operator norm for random matrices can be traced back to the elegant work of \cite{bai1993limit}, where asymptotic convergence was established for the eigenvalues of sample covariance matrices in the scenario of i.i.d. Gaussian entries. See, e.g., the 
works of \cite{latala2005some}, \cite{davidson2001local}, \cite{vershynin2010introduction}, \cite{fanfanlvyangyu2025}, and references therein. Specifically, \cite{vershynin2010introduction} focused on the nonasymptotic behaviors of singular values with independent rows or columns where the rows or columns are isotropic sub-Gaussian. Here, we will extend the results in \cite{vershynin2010introduction} to general sub-Weibull distributions, and improve the results when standard deviation $\sigma_X$ needs to be distinguished from the Orlicz norm.

Let $X$ be a $d_1\times d_2$ random matrix, with rows $X=\l(X_1 ,\ldots,X_{d_1} \r)^\top $. Here, $X_1,\ldots,X_{d_1}\in\mathbb{R}^{d_2}$ are independently distributed. We emphasize that the components of $X_i$ may be dependent. The operator norm of $X$ is defined as 
$$\|X\|:=\sup_{u\in\mathbb{S}^{d-1}}\|Xu\|.$$ 
The lemma below upper bounds the difference between $\frac{1}{d_1}X^\top X$ and its expectation $\Sigma$.
\begin{lemma}
    Assume that random matrix $X= (X_1 ,\ldots,X_{d_1} )^\top\in\RR^{d_1\times d_2}$ contains i.i.d. mean-zero rows, and there exist some $\alpha\geq 1$, $\sigma,L>0$ such that for any $u\in\SS^{d_2-1}$ and all $k\geq 2$, 
    \begin{align}
        \EE \l|\l(X_i^\top u\r)^2-\EE\l(X_i^\top u\r)^2 \r|^k\leq k^{\frac{k}{\alpha}}\sigma^2L^{k-2}.
        \label{eq1:new.Sec.randmat}
    \end{align} 
    Denote by $\Sigma:=\EE X_iX_i^{\top}$ the population covariance matrix. Then when $\alpha\geq 2$, it holds for all $t>0$ that with probability over $1- \exp(-t)$, 
    \begin{align*}
    \l\|\frac{1}{d_1}X^{\top}X -\Sigma\r\|\leq C\min\l\{L\l(\frac{t+d_2}{d_1}\r)^{\frac{1}{\alpha}},\sigma\l(\frac{t+d_2}{d_1}\r)^{\frac{1}{2}}+L\cdot\frac{t+d_2}{d_1}\r\}.
\end{align*}
When $\alpha\in[1,2]$, it holds for all $t>0$ that with probability over $1-\exp(-t)$, 
\begin{align*}
    \l\|\frac{1}{d_1}X^{\top}X -\Sigma\r\|\leq C\sigma\l(\frac{t+d_2}{d_1}\r)^{\frac{1}{2}}+CL\min\l\{\l(\frac{t+d_2}{d_1}\r)^{\frac{1}{\alpha}},\frac{t+d_2}{d_1}\r\}.
\end{align*}
    \label{lem:matrix}
\end{lemma}

Lemma~\ref{lem:matrix} above unveils an interesting phase transition at $\alpha=2$. Let $s_{\max}(X)$ be the maximum singular value of $X$, and $s_{\min}(X)$ its smallest nonzero singular value. Denote by $M_{\max} :=\sqrt{\lambda_{\max}(\Sigma)}$ and $M_{\min} :=\sqrt{\lambda_{\min}(\Sigma)}$.

\begin{theorem}
 Assume that the same conditions as in Lemma~\ref{lem:matrix} are satisfied. Then for any $\alpha\geq 2$ and $t>0$, we have that with probability over $1-\exp(-t)$, 
    \begin{multline*}
        d_1 M_{\min}^2-C\min\l\{L(d_2+t)^{\frac{1}{\alpha}}d_1^{\frac{1}{\beta}},\sigma(d_2+t)^{\frac{1}{2}}d_1^{\frac{1}{2}}+L(d_2+t)\r\}\leq s_{\min}^2\\
        \leq s_{\max}^2\leq d_1 M_{\max}^2+C\min\l\{L(d_2+t)^{\frac{1}{\alpha}}d_1^{\frac{1}{\beta}},\sigma(d_2+t)^{\frac{1}{2}}d_1^{\frac{1}{2}}+L(d_2+t)\r\}.
    \end{multline*}
    For any $1\leq \alpha\leq 2$ and $t>0$, we have that with probability over $1-\exp(-t)$,
    \begin{multline*}
        d_1 M_{\min}^2-\sigma(d_2+t)^{\frac{1}{2}}d_1^{\frac{1}{2}}-CL\min\l\{(d_2+t)^{\frac{1}{\alpha}}d_1^{\frac{1}{\beta}},d_2+t\r\}\leq s_{\min}^2\\
        \leq s_{\max}^2\leq d_1 M_{\max}^2+C\sigma(d_2+t)^{\frac{1}{2}}d_1^{\frac{1}{2}}+CL\min\l\{(d_2+t)^{\frac{1}{\alpha}}d_1^{\frac{1}{\beta}},d_2+t\r\}.
    \end{multline*}
    \label{thm:singular_value}
\end{theorem}

Theorem~\ref{thm:singular_value} above provides high probability bounds for the singular values of $X$, with explicit dependence on $M_{\max}, M_{\min}, \sigma,L$, and $\alpha,d_1,d_2$. Theorem~\ref{thm:singular_value} is consistent with the classical asymptotic results for the case of i.i.d. Gaussian entries \citep{bai1993limit}.

\begin{example}[Random matrix with i.i.d. sub-Gaussian entries]
    Here, we provide the example values of $(\sigma,L)$ when the $d_1\times d_2$ random matrix $X$ consists of i.i.d. mean-zero sub-Gaussian entries. Assume that each entry of $X$ has variance $\sigma_{X}^2$ and Orlicz norm $\|X_{ij}\|_{\Psi_2}$. Then we have $M_{\max}=\sigma_{X}$ and for any $u\in\SS^{d-1}$, 
    $$\operatorname{var}(X_i^{\top}u)=\sigma_{X}^2<\infty,\quad\| X_i^{\top} u\|_{\Psi_2}\leq C\|X_{ij}\|_{\Psi_2}=:K<\infty.$$
    In particular, Example~\ref{exp:8} implies that the conditions of Lemma~\ref{lem:matrix} hold with $\alpha=1$ and three valid pairs of $(\sigma,L)$:
    $(K^2,K^2 )$, $ (\sigma_{X}^{1/2}K^{3/2},CK^2 )$, and 
    $(\sigma_{X}K\log^{1/2}\l(\frac{2K}{\sigma_{X}}\r),CK^2\log\l(\frac{2K}{\sigma_{X}}\r) ).$
    Hence, an application of Theorem~\ref{thm:singular_value} and the above values of $(\sigma,L)$ yields that for any $t>0$, we have that with probability over $1-\exp(-t)$,
    \begin{multline*}
        s_{\max}^2\leq d_1 \sigma_{X}^2+C\sqrt{d_2+t}\min\l\{K^{3/2} {\sigma_{X}^{1/2}}\sqrt{d_1}+K^2\sqrt{d_2+t},\r.\\ \l.\sigma_{X}K\sqrt{d_1} \log^{\frac{1}{2}}\l(\frac{2K}{\sigma_{X}}\r)+K^2\sqrt{d_2+t} \log\l(\frac{2K}{\sigma_{X}}\r)\r\}.
    \end{multline*}
    Moreover, when $t\geq \frac{\sigma_{X}}{K}d_1-d_2$, it can be simplified as $\PP(s_{\max}\leq \sqrt{d_1} \sigma_{X}+CK\sqrt{d_2+t})\geq 1-\exp(-t)$.
\end{example}

\subsection{Covariance matrix estimation} \label{new.Sec.covmatest}

As another application, we examine in this section the problem of mean and covariance matrix estimation based on a sample of $n$ observed random vectors. See, e.g., the works of 
\cite{depersin2022robust}, \cite{minsker2018sub}, \cite{koltchinskii2017concentration}, and references therein. We emphasize that our focus here is \textit{different}. In particular, \cite{koltchinskii2017concentration} investigated the delicate dependence of estimation error on the effective rank, under the assumption that the standard deviation of $X_i^\top u$ has the \textit{same} scale as the Orlicz norm. In contrast, we aim to \textit{distinguish} these two quantities and examine their effects on covariance matrix estimation.

Assume that $X_1,\ldots,X_{n}\in\mathbb{R}^{d}$ are i.i.d. $d$-dimensional random vectors with \textit{unknown} mean $\mu$ and covariance matrix $\Sigma$, i.e., 
\begin{align*}
    \EE X_i=\mu,\quad \EE(X_i-\mu)(X_i-\mu)^{\top}=\Sigma.
\end{align*}
We estimate $\mu$ and $\Sigma$ with the sample mean and sample variance matrix
\begin{align*}
    \wh{\mu}:=\frac{1}{n}\sum_{i=1}^n X_i,\quad \wh{\Sigma}:=\frac{1}{n}\sum_{i=1}^n (X_i-\wh{\mu})(X_i-\wh{\mu})^{\top},
\end{align*}
respectively. The lemma below bounds the estimation error of sample mean under the characterization of moments.

\begin{lemma}[Mean estimation]
Assume that $X_1,\ldots,X_n$ are i.i.d. $d$-dimensional random vectors, and for any $u\in\SS^{d-1}$, $ \l(X-\mu\r)^{\top}u$ satisfies Definition~\ref{assm:sigma_L} with some $\alpha\geq 1$ and $(\sigma,L)$. 
Then we have that for $\alpha\geq 2$ and all $t > 0$, 
    \begin{align*}
        \PP\l(\|\wh{\mu}-\mu\|\geq C \min\l\{L\l(\frac{t+d}{n}\r)^{\frac{1}{\alpha}},\sigma\l(\frac{t+d}{n}\r)^{\frac{1}{2}}+L\frac{t+d}{n}\r\} \r)\leq \exp(-t),
    \end{align*}
    and for $1\leq \alpha\leq 2$ and all $t > 0$, 
    \begin{align*}
        \PP\l(\|\wh{\mu}-\mu\|\geq C\sigma\l(\frac{t+d}{n}\r)^{\frac{1}{2}}+C \min\l\{L\l(\frac{t+d}{n}\r)^{\frac{1}{\alpha}},L\frac{t+d}{n}\r\} \r)\leq \exp(-t).
    \end{align*}
    \label{lem:mean_est}
\end{lemma}

Lemma~\ref{lem:mean_est} above again reveals an interesting phase transition at $\alpha=2$. However, in both regimes, for sufficiently small $t$, the concentration presents itself as sub-Gaussian with deviation $\sigma\l(\frac{t+d}{n}\r)^{\frac{1}{2}}$, whereas when $t$ is large enough, the deviation becomes $L\l(\frac{t+d}{n}\r)^{\frac{1}{\alpha}}$.

\begin{theorem}[Covariance matrix estimation]
Assume that $X_1,\ldots,X_n$ are i.i.d. $d$-dimensional random vectors, and for any $u\in\SS^{d-1}$, $ \l(X-\mu\r)^{\top}u$ satisfy Definition~\ref{assm:sigma_L} with some $\alpha\geq 1$ and $(\sigma,L)$. 
Then we have that for $\alpha\geq 4$ and all $t > 0$, 
    \begin{align*}
        \PP\l(\l\|\wh\Sigma-\Sigma \r\|\geq C \min\l\{L^2\l(\frac{t+d}{n}\r)^{\frac{2}{\alpha}},\sigma L\l(\frac{t+d}{n}\r)^{\frac{1}{2}}+L^2 \frac{t+d}{n} \r\}\r)\leq 2\exp(-t),
    \end{align*}
    and for $2\leq \alpha\leq 4$ and all $t > 0$, 
    \begin{align*}
    \PP\l(\l\|\wh\Sigma-\Sigma \r\|\geq C \sigma L\l(\frac{t+d}{n}\r)^{\frac{1}{2}}+C\min\l\{L^2\l(\frac{t+d}{n}\r)^{\frac{2}{\alpha}},L^2 \frac{t+d}{n} \r\}\r)\leq 2\exp(-t).
\end{align*}
    \label{thm:cov_est}
\end{theorem}

Theorem~\ref{thm:cov_est} above upper bounds the estimation error of sample covariance matrix. Here, the estimation error also has a sub-Gaussian tail for sufficiently small $t$, whereas it admits a $\frac{\alpha}{2}$ decay tail when $t$ is large enough. Indeed, the sample covariance matrix is quadratic in $X$ so that the power on dimensionality $d$ is $\frac{2}{\alpha}$. For a random vector $X_i$, let $\|X_i\|_{\Psi_\alpha}:=\sup_{u\in\SS^{d-1}} \|X_i^{\top} u\|_{\Psi_\alpha}$. If $\|X_i\|_{\Psi_\alpha}<\infty$, $X_i$ satisfies the conditions of Theorem~\ref{thm:cov_est} with $\alpha$ and infinite pairs of $(\sigma,L)$. Denote by $\sigma_{X_i}^2\geq \sup_{u\in\SS^{d-1}}\operatorname{var}(X_i^\top u)$. We provide three specific examples of $(\sigma,L)$ below
\begin{align*}
    &\sigma_1:=\|X_i\|_{\Psi_\alpha},\quad L_1:=\|X_i\|_{\Psi_\alpha};\\
    &\sigma_2:=\sigma_{X_i},\quad L_2:=C\|X_i\|_{\Psi_\alpha}\log^{\frac{1}{\alpha}}\l(\frac{2\|X_i\|_{\Psi_\alpha}}{\sigma_{X_i}}\r);\\
&\sigma_3:=\sqrt{\sigma_{X_i}\|X_i\|_{\Psi_\alpha}},\quad L_3:=\|X_i\|_{\Psi_\alpha},\\
\end{align*}
which follow naturally from Examples~\ref{exp:1} and \ref{exp:2}.

\begin{example}[Covariance matrix estimation for sub-Gaussian entries]
    We illustrate Theorem~\ref{thm:cov_est}  when $X_i$ consists of i.i.d. components. Assume that each component of the random vector has variance $\sigma_{X}^2$ and Orlicz norm $\|X_{ij}\|_{\Psi_2}$. Then for any $u\in\SS^{d-1}$, it holds that 
    $$\operatorname{var}(X_i^{\top}u)=\sigma_{X}^2,\quad \| X_i^{\top} u\|_{\Psi_2}\leq C\|X_{ij}\|_{\Psi_2}=:K,$$
    which along with Examples~\ref{exp:1} and \ref{exp:2} entail that $X$ satisfies the conditions of Theorem~\ref{thm:cov_est} with $\alpha$ and the following three pairs of $(\sigma,L)$
    $$(K,K),\;(\sqrt{\sigma_{X}K},K),\;(\sigma_{X},K\log^{1/2}\l(\frac{2K}{\sigma_{X}}\r)).$$
    Consequently, an application of Theorem~\ref{thm:cov_est} and the above three pairs of $(\sigma,L)$ gives that for all $t >0$, with probability over $1-2\exp(-t)$ 
    \begin{multline*}
        \l\|\wh\Sigma-\Sigma \r\|\leq C \min \l\{\sigma_{X} K\log^{1/2}\l(\frac{2K}{\sigma_{X}}\r)\l(\frac{t+d}{n}\r)^{\frac{1}{2}}\r.\\ \l. +K^2\log\l(\frac{2K}{\sigma_{X}}\r) \frac{t+d}{n}, \sqrt{\sigma_{X}}K^{\frac{3}{2}}\l(\frac{t+d}{n}\r)^{\frac{1}{2}}+K^2\frac{t+d}{n}\r\}.
    \end{multline*}
    The covariance matrix estimation error above unveils a \textit{delicate interplay} between the standard deviation $\sigma_{X}$ and the Orlicz norm. For sufficiently small $t$, it is dominated by $\sigma_{X} K\log^{1/2}\l(\frac{2K}{\sigma_{X}}\r)\l(\frac{t+d}{n}\r)^{\frac{1}{2}} $, whereas when $t$ is large enough, it is dominated by $K^2\frac{t+d}{n} $.
\end{example}

\section{Discussions} \label{Sec.disc}

We have investigated in this paper the problem of how to develop sharp concentration inequalities of sub-Weibull random variables with general rate parameter $\alpha \geq 1$, including the commonly used sub-Gaussian and sub-exponential distributions with $\alpha = 2$ and $1$, respectively. Such new theoretical results will enable us to conduct more precise non-asymptotic analyses across different statistical and machine learning applications. Our unified concentration bounds involving the Orlicz norm have revealed an interesting phase transition at $\alpha = 2$, with the minimum of two quantities switching to the maximum once $\alpha$ is above $2$. Further, when the Orlicz norm can exceed the standard deviation greatly, we have established sharp, flexible concentration bounds that involve the variance and a mixing of Orlicz $\Psi_\alpha$-tails through the min and max functions. These sharp concentration inequalities are new even for the cases of sub-Gaussian and sub-exponential distributions with $\alpha = 2$ and $1$. We have showcased the utilities of our new theory with applications to martingales, random vectors, random matrices, and covariance matrix estimation. It would be interesting to extend our theory to more general settings of Banach space-valued random variables, reproducing kernel Hilbert spaces (RKHSs), and time series or online adaptive data. These problems are beyond the scope of the current paper and will be interesting topics for future research.


\bibliographystyle{plainnat}
\bibliography{reference}

@article{fanfanlvyangyu2025,
  title={Asymptotic theory of eigenvectors for latent embeddings with generalized {L}aplacian matrices},
  author={Fan, J. and Fan, Y. and Lv, J. and Yang, F. and Yu, D.},
  journal={arXiv preprint arXiv:2503.00640},
  year={2025}
}

@article{talagrand1994supremum,
  title={The supremum of some canonical processes},
  author={Talagrand, Michel},
  journal={American Journal of Mathematics},
  volume={116},
  number={2},
  pages={283--325},
  year={1994},
  publisher={JSTOR}
}

@article{zajkowski2020norms,
  title={On norms in some class of exponential type {O}rlicz spaces of random variables},
  author={Zajkowski, Krzysztof},
  journal={Positivity},
  volume={24},
  number={5},
  pages={1231--1240},
  year={2020},
  publisher={Springer}
}

@article{zhang2022sharper,
  title={Sharper sub-{W}eibull concentrations},
  author={Zhang, Huiming and Wei, Haoyu},
  journal={Mathematics},
  volume={10},
  number={13},
  pages={2252},
  year={2022},
  publisher={MDPI}
}

@article{kuchibhotla2022moving,
  title={Moving beyond sub-{G}aussianity in high-dimensional statistics: applications in covariance estimation and linear regression},
  author={Kuchibhotla, Arun Kumar and Chakrabortty, Abhishek},
  journal={Information and Inference},
  volume={11},
  number={4},
  pages={1389--1456},
  year={2022},
  publisher={Oxford University Press}
}

@article{van2013bernstein,
  title={The {B}ernstein-{O}rlicz norm and deviation inequalities},
  author={van de Geer, Sara and Lederer, Johannes},
  journal={Probability Theory and Related Fields},
  volume={157},
  number={1},
  pages={225--250},
  year={2013},
  publisher={Springer Berlin Heidelberg}
}

@article{hao2019bootstrapping,
  title={Bootstrapping upper confidence bound},
  author={Hao, Botao and Abbasi Yadkori, Yasin and Wen, Zheng and Cheng, Guang},
  journal={Advances in Neural Information Processing Systems},
  volume={32},
  year={2019}
}

@article{adamczak2011restricted,
  title={Restricted isometry property of matrices with independent columns and neighborly polytopes by random sampling},
  author={Adamczak, Radoslaw and Litvak, Alexander E. and Pajor, Alain and Tomczak-Jaegermann, Nicole},
  journal={Constructive Approximation},
  volume={34},
  number={1},
  pages={61--88},
  year={2011},
  publisher={Springer}
}

@book{boucheron2003concentration,
  title={Concentration Inequalities},
  author={Boucheron, St{\'e}phane and Lugosi, G{\'a}bor and Bousquet, Olivier},
  year={2003},
  publisher={Oxford University Press}
}

@book{ledoux2013probability,
  title={Probability in Banach Spaces: Isoperimetry and Processes},
  author={Ledoux, Michel and Talagrand, Michel},
  year={2013},
  publisher={Springer Science \& Business Media}
}

@article{talagrand1989isoperimetry,
  title={Isoperimetry and integrability of the sum of independent {B}anach-space valued random variables},
  author={Talagrand, Michel},
  journal={The Annals of Probability},
  pages={1546--1570},
  year={1989},
  publisher={JSTOR}
}

@book{vershynin2018high,
  title={High-Dimensional Probability: An Introduction with Applications in Data Science},
  author={Vershynin, Roman},
  volume={47},
  year={2018},
  publisher={Cambridge University Press}
}

@article{latala1997estimation,
  title={Estimation of moments of sums of independent real random variables},
  author={Latala, Rafal},
  journal={The Annals of Probability},
  volume={25},
  number={3},
  pages={1502--1513},
  year={1997},
  publisher={Institute of Mathematical Statistics}
}

@article{koltchinskii2011neumann,
  title={Von {N}eumann Entropy Penalization and Low-Rank Matrix Estimation},
  author={Koltchinskii, Vladimir},
  journal={The Annals of Statistics},
  pages={2936--2973},
  year={2011},
  publisher={JSTOR}
}

@article{recht2011simpler,
  title={A simpler approach to matrix completion.},
  author={Recht, Benjamin},
  journal={Journal of Machine Learning Research},
  volume={12},
  number={12},
  year={2011}
}

@article{gross2010quantum,
  title={Quantum state tomography via compressed sensing},
  author={Gross, David and Liu, Yi-Kai and Flammia, Steven T. and Becker, Stephen and Eisert, Jens},
  journal={Physical Review Letters},
  volume={105},
  number={15},
  pages={150401},
  year={2010},
  publisher={APS}
}

@article{gross2011recovering,
  title={Recovering low-rank matrices from few coefficients in any basis},
  author={Gross, David},
  journal={IEEE Transactions on Information Theory},
  volume={57},
  number={3},
  pages={1548--1566},
  year={2011},
  publisher={IEEE}
}

@article{alquier2013sparse,
  title={Sparse Single-Index Model},
  author={Alquier, Pierre and Biau, G{\'e}rard},
  journal={Journal of Machine Learning Research},
  volume={14},
  pages={243--280},
  year={2013}
}

@article{ahlswede2002strong,
  title={Strong converse for identification via quantum channels},
  author={Ahlswede, Rudolf and Winter, Andreas},
  journal={IEEE Transactions on Information Theory},
  volume={48},
  number={3},
  pages={569--579},
  year={2002},
  publisher={IEEE}
}

@article{shen2026sgd,
  title={{SGD} with Dependent Data: optimal Estimation, Regret, and Inference},
  author={Shen, Yinan and Zhang, Yichen and Zhou, Wen-Xin},
  journal={arXiv preprint arXiv:2601.01371},
  year={2026}
}

@article{rio2013extensions,
  title={Extensions of the {H}oeffding-{A}zuma inequalities},
  author={Rio, Emmanuel},
  journal={Electronic Communications in Probability},
  volume={18},
  number={54},
  pages={6p},
  year={2013}
}

@article{bennett1962probability,
  title={Probability inequalities for the sum of independent random variables},
  author={Bennett, George},
  journal={Journal of the American Statistical Association},
  volume={57},
  number={297},
  pages={33--45},
  year={1962},
  publisher={Taylor \& Francis}
}

@article{minsker2017some,
  title={On some extensions of {B}ernstein’s inequality for self-adjoint operators},
  author={Minsker, Stanislav},
  journal={Statistics \& Probability Letters},
  volume={127},
  pages={111--119},
  year={2017},
  publisher={Elsevier}
}

@incollection{koltchinskii2016perturbation,
  title={Perturbation of linear forms of singular vectors under {G}aussian noise},
  author={Koltchinskii, Vladimir and Xia, Dong},
  booktitle={High Dimensional Probability VII: The Carg{\`e}se Volume},
  pages={397--423},
  year={2016},
  publisher={Springer}
}

@article{jeong2022sub,
  title={Sub-{G}aussian Matrices on Sets: optimal Tail Dependence and Applications},
  author={Jeong, Halyun and Li, Xiaowei and Plan, Yaniv and Yilmaz, Ozgur},
  journal={Communications on Pure and Applied Mathematics},
  volume={75},
  number={8},
  pages={1713--1754},
  year={2022},
  publisher={Wiley Online Library}
}

@article{vershynin2010introduction,
  title={Introduction to the non-asymptotic analysis of random matrices},
  author={Vershynin, Roman},
  journal={arXiv preprint arXiv:1011.3027},
  year={2010}
}

@article{lin2025semiparametric,
  title={Semiparametric inference based on adaptively collected data},
  author={Lin, Licong and Khamaru, Koulik and Wainwright, Martin J.},
  journal={The Annals of Statistics},
  volume={53},
  number={3},
  pages={989--1014},
  year={2025},
  publisher={Institute of Mathematical Statistics}
}

@article{khamaru2025near,
  title={Near-optimal inference in adaptive linear regression},
  author={Khamaru, Koulik and Deshpande, Yash and Lattimore, Tor and Mackey, Lester and Wainwright, Martin J.},
  journal={The Annals of Statistics},
  volume={53},
  number={6},
  pages={2329--2355},
  year={2025},
  publisher={Institute of Mathematical Statistics}
}

@article{zhang2018tensor,
  title={Tensor {SVD}: statistical and computational limits},
  author={Zhang, Anru and Xia, Dong},
  journal={IEEE Transactions on Information Theory},
  volume={64},
  number={11},
  pages={7311--7338},
  year={2018},
  publisher={IEEE}
}

@article{abdalla2026dimension,
  title={On the dimension-free concentration of simple tensors via matrix deviation},
  author={Abdalla, Pedro and Vershynin, Roman},
  journal={Journal of Theoretical Probability},
  volume={39},
  number={1},
  pages={3},
  year={2026},
  publisher={Springer}
}

@article{ma2024high,
  title={High-probability minimax lower bounds},
  author={Ma, Tianyi and Verchand, Kabir A. and Samworth, Richard J.},
  journal={arXiv preprint arXiv:2406.13447},
  year={2024}
}

@article{minsker2018sub,
  title={Sub-{G}aussian estimators of the mean of a random matrix with heavy-tailed entries},
  author={Minsker, Stanislav},
  journal={The Annals of Statistics},
  volume={46},
  number={6A},
  pages={2871--2903},
  year={2018},
  publisher={JSTOR}
}

@article{depersin2022robust,
  title={Robust sub-{G}aussian Estimation of a Mean Vector in Nearly Linear Time},
  author={Depersin, Jules and Lecu{\'e}, Guillaume},
  journal={The Annals of Statistics},
  volume={50},
  number={1},
  pages={511--536},
  year={2022}
}

@article{zhou2025deflated,
  title={Deflated {HeteroPCA}: overcoming the curse of ill-conditioning in heteroskedastic {PCA}},
  author={Zhou, Yuchen and Chen, Yuxin},
  journal={The Annals of Statistics},
  volume={53},
  number={1},
  pages={91--116},
  year={2025},
  publisher={Institute of Mathematical Statistics}
}

@article{adamczak2008tail,
  title={A tail inequality for suprema of unbounded empirical processes with applications to {M}arkov chains.},
  author={Adamczak, Radoslaw},
  journal={Electronic Journal of Probability},
  volume={13},
  pages={1000--1034},
  year={2008},
  publisher={University of Washington, Department of Mathematics, Seattle, WA; Duke~…}
}

@book{tao2023topics,
  title={Topics in Random Matrix Theory},
  author={Tao, Terence},
  volume={132},
  year={2012},
  publisher={American Mathematical Society}
}

@article{bai1993limit,
  title={Limit of the smallest eigenvalue of a large dimensional sample covariance matrix},
  author={Bai, Z. D. and Yin, Y. Q.},
  journal={Annals of Probability},
  volume={21},
  number={3},
  pages={1275--1294},
  year={1993},
  publisher={World Scientific}
}

@article{latala2005some,
  title={Some estimates of norms of random matrices},
  author={Latala, Rafal},
  journal={Proceedings of the American Mathematical Society},
  volume={133},
  number={5},
  pages={1273--1282},
  year={2005}
}

@incollection{davidson2001local,
  title={Local operator theory, random matrices and {B}anach spaces},
  author={Davidson, Kenneth R. and Szarek, Stanislaw J.},
  booktitle={Handbook of the Geometry of Banach Spaces},
  volume={1},
  pages={317--366},
  year={2001},
  publisher={Elsevier}
}

@article{koltchinskii2017concentration,
  title={Concentration inequalities and moment bounds for sample covariance operators},
  author={Koltchinskii, Vladimir and Lounici, Karim},
  journal={Bernoulli},
  pages={110--133},
  year={2017},
  publisher={JSTOR}
}

@article{weyl1912asymptotische,
  title={Das asymptotische Verteilungsgesetz der Eigenwerte linearer partieller Differentialgleichungen (mit einer Anwendung auf die Theorie der Hohlraumstrahlung)},
  author={Weyl, Hermann},
  journal={Mathematische Annalen},
  volume={71},
  number={4},
  pages={441--479},
  year={1912},
  publisher={Springer}
}


\newpage
\appendix
\setcounter{page}{1}
\setcounter{section}{0}
\renewcommand{\theequation}{A.\arabic{equation}}
\setcounter{equation}{0}

\begin{center}
    \textbf{\Large 
    Supplementary Material to ``Sharp Concentration Inequalities: Phase Transition and Mixing of Orlicz Tails with Variance''}

\medskip

Yinan Shen and Jinchi Lv
\end{center}

\smallskip

This Supplementary Material contains all the proofs of the main results and additional technical details.

\section{Proofs for Section~\ref{sec:univariate}}
\label{sec:proof:univariate}

This section presents the proofs of Lemmas~\ref{lem:momGeneratingFunct}--\ref{lem:momGeneratingFunct_lower}, Theorems~\ref{thm:conc_univariates}--\ref{cor:moment_univariate}, and Corollary~\ref{cor:norm_univariate} in Section~\ref{sec:univariate}.

\subsection{Proof of Lemma~\ref{lem:momGeneratingFunct}}

To prove this lemma, we will employ the Taylor expansion of function $\exp(\cdot)$ that involves a sum of series with $\lambda^k/\l[\frac{k}{\beta}\r]!$. Intuitively, for small enough $\lambda$, the higher-order terms are much smaller than the second-order term $\lambda^2$, whereas for large enough $\lambda$, the higher-order terms dominate the summation. To figure out the sum of the series, we will consider the power of $\lambda$ and the factorial that appears in the denominator $\l(k,\l[\frac{k}{\beta}\r]!\r)$, which can be roughly approximated by $\l( \beta\cdot\l[\frac{k}{\beta}\r],\l[\frac{k}{\beta}\r]! \r)$, but rigorous and delicate analyses are needed to derive the desired bounds. We also emphasize that the bound of the moment generating function varies across scenarios of $\alpha\geq 2, 1 < \beta \leq 2$ and $1 <  \alpha < 2,\beta > 2$. As such, we will analyze the moment generating function over these two scenarios separately.


In view of Lemma~\ref{lem:moment_bound} in Section \ref{new.SecD}, we have that 
$$\|X\|_k\leq C\|X\|_{\Psi_\alpha}k^{\frac{1}{\alpha}}$$ for some constant $C>0$. Then it holds that 
    \begin{align}
        \EE\l\{\exp(\lambda X)\r\}=\EE\l\{1+\sum_{k=1}^{\infty}\frac{1}{k!}\lambda^k X^k\r\}\leq 1+\sum_{k=2}^{\infty}\frac{k^{\frac{k}{\alpha}}}{k!}\lambda^k(C\|X\|_{\Psi_\alpha})^{k}.
        \label{eq1:lem:momGeneratingFunct}
    \end{align}
Recall that given $\alpha>1$, its conjugate $\beta>1$ satisfies $\frac{1}{\alpha}+\frac{1}{\beta}=1$. An application of Stirling's formula gives that for all $k=1,2,\ldots$,
\begin{align*}
    \frac{k^{\frac{k}{\alpha}}}{k!}\leq C\frac{c^{k}}{k^{\frac{k}{\beta}+\frac{1}{2}}},
\end{align*}
where $c,C>0$ are some universal constants. We can deduce that 
\begin{align*}
    \frac{k^{\frac{k}{\alpha}}}{k!}\lambda^k(C\|X\|_{\Psi_\alpha})^{k}&\leq C\frac{1}{ k^{\frac{k}{\beta}+\frac{1}{2}}}\lambda^k(C\|X\|_{\Psi_\alpha})^{k}\leq C\frac{1}{ \l(\frac{k}{\beta}\r)^{\frac{k}{\beta}+\frac{1}{2}}}\frac{1}{\beta^{\frac{k}{\beta}+\frac{1}{2}}}\lambda^k(C\|X\|_{\Psi_\alpha})^{k}\\
    &\leq C\frac{1}{ \l[\frac{k}{\beta}\r]^{\l[\frac{k}{\beta}\r]+\frac{1}{2}}}\lambda^k(C\|X\|_{\Psi_\alpha})^{k}\leq C\frac{1}{ \l[\frac{k}{\beta}\r]!}\lambda^k(C\|X\|_{\Psi_\alpha})^{k},
\end{align*}
    where $[x]$ represents the largest integer that is no bigger than $x$. Consequently, 
    \eqref{eq1:lem:momGeneratingFunct} above can be further bounded as 
    \begin{align}
        \EE\l\{\exp(\lambda X)\r\}\leq 1+\sum_{k=2}^{\infty} \frac{1}{\l[\frac{k}{\beta}\r]!}\lambda^kC^k\|X\|_{\Psi_{\alpha}}^{k}.
        \label{eq2:lem:momGeneratingFunct}
    \end{align}
    We will consider the two scenarios of $\alpha\geq 2, 1 < \beta \leq 2$ and $1 <  \alpha < 2,\beta > 2$ seperately. 

    \textit{Case 1: $\alpha\geq 2, 1 < \beta \leq 2$}. 
    In this case, the sequence $\l\{\l[\frac{2}{\beta}\r], \l[\frac{3}{\beta}\r],\ldots\r\}$ contains all positive integers and each integer appears in the sequence at most $2$ times, which is due to $\l[\frac{k}{\beta}\r]<\l[\frac{k+2}{\beta}\r]$. For even integers, it holds that $\l[\frac{2k}{\beta}\r]\geq k $, and for odd integers, it holds that $\l[\frac{2k+1}{\beta}\r]\geq k$. As a result, 
    \eqref{eq2:lem:momGeneratingFunct} above can be further bounded as 
    \begin{align}
        \EE\l\{\exp(\lambda X)\r\}\leq 1+\frac{1}{1!}C^2\lambda^2\|X\|_{\Psi_{\alpha}}^2+\frac{1}{1!}C^3\lambda^3\|X\|_{\Psi_{\alpha}}^3+\frac{1}{2!}C^4\lambda^4\|X\|_{\Psi_{\alpha}}^4+\cdots.
        \label{eq3:lem:momGeneratingFunct}
    \end{align}
    We will bound 
    \eqref{eq3:lem:momGeneratingFunct} above from two perspectives. First, note that 
    \begin{align}
        \EE\l\{\exp(\lambda X)\r\}&\leq 1+\sum_{k=1}^{\infty}\frac{1}{k!} (C\lambda\|X\|_{\Psi_{\alpha}})^{2k}(1+C\lambda\|X\|_{\Psi_{\alpha}}).
        \label{eq4:lem:momGeneratingFunct}
    \end{align}
    We then extract the common term $1+C\lambda\|X\|_{\Psi_{\alpha}}$ and can deduce that 
    \begin{align*}
        \EE\l\{\exp(\lambda X)\r\}&\leq 1+(1+C\lambda\|X\|_{\Psi_{\alpha}}) \l(\exp(C^2\lambda^2\|X\|_{\Psi_{\alpha}}^2)-1 \r)\\
        &\leq (1+C\lambda\|X\|_{\Psi_{\alpha}})\exp(C^2\lambda^2\|X\|_{\Psi_{\alpha}}^2 ),
    \end{align*}
    where the last step above is due to $1+C\lambda\|X\|_{\Psi_{\alpha}}\geq 1$. Hence, it follows from $1+x\leq \exp(x)$ that  $$\EE\l\{\exp(\lambda X)\r\}\leq \exp(C^2\lambda^2\|X\|_{\Psi_{\alpha}}^2 + C\lambda\|X\|_{\Psi_{\alpha}}). $$
    
    On the other hand, in light of  $1+C\lambda\|X\|_{\Psi_{\alpha}}\geq 1$, 
    \eqref{eq4:lem:momGeneratingFunct} above can be further bounded as 
    \begin{align*}
        \EE\l\{\exp(\lambda X)\r\}&\leq 1+\sum_{k=1}^{\infty}\frac{1}{k!} (C\lambda\|X\|_{\Psi_{\alpha}})^{2k}(1+C\lambda\|X\|_{\Psi_{\alpha}})^{k}\\
        &\leq \exp(C^2\lambda^2\|X\|_{\Psi_{\alpha}}^2+C^3\lambda^3\|X\|_{\Psi_{\alpha}}^3 ).
    \end{align*}
    Combining the above two results yields that 
    \begin{align*}
        & \EE\l\{\exp(\lambda X)\r\}\\
        &\leq \exp\bigg(\min\bigg\{C^2\lambda^2\|X\|_{\Psi_{\alpha}}^2+C^3\lambda^3\|X\|_{\Psi_{\alpha}}^3, C^2\lambda^2\|X\|_{\Psi_{\alpha}}^2 + C\lambda\|X\|_{\Psi_{\alpha}}\bigg\}\bigg),
    \end{align*}
    which is in fact equivalent to 
    \begin{align*}
        \EE\l\{\exp(\lambda X)\r\}\leq\exp(C\lambda^2\|X\|_{\Psi_{\alpha}}^2).
    \end{align*}
    To see why, when $C^2\lambda^2\|X\|_{\Psi_{\alpha}}^2+C^3\lambda^3\|X\|_{\Psi_{\alpha}}^3\leq  C^2\lambda^2\|X\|_{\Psi_{\alpha}}^2 + C\lambda\|X\|_{\Psi_{\alpha}}$, it holds that $$C\lambda\|X\|_{\Psi_\alpha}\leq 1,$$ under which we have that $$C^2\lambda^2\|X\|_{\Psi_{\alpha}}^2+C^3\lambda^3\|X\|_{\Psi_{\alpha}}^3 \leq 2C^2\lambda^2\|X\|_{\Psi_\alpha}^2.$$ Similar arguments apply when $C^2\lambda^2\|X\|_{\Psi_{\alpha}}^2+C^3\lambda^3\|X\|_{\Psi_{\alpha}}^3\geq  C^2\lambda^2\|X\|_{\Psi_{\alpha}}^2 + C\lambda\|X\|_{\Psi_{\alpha}}$.

    We then proceed to provide another bound of the moment generating function. Observe that 
    \eqref{eq2:lem:momGeneratingFunct} can be rewritten as 
    \begin{align}
        \EE\l\{\exp(\lambda X)\r\}\leq 1+\sum_{k=2}^{\infty} \frac{1}{\l[\frac{k}{\beta}\r]!}\l(\lambda C\|X\|_{\Psi_{\alpha}}\r)^{\beta\cdot \l[\frac{k}{\beta}\r]+\l(k-\beta\cdot \l[\frac{k}{\beta}\r]\r)},
        \label{eq5:lem:momGeneratingFunct}
    \end{align}
where $k-\beta\cdot\l[\frac{k}{\beta}\r]\in[0,\beta)$. When $\lambda C\|X\|_{\Psi_{\alpha}}\leq 1$, we have that $$\l(\lambda C\|X\|_{\Psi_{\alpha}}\r)^{\beta\cdot \l[\frac{k}{\beta}\r]+\l(k-\beta\cdot \l[\frac{k}{\beta}\r]\r)}\leq \l(\lambda C\|X\|_{\Psi_{\alpha}}\r)^{\beta\cdot \l[\frac{k}{\beta}\r]} ,$$ 
which shows that 
\eqref{eq5:lem:momGeneratingFunct} can be bounded as 
\begin{align*}
    \EE\l\{\exp(\lambda X)\r\}&\leq 1+\sum_{k=2}^{\infty} \frac{1}{\l[\frac{k}{\beta}\r]!}\l(\lambda C\|X\|_{\Psi_{\alpha}}\r)^{\beta\cdot \l[\frac{k}{\beta}\r]}\leq 1+\sum_{k=1}^{\infty} \frac{2}{k!}\l(\lambda C\|X\|_{\Psi_{\alpha}}\r)^{\beta\cdot k}\\
    &\leq\exp(C\lambda^\beta \|X\|_{\Psi_{\alpha}}^{\beta}). 
\end{align*}
When $\lambda C\|X\|_{\Psi_{\alpha}}\geq 1$, we have that $$\l(\lambda C\|X\|_{\Psi_{\alpha}}\r)^{\beta\cdot \l[\frac{k}{\beta}\r]+\l(k-\beta\cdot \l[\frac{k}{\beta}\r]\r)}\leq \l(\lambda C\|X\|_{\Psi_{\alpha}}\r)^{\beta\cdot \l[\frac{k}{\beta}\r]+\beta} ,$$ 
which shows that 
\eqref{eq5:lem:momGeneratingFunct} can be bounded as 
\begin{align*}
    \EE\l\{\exp(\lambda X)\r\}&\leq 1+\sum_{k=2}^{\infty} \frac{1}{\l[\frac{k}{\beta}\r]!}\l(\lambda C\|X\|_{\Psi_{\alpha}}\r)^{\beta\cdot \l[\frac{k}{\beta}\r]+\beta}\leq 1+\sum_{k=1}^{\infty} \frac{2}{k!}\l(\lambda C\|X\|_{\Psi_{\alpha}}\r)^{\beta\cdot (k+1)}\\
    &\leq 1+C\lambda^\beta\|X\|_{\Psi_{\alpha}}^\beta\l(\exp(C\lambda^\beta \|X\|_{\Psi_{\alpha}}^{\beta})-1\r)\leq C\lambda^\beta\|X\|_{\Psi_{\alpha}}^\beta\exp(C\lambda^\beta \|X\|_{\Psi_{\alpha}}^{\beta})\\
    &\leq \exp(C\lambda^\beta\|X\|_{\Psi_{\alpha}}^\beta)\cdot\exp(C\lambda^\beta \|X\|_{\Psi_{\alpha}}^{\beta})\\
    &= \exp(C_1\lambda^\beta \|X\|_{\Psi_{\alpha}}^{\beta}),
\end{align*}
where $C_1$ is another positive constant. 
Thus, combining the above results, we can obtain that when $\alpha\geq 2$, it holds that 
    \begin{align*}
        \EE\l\{\exp(\lambda X)\r\}\leq \exp(C\lambda^\beta \|X\|_{\Psi_{\alpha}}^{\beta}),\quad \EE\l\{\exp(\lambda X)\r\}\leq \exp(C\lambda^2 \|X\|_{\Psi_{\alpha}}^{2})
    \end{align*}
    for all $\lambda\geq 0$.

    \textit{Case 2: $1 <  \alpha < 2,\beta > 2$}. 
    In this case, 
    \eqref{eq2:lem:momGeneratingFunct} can be further written as
    \begin{equation}
        \begin{split}
             &\EE\l\{\exp(\lambda X)\r\}\leq 1+\frac{1}{0!}\lambda^2C^2\|X\|_{\Psi_{\alpha}}^2+\cdots+\frac{1}{0!}\lambda^{[\beta]}C^{[\beta]}\|X\|_{\Psi_{\alpha}}^{[\beta]}\\
        &\quad+\frac{1}{1!}\lambda^{[\beta]+1}C^{[\beta]+1}\|X\|_{\Psi_{\alpha}}^{[\beta]+1}+\cdots+\frac{1}{1!}\lambda^{2[\beta]+1}C^{2[\beta]+1}\|X\|_{\Psi_{\alpha}}^{2[\beta]+1}\\
        &\quad+\frac{1}{2!}\lambda^{2[\beta]+2}C^{2[\beta]+2}\|X\|_{\Psi_{\alpha}}^{2[\beta]+2}+\cdots+\frac{1}{2!}\lambda^{3[\beta]+2}C^{3[\beta]+2}\|X\|_{\Psi_{\alpha}}^{3[\beta]+2}\\
        &\quad+\cdots,
        \end{split}
        \label{eq6:lem:momGeneratingFunct}
    \end{equation}
where $0!=1!=1$ by convention. Let us consider different ranges of $\lambda C\|X\|_{\Psi_{\alpha}}$. When $\lambda C\|X\|_{\Psi_{\alpha}}<\tau$ with $\tau>0$ some given number, we can incorporate terms of 
\eqref{eq6:lem:momGeneratingFunct} by the common denominator and deduce that 
\begin{align*}
    &\EE\l\{\exp(\lambda X)\r\}\\
    &\leq 1+\frac{1}{0!}\lambda^2C^2\|X\|_{\Psi_{\alpha}}^2\cdot\l(1+\cdots+\tau^{[\beta]-2}\r)+\frac{1}{1!}\lambda^{[\beta]+1}C^{[\beta]+1}\|X\|_{\Psi_{\alpha}}^{[\beta]+1}\cdot\l(1+\cdots+\tau^{[\beta]}\r)\\
    &\quad+\frac{1}{2!}\lambda^{2[\beta]+2}C^{2[\beta]+2}\|X\|_{\Psi_{\alpha}}^{2[\beta]+2}\cdot\l(1+\cdots+\tau^{[\beta]} \r)+\cdots.
\end{align*}
Specifically, if $\lambda C\|X\|_{\Psi_{\alpha}}\leq\tau<1$, 
\eqref{eq6:lem:momGeneratingFunct} above can be bounded as
\begin{align*}
    \EE\l\{\exp(\lambda X)\r\}&\leq 1+\frac{1}{1-\tau}\lambda^2C^2\|X\|_{\Psi_{\alpha}}^2+\frac{1}{1-\tau}\frac{1}{2!}\lambda^{2[\beta]+2}C^{2[\beta]+2}\|X\|_{\Psi_{\alpha}}^{2[\beta]+2}+\cdots\\
    &\leq\exp(C\lambda^2\|X\|_{\Psi_{\alpha}}^2/(1-\tau)).
\end{align*}
If more generally $\lambda C\|X\|_{\Psi_{\alpha}}\leq 1 $, 
\eqref{eq6:lem:momGeneratingFunct} can be bounded as 
\begin{align*}
    \EE\l\{\exp(\lambda X)\r\}&\leq 1+(2[\beta]-1)\lambda^2C^2\|X\|_{\Psi_{\alpha}}^2+([\beta]+1)\frac{1}{2!}\lambda^{2[\beta]+2}C^{2[\beta]+2}\|X\|_{\Psi_{\alpha}}^{2[\beta]+2}+\cdots\\
    &\leq\exp(C\beta\lambda^2\|X\|_{\Psi_{\alpha}}^2)
\end{align*}
accordingly.

When $\lambda C\|X\|_{\Psi_{\alpha}}\geq\tau$ with $\tau\in(0,1)$ some given number, 
\eqref{eq6:lem:momGeneratingFunct} can be upper bounded as 
\begin{align*}
    \EE\l\{\exp(\lambda X)\r\}&\leq 1+\frac{1}{\tau^{[\beta]+1}}\frac{1}{1-\tau}\frac{1}{0!}\lambda^\beta C^{\beta}\|X\|_{\Psi_{\alpha}}^{\beta}+\frac{1}{\tau^{[\beta]+1}}\frac{1}{1-\tau}\frac{1}{1!}\lambda^{2\beta}C^{2\beta}\|X\|_{\Psi_{\alpha}}^{2\beta}+\cdots\\
    &\leq 1+\frac{1}{\tau^{[\beta]+1}}\frac{1}{1-\tau}\lambda^\beta C^{\beta}\|X\|_{\Psi_{\alpha}}^{\beta}\cdot\exp(C^{\beta}\lambda^{\beta}\|X\|_{\Psi_{\alpha}}^{\beta})\\
    &\leq 1+\l(\exp\l( \frac{1}{\tau^{[\beta]+1}}\frac{1}{1-\tau}\lambda^\beta C^{\beta}\|X\|_{\Psi_{\alpha}}^{\beta}\r) -1\r)\cdot \exp(C^{\beta}\lambda^{\beta}\|X\|_{\Psi_{\alpha}}^{\beta}).
\end{align*}
Then by resorting to $\exp(C^{\beta}\lambda^{\beta}\|X\|_{\Psi_{\alpha}}^{\beta})\geq 1$ and $1+x\leq \exp(x)$, it holds that $$\EE\l\{\exp(\lambda X)\r\}\leq \exp\l(C\frac{1}{\tau^{[\beta]+1}}\frac{1}{1-\tau}\lambda^{\beta}\|X\|_{\Psi_{\alpha}}^{\beta}\r).$$
If further $\lambda C\|X\|_{\Psi_{\alpha}}\geq 1$, 
it follows from \eqref{eq6:lem:momGeneratingFunct} that 
\begin{align*}
    \EE\l\{\exp(\lambda X)\r\}&\leq 1+\l([\beta]-1\r)\frac{1}{0!} \lambda^\beta C^{\beta}\|X\|_{\Psi_{\alpha}}^{\beta}+([\beta]+1)\frac{1}{1!}\lambda^{2\beta}C^{2\beta}\|X\|_{\Psi_{\alpha}}^{2\beta}+\cdots.
\end{align*}
We then extract the common term $\lambda^\beta C^{\beta}\|X\|_{\Psi_{\alpha}}^{\beta} $ and can show that 
\begin{align*}
    \EE\l\{\exp(\lambda X)\r\}&\leq 1+([\beta]+1)\lambda^\beta C^{\beta}\|X\|_{\Psi_{\alpha}}^{\beta}\cdot\exp\l(\lambda^\beta C^{\beta}\|X\|_{\Psi_{\alpha}}^\beta \r)\\
    &\leq \l(1+([\beta]+1)\lambda^\beta C^{\beta}\|X\|_{\Psi_{\alpha}}^{\beta}\r)\cdot\exp\l(\lambda^\beta C^{\beta}\|X\|_{\Psi_{\alpha}}^\beta \r)\\
    &\leq \exp\l(([\beta]+1)\lambda^\beta C^{\beta}\|X\|_{\Psi_{\alpha}}^{\beta} \r)\cdot\exp\l(\lambda^\beta C^{\beta}\|X\|_{\Psi_{\alpha}}^\beta \r)\\
    &\leq \exp\l(([\beta]+1)\lambda^\beta C_1^{\beta}\|X\|_{\Psi_{\alpha}}^{\beta} \r),
\end{align*}
    where the last two steps above have used the facts that $1+x\leq \exp(x)$ and $ \exp\l(x \r)\geq 1$ for $x\geq 0$. 
    
    Therefore, in view of $[\beta]+1\leq 2\beta$, we can obtain that when $\alpha\in(1,2)$, it holds that 
    \begin{align*}
        &\EE\l\{\exp(\lambda X)\r\}\leq\exp(C_1\beta\lambda^2\|X\|_{\Psi_{\alpha}}^2) \ \ \text{ if } \lambda\leq\frac{1}{C\|X\|_{\Psi_{\alpha}}},\\
        &\EE\l\{\exp(\lambda X)\r\}\leq\exp(C_2^\beta\beta\lambda^\beta\|X\|_{\Psi_{\alpha}}^\beta) \ \ \text{ if } \lambda\geq\frac{1}{C\|X\|_{\Psi_{\alpha}}},
    \end{align*}
    where $C, C_1, C_2 > 0$ are some constants. Moreover, we have that for any $\tau\in(0,1)$, 
    \begin{align*}
        &\EE\l\{\exp(\lambda X)\r\}\leq \exp\l(C_3\frac{1}{1-\tau}\lambda^2\|X\|_{\Psi_\alpha}^2\r) \ \ \text{ if } \lambda\leq \tau/(C\|X\|_{\Psi_\alpha}),\\
        &\EE\l\{\exp(\lambda X)\r\}\leq \exp\l(C_4^\beta\frac{\tau^{-[\beta]-1}}{1-\tau} \lambda^\beta \|X\|_{\Psi_\alpha}^{\beta}\r) \ \ \text{ if } \lambda\geq \tau/(C\|X\|_{\Psi_\alpha}),
    \end{align*}
where $C, C_3, C_4 > 0$ are some constants. This completes the proof of Lemma~\ref{lem:momGeneratingFunct}.

\subsection{Proof of Lemma~\ref{lem:momGeneratingFunct_lower}}

It is well-known that sub-Weibull random variables satisfy that $\EE|X|^k\leq C^kk^{\frac{k}{\alpha}}\|X\|_{\Psi_\alpha}^k$. However, a more delicate bound can be derived. Our technical lemma establishes that $\EE|X|^3\leq C\sigma_X^2\|X\|_{\Psi_\alpha}\log^{\frac{1}{\alpha}}(\frac{2\|X\|_{\Psi_\alpha}}{\sigma_X})$ (see Lemma~\ref{teclem:cube} in Section \ref{new.SecD}), and see Examples~\ref{exp:1} and \ref{exp:2} for the bound with general $k\geq 2$, which is key to the proof of the current lemma.


Let us first expand the exponential function
    \begin{align*}
        \EE\l\{\exp(\lambda X)\r\}&=\EE\l\{1+\sum_{k=1}^{\infty}\frac{1}{k!}\lambda^k X^k\r\}= 1+\frac{1}{2}\lambda^2\EE\l\{X^2\r\}+\sum_{k=3}^{\infty}\frac{1}{k!}\lambda^k\EE\l\{X^k\r\}\\
        &\geq 1+\frac{1}{2}\lambda^2\EE\l\{X^2\r\}-\sum_{k=3}^{\infty}\frac{1}{k!}\lambda^k\EE\l|X\r|^k.
    \end{align*}
    We then examine the high-order term in the expression above and can write it as 
    \begin{align*}
        \sum_{k=3}^{\infty}\frac{1}{k!}\lambda^k\EE\l|X\r|^k&=\EE \lambda^3|X|^3\frac{\exp(\lambda|X|)-1-\lambda|X|-\frac{1}{2}\lambda^2X^2}{\lambda^3|X|^3}\cdot\II\l\{|X|\leq \tau\r\}\\
        &~~~~+\EE \lambda^3|X|^3\frac{\exp(\lambda|X|)-1-\lambda|X|-\frac{1}{2}\lambda^2X^2}{\lambda^3|X|^3}\cdot\II\l\{|X|>\tau\r\}.
    \end{align*}
    For $\lambda\leq 1/\tau$, it holds that 
    \begin{align*}
        & \EE \lambda^3|X|^3\frac{\exp(\lambda|X|)-1-\lambda|X|-\frac{1}{2}\lambda^2X^2}{\lambda^3|X|^3}\cdot\II\l\{|X| \leq \tau\r\} \\
        & \leq \EE \lambda^3|X|^3\frac{e-1}{1}\leq c\EE\lambda^3|X|^3.
    \end{align*}
    Hence, it follows from Lemma~\ref{teclem:cube} that 
   \begin{align*}
       &\EE \lambda^3|X|^3\frac{\exp(\lambda|X|)-1-\lambda|X|-\frac{1}{2}\lambda^2X^2}{\lambda^3|X|^3}\cdot\II\l\{|X|\leq \tau\r\}\\
       &\leq C\lambda^3\sigma_X^2\|X\|_{\Psi_\alpha}\l(\log\l(\frac{2\|X\|_{\Psi_\alpha}}{\sigma_X}\r)\r)^{\frac{1}{\alpha}}.
   \end{align*}
   
   On the other hand, the term with $|X|\geq \tau$ can be bounded as 
   \begin{align*}
       & \EE \lambda^3|X|^3\frac{\exp(\lambda|X|)-1-\lambda|X|-\frac{1}{2}\lambda^2X^2}{\lambda^3|X|^3}\cdot\II\l\{|X|\geq \tau\r\}\\
       & \leq \lambda^3 \EE |X|^3 \frac{\exp(\frac{|X|}{C\|X\|_{\Psi_\alpha}})-1-(\frac{|X|}{C\|X\|_{\Psi_\alpha}})-\frac{1}{2}(\frac{|X|}{C\|X\|_{\Psi_\alpha}})^2}{(\frac{|X|}{C\|X\|_{\Psi_\alpha}})^3}\II\{|X|\geq \tau\}\\
       & \leq C\lambda^3\|X\|_{\Psi_\alpha}^3\EE\exp\l(\frac{|X|}{C\|X\|_{\Psi_\alpha}}\r)\cdot \II\{|X|\geq \tau\},
   \end{align*}
   where the first step above has utilized the fact that the fraction is monotone increasing with respect to $\lambda$, and the routine is partially borrowed from \cite{koltchinskii2011neumann}. Then with the choice of $\tau=C\|X\|_{\Psi_\alpha}\l(\log\l(\frac{2\|X\|_{\Psi_\alpha}}{\sigma_X}\r)\r)^{\frac{1}{\alpha}}$ as in \cite{koltchinskii2011neumann}, we can deduce that 
   \begin{align*}
       \EE\exp\l(\frac{|X|}{C\|X\|_{\Psi_\alpha}}\r)\cdot \II\{|X|\geq \tau\}&\leq \sqrt{\EE\exp\l(\frac{2|X|}{C\|X\|_{\Psi_\alpha}}\r)}\cdot \sqrt{\PP\{|X|\geq \tau\}}\\
       &\leq C\exp\l(-c\l(\frac{\tau}{\|X\|_{\Psi_\alpha}}\r)^\alpha\r)\leq \l(\frac{\sigma_X}{\|X\|_{\Psi_\alpha}}\r)^2.
   \end{align*}
   
   Combining the above results leads to 
   \begin{align*}
       \sum_{k=3}^{\infty}\frac{1}{k!}\lambda^k\EE\l|X\r|^k\leq C\lambda^3\sigma_X^2\|X\|_{\Psi_\alpha}\l(\log\l(\frac{2\|X\|_{\Psi_\alpha}}{\sigma_X}\r)\r)^{\frac{1}{\alpha}}.
   \end{align*}
   Consequently, we can obtain the following lower bound for the moment generating function
    \begin{align*}
        \EE\l\{\exp(\lambda X)\r\}\geq 1+\frac{1}{2}\lambda^2\sigma_X^2-C\lambda^3\sigma_X^2\|X\|_{\Psi_\alpha}\l(\log\l(\frac{2\|X\|_{\Psi_\alpha}}{\sigma_X}\r)\r)^{\frac{1}{\alpha}}.
    \end{align*}
    For $\lambda\leq \frac{1}{\|X\|_{\Psi_\alpha}}\l(\log\l(\frac{2\|X\|_{\Psi_\alpha}}{\sigma_X}\r)\r)^{-\frac{1}{\alpha}}$, we can show that 
    \begin{align*}
        \EE\l\{\exp(\lambda X)\r\}\geq 1+\frac{1}{4}\lambda^2\sigma_X^2\geq \exp\l(\lambda^2\sigma_X^2/8\r),
    \end{align*}
    where the last step above is due to $1+x\geq \exp(x/2)$ for $x\in[0,1]$.
This concludes the proof of Lemma~\ref{lem:momGeneratingFunct_lower}.

\subsection{Proof of Theorem~\ref{thm:conc_univariates}}
The proof of Theorem~\ref{thm:conc_univariates} is rooted on Lemma~\ref{lem:momGeneratingFunct}. 
By the Markov inequality, for any $\lambda>0$ we have that 
\begin{align*}
    \PP\l(\sum_{i=1}^n a_iX_i\geq t\r)\leq \frac{1}{\exp(\lambda t)}\EE\l\{ \exp\l(\lambda\sum_{i=1}^n a_iX_i \r)\r\}=\frac{1}{\exp(\lambda t)}\prod_{i=1}^n \EE\l\{\lambda a_iX_i\r\}.
\end{align*}
    For $\alpha\geq 2$, an application of Lemma~\ref{lem:momGeneratingFunct} gives that
    \begin{align*}
        \EE\l\{\lambda a_iX_i\r\}\leq \exp\l(C_1\min\l\{\lambda^2 a_i^2\|X_i\|_{\Psi_{\alpha}}^{2},\, \lambda^{\beta} |a_i|^{\beta}\|X_i\|_{\Psi_{\alpha}}^{\beta}\r\}\r),
    \end{align*}
    which yields that 
    \begin{align*}
        \PP\l(\sum_{i=1}^n a_iX_i\geq t\r)\leq \exp\l(C_1\min\l\{\lambda^2 \sum_{i=1}^n a_i^2\|X_i\|_{\Psi_{\alpha}}^{2},\, \lambda^{\beta} \sum_{i=1}^n|a_i|^{\beta}\|X_i\|_{\Psi_{\alpha}}^{\beta}\r\}-\lambda t\r).
    \end{align*}
    Inserting $\lambda=\max\l\{\frac{t}{2C_1\sum_{i=1}^n a_i^2\|X_i\|_{\Psi_{\alpha}}^2}, \l(\frac{t}{C_1\beta\sum_{i=1}^n |a_i|^{\beta} \|X\|_{\Psi_{\alpha}}^{\beta}}\r)^{^{\frac{1}{\beta-1}}}\r\}$ into the above expression, it holds that 
    \begin{align*}
        \PP\l(\sum_{i=1}^n a_iX_i\geq t\r)\leq \exp\l(-C_2\max\l\{\frac{t^2}{\sum_{i=1}^n a_i^2\|X_i\|_{\Psi_{\alpha}}^2 },\frac{t^{\alpha}}{\l(\sum_{i=1}^n |a_i|^{\beta} \|X_i\|_{\Psi_{\alpha}}^{\beta} \r)^{\frac{1}{\beta-1}}} \r\}\r).
    \end{align*}
    
    For $\alpha\in(1,2)$, by invoking Lemma~\ref{lem:momGeneratingFunct} we have that 
    \begin{align*}
        \EE\l\{\lambda a_i X_i\r\}\leq \exp\l(C_3 \lambda^2 a_i^2\|X_i\|_{\Psi_{\alpha}}^{2}+C_3^\beta \lambda^{\beta} 2^{\beta}|a_i|^{\beta}\|X_i\|_{\Psi_{\alpha}}^{\beta}\r).
    \end{align*}
    Then it follows that 
    \begin{align*}
         \PP\l(\sum_{i=1}^n a_iX_i\geq t\r)\leq \exp\l(C_3 \lambda^2 \sum_{i=1}^n a_i^2\|X_i\|_{\Psi_{\alpha}}^{2}+C_3^\beta2^{\beta} \lambda^{\beta} \sum_{i=1}^n|a_i|^{\beta}\|X_i\|_{\Psi_{\alpha}}^{\beta}-\lambda t\r).
    \end{align*}
    Inserting $\lambda=\min\l\{\frac{t}{4C_3\sum_{i=1}^{n} a_i^2\|X\|_{\Psi_{\alpha}}^2},\l(\frac{t}{4C_3^{\beta
    }2^{\beta}\sum_{i=1}^n |a_i|^{\beta}\|X\|_{\Psi_{\alpha}}^{\beta} }\r)^{\frac{1}{\beta-1}} \r\}$ into the above expression yields that 
    \begin{align}
        \PP\l(\sum_{i=1}^n a_iX_i\geq t\r)\label{eq1:proof:thm:conc_univariates}\leq \exp\l(-C_5\min\l\{\frac{t^2}{\sum_{i=1}^n a_i\|X_i\|_{\Psi_{\alpha}}^2 }, \frac{t^{\alpha}}{(\sum_{i=1}^n |a_i|^{\beta} \|X_i\|_{\Psi_{\alpha}}^{\beta})^{\frac{1}{\beta-1}}}\r\}\r),
    \end{align}
    where we have used $\beta\geq 2$ and $C_3^{\frac{\beta}{\beta-1}}\leq C_3^2$. 
    
    
    Finally, for $\alpha=1$, an application of  Lemma~\ref{lem:momGeneratingFunct} and setting $$\lambda=\min\l\{\frac{t}{C\sum_{i=1}^n a_i^2\|X\|_{\Psi_1}^2},\frac{1}{\max_{i=1}^n|a_i|\|X_i\|_{\Psi_1}}\r\}$$ give that  
    $$\PP\l(\sum_{i=1}^n a_iX_i\geq t\r)\leq \exp\l(-C_5\min\l\{\frac{t^2}{\sum_{i=1}^n a_i\|X_i\|_{\Psi_{1}}^2 }, \frac{t}{\max_{i=1}^n |a_i|\|X_i\|_{\Psi_1}}\r\}\r), $$ 
    which has also been proved by Theorem~2.9.1 of \cite{vershynin2018high}. Due to $$\lim_{\beta\to\infty}(\sum_{i=1}^n |a_i|^{\beta} \|X\|_{\Psi_{\alpha}}^{\beta})^{\frac{1}{\beta-1}}=\max_{i=1,\ldots,n} |a_i|\|X_i\|_{\Psi_\alpha}, $$ the concentration inequality for $\alpha=1$ is contained by 
    (\ref{eq1:proof:thm:conc_univariates}). This completes the proof of Theorem~\ref{thm:conc_univariates}. 

\subsection{Proof of Theorem~\ref{cor:moment_univariate}} 

Due to the phase transition over the regimes of $\alpha\geq 2$ and $\alpha\in[1,2]$, the following proof first establishes the bound for the case of $\alpha\geq1$ and then provides a bound that holds only for the case of $\alpha\geq 2$.

Let us first consider the case of $\alpha\geq1$. For any $p>0$, by the definition of integral it holds that 
    \begin{align}
        \EE\l\{ \l|\sum_{i=1}^n a_iX_i \r|^p\r\}= \int_{0}^{\infty} pt^{p-1}\cdot\PP\l(\l|\sum_{i=1}^n a_iX_i \r|\geq t \r)\; dt. \label{eq4:cor:moment_univariate}
    \end{align}
    An application of Theorem~\ref{thm:conc_univariates} leads to 
    \begin{align}\label{eq2:cor:moment_univariate}
        \EE\l\{ \l|\sum_{i=1}^n a_iX_i \r|^p\r\} & \leq 
        \int_{0}^{\tau} pt^{p-1}\cdot \exp\l(-\frac{Ct^2}{\sum_{i=1}^n a_i^2\|X_i\|_{\Psi_{\alpha}}^2}\r)\, dt\\
        &\quad+\int_{\tau}^{\infty} pt^{p-1}\cdot \exp\l(-\frac{Ct^\alpha}{(\sum_{i=1}^n |a_i|^\beta\|X_i\|_{\Psi_{\alpha}}^\beta)^{\frac{\alpha}{\beta}}}\r)\, dt, \nonumber
    \end{align}
    where quantity $\tau > 0$ satisfies the following equation
    \begin{align}
        \frac{\tau^2}{\sum_{i=1}^n a_i^2\|X_i\|_{\Psi_{\alpha}}^2}=\frac{\tau^\alpha}{(\sum_{i=1}^n |a_i|^\beta\|X_i\|_{\Psi_{\alpha}}^\beta)^{\frac{\alpha}{\beta}}}.
        \label{eq1:cor:moment_univariate}
    \end{align}
    It suffices to bound each term on the right-hand side of 
    \eqref{eq2:cor:moment_univariate} above. Specifically, the first term admits that 
    \begin{align*}
        \int_{0}^{\tau} pt^{p-1}\cdot \exp\l(-\frac{Ct^2}{\sum_{i=1}^n a_i^2\|X_i\|_{\Psi_\alpha}^2}\r)\, dt&\leq \int_{0}^{\infty} pt^{p-1}\cdot \exp\l(-\frac{Ct^2}{\sum_{i=1}^n a_i^2\|X_i\|_{\Psi_\alpha}^2}\r)\, dt\\
        &\leq C_1^{p/2}p^{p/2}\l(\sum_{i=1}^n a_i^2\|X_i\|_{\Psi_\alpha}^2\r)^{p/2},
    \end{align*}
    where the last step above follows from standard integrals and Gamma function properties. 

    For the second term on the right-hand side of 
    \eqref{eq2:cor:moment_univariate}, it holds that 
    \begin{align*}
        & \int_{\tau}^{\infty} pt^{p-1}\cdot \exp\l(-\frac{Ct^\alpha}{(\sum_{i=1}^n |a_i|^\beta\|X_i\|_{\Psi_{\alpha}}^\beta)^{\frac{\alpha}{\beta}}}\r)\, dt\\
        & =\frac{p}{\alpha}\l(\sum_{i=1}^n |a_i|^{\beta}\|X_i\|_{\Psi_{\alpha}}^{\beta}\r)^{\frac{p}{\beta}} \int_{\frac{\tau^\alpha}{(\sum_{i=1}^n |a_i|^\beta\|X_i\|_{\Psi_{\alpha}}^\beta)^{\frac{\alpha}{\beta}}} }^{\infty} \;s^{\frac{p}{\alpha}-1}\exp(-Cs)\, ds.
    \end{align*}
    Then when $p\leq \alpha$, we have $\frac{p}{\alpha}-1\leq 0$ and thus it follows from Lemma~\ref{teclem:integral} in Section \ref{new.SecD} that 
    \begin{align*}
        & \int_{\tau_2}^{\infty} pt^{p-1}\cdot \exp\l(-\frac{Ct^\alpha}{(\sum_{i=1}^n a_i^\beta\|X_i\|_{\Psi_{\alpha}}^\beta)^{\frac{\alpha}{\beta}}}\r)\, dt
        \leq \frac{p}{\alpha}\l(\sum_{i=1}^n |a_i|^{\beta}\|X_i\|_{\Psi_{\alpha}}^{\beta}\r)^{\frac{p}{\beta}}\\
        & \quad\times
        \l(\frac{(\sum_{i=1}^{n} |a_i|^{\beta} \|X\|_{\Psi_{\alpha}}^{\beta})^{\frac{1}{\beta}}}{(\sum_{i=1}^n a_i^2\|X_{i}\|_{\Psi_{\alpha}}^2)^{\frac{1}{2}}}\r)^{\frac{2(p-\alpha)}{\alpha-2}}\exp\l(-C\l(\frac{(\sum_{i=1}^{n} |a_i|^{\beta} \|X\|_{\Psi_{\alpha}}^{\beta})^{\frac{1}{\beta}}}{(\sum_{i=1}^n a_i^2\|X_{i}\|_{\Psi_{\alpha}}^2)^{\frac{1}{2}}}\r)^{\frac{2\alpha}{\alpha-2}}\r)\\
        & \leq p\l(\sum_{i=1}^n |a_i|^{\beta}\|X_i\|_{\Psi_{\alpha}}^{\beta}\r)^{\frac{p}{\beta}}\exp\l(-C\l(\frac{(\sum_{i=1}^{n} |a_i|^{\beta} \|X\|_{\Psi_{\alpha}}^{\beta})^{\frac{1}{\beta}}}{(\sum_{i=1}^n a_i^2\|X_{i}\|_{\Psi_{\alpha}}^2)^{\frac{1}{2}}}\r)^{\frac{2\alpha}{\alpha-2}}\r),
    \end{align*}
    where the last step above has exploited the fact that $\|r\|_{q_1}\leq \|r\|_{q_2}$ if $q_1\geq q_2$ for any vector $r$. 
    
    On the other hand, when $p\geq \alpha$, we have $\frac{p}{\alpha}-1\geq 0$ and thus it follows from Lemma~\ref{teclem:integral} that
\begin{align*}
        & \int_{\tau}^{\infty} pt^{p-1}\cdot \exp\l(-\frac{Ct^\alpha}{(\sum_{i=1}^n a_i^\beta\|X_i\|_{\Psi_{\alpha}}^\beta)^{\frac{1}{\beta-1}}}\r)\, dt\\
        & \leq C_1^{p}p^{\frac{p}{\alpha}}\l(\sum_{i=1}^n a_i^\beta \|X_i\|_{\Psi_{\alpha}}^\beta\r)^{\frac{p}{\beta}}\exp\l(-C\l(\frac{(\sum_{i=1}^{n} |a_i|^{\beta} \|X\|_{\Psi_{\alpha}}^{\beta})^{\frac{1}{\beta}}}{(\sum_{i=1}^n a_i^2\|X_{i}\|_{\Psi_{\alpha}}^2)^{\frac{1}{2}}}\r)^{\frac{2\alpha}{\alpha-2}}\r).
    \end{align*}
    Hence, for all $p\geq 1$, we have the following bound for the second term on the right-hand side of 
    \eqref{eq2:cor:moment_univariate}
    \begin{align*}
        & \int_{\tau}^{\infty} pt^{p-1}\cdot \exp\l(-\frac{Ct^\alpha}{(\sum_{i=1}^n |a_i|^\beta\|X_i\|_{\Psi_{\alpha}}^\beta)^{\frac{\alpha}{\beta}}}\r)\, dt\\
        & \leq C_1^{p}p^{\frac{p}{\alpha}}\l(\sum_{i=1}^n a_i^\beta \|X_i\|_{\Psi_{\alpha}}^\beta\r)^{\frac{p}{\beta}}\exp\l(-C\l(\frac{(\sum_{i=1}^{n} |a_i|^{\beta} \|X\|_{\Psi_{\alpha}}^{\beta})^{\frac{1}{\beta}}}{(\sum_{i=1}^n a_i^2\|X_{i}\|_{\Psi_{\alpha}}^2)^{\frac{1}{2}}}\r)^{\frac{2\alpha}{\alpha-2}}\r).
    \end{align*}
    Combining the above results yields that 
    \begin{align*}
        & \EE\l\{ \l|\sum_{i=1}^n a_iX_i \r|^p\r\}\leq  C_1^p p^{\frac{p}{2}} \l(\sum_{i=1}^n a_i^2\|X_i\|_{\Psi_\alpha}^2 \r)^{\frac{p}{2}}\\
        &\quad +C_1^{p}p^{\frac{p}{\alpha}}\l(\sum_{i=1}^n a_i^\beta \|X_i\|_{\Psi_{\alpha}}^\beta\r)^{\frac{p}{\beta}}\exp\l(-C\l(\frac{(\sum_{i=1}^{n} |a_i|^{\beta} \|X\|_{\Psi_{\alpha}}^{\beta})^{\frac{1}{\beta}}}{(\sum_{i=1}^n a_i^2\|X_{i}\|_{\Psi_{\alpha}}^2)^{\frac{1}{2}}}\r)^{\frac{2\alpha}{\alpha-2}}\r).
    \end{align*}
    
    Further, for the case of $\alpha\geq 2$, by resorting to Theorem~\ref{thm:conc_univariates} we have that
    \begin{align*}
        & \PP\l(\l|\sum_{i=1}^n a_iX_i \r|\geq t\r)\\
        & \leq 2\min\l\{\exp\l(-C\frac{t^2}{\sum_{i=1}^n a_i\|X_i\|_{\Psi_\alpha}^2}\r),\exp\l(-C\frac{t^{\alpha}}{\l(\sum_{i=1}^n |a_i|^\beta \|X_i\|_{\Psi_\alpha}^{\beta}\r)} \r)\r\}.
    \end{align*}
    Therefore, by inserting the above expression into 
    \eqref{eq4:cor:moment_univariate} and standard integrals, we can obtain that 
    \begin{align*}
        \EE \l\{ \l|\sum_{i=1}^n a_iX_i \r|^p\r\}\leq C_1^p\min\l\{ p^{\frac{p}{2}}\l(\sum_{i=1}^n a_i^2\|X_i\|_{\Psi_\alpha}^2 \r)^{\frac{p}{2}},p^{\frac{p}{\alpha}}\l(\sum_{i=1}^n |a_i|^\beta\|X_i\|_{\Psi_\alpha}^\beta\r)^{\frac{p}{\beta}} \r\},
    \end{align*}
    which concludes the proof of Theorem~\ref{cor:moment_univariate}.

\subsection{Proof of Corollary~\ref{cor:norm_univariate}}

In what follows, we will derive the bound of the Orlicz norm based on Theorems~\ref{thm:conc_univariates} and \ref{cor:moment_univariate}. Again, we will analyze the cases of $\alpha\geq 2$ and $1\leq \alpha\leq 2$ separately.

\textit{Case 1: $\alpha\geq 2$}. 
In this case, according to Theorem~\ref{cor:moment_univariate}, it holds that for all $p\geq 1$, 
    \begin{align*}
        \l\| \sum_{i=1}^n a_iX_i\r\|_p\leq C_1 p^{\frac{1}{\alpha}}\cdot \l(\sum_{i=1}^n |a_i|^{\beta}\|X_i\|_{\Psi_{\alpha}}^{\beta}\r)^{\frac{1}{\beta}}.
    \end{align*}
    Denote by $K:= C_2\l(\sum_{i=1}^n |a_i|^{\beta}\|X_i\|_{\Psi_{\alpha}}^{\beta}\r)^{\frac{1}{\beta}}$ with $C_2>0$ some sufficiently large constant. It follows that 
    \begin{align*}
        \EE\l\{\exp\l( \l|\frac{\sum_{i=1}^n a_iX_i }{K}\r|^\alpha\r)\r\}&=1+\sum_{p=1}^{\infty}\frac{1}{p!}\frac{1}{K^{p}}\l\| \sum_{i=1}^n a_iX_i\r\|_{p\alpha}^{p\alpha}\\
        &\leq 1+\sum_{p=1}^{\infty}\frac{1}{p!}\frac{1}{K^{\alpha p}} C_1^{p\alpha}p^p \alpha^p \cdot \l(\sum_{i=1}^n |a_i|^{\beta}\|X_i\|_{\Psi_{\alpha}}^{\beta}\r)^{\frac{p\alpha}{\beta}}.
    \end{align*}
    Then by Stirling's formula and inserting the value of $K$, we can deduce that 
    \begin{align*}
        \EE\l\{\exp\l( \l|\frac{\sum_{i=1}^n a_iX_i }{K}\r|^\alpha\r)\r\}\leq 1+\sum_{p=1}^{\infty}\frac{1}{\sqrt{p}}\frac{C_3^p}{C_2^{\alpha p}} C_1^{p\alpha} \alpha^p .
    \end{align*}
    Hence, with $C_2\geq C_1+C_3$ and $C_2\geq \max_{\alpha\geq 2}\alpha^{\frac{1}{\alpha}}$, it holds that $$ \EE\l\{\exp\l( \l|\frac{\sum_{i=1}^n a_iX_i }{K}\r|^\alpha\r)\r\}\leq2.$$ This establishes that
    \begin{align*}
        \l\|\sum_{i=1}^n a_iX_i \r\|_{\Psi_{\alpha}}\leq C_2\l(\sum_{i=1}^n |a_i|^{\beta}\|X_i\|_{\Psi_{\alpha}}^{\beta}\r)^{\frac{1}{\beta}}.
    \end{align*}
    Moreover, an application of Theorem~\ref{cor:moment_univariate} gives that for all $p\geq 1$, 
    \begin{align*}
        \l\| \sum_{i=1}^n a_iX_i\r\|_p\leq C_1 p^{\frac{1}{2}}\cdot \l(\sum_{i=1}^n a_i^{2}\|X_i\|_{\Psi_{\alpha}}^{2}\r)^{\frac{1}{2}},
    \end{align*}
    which entails that $\l\|\sum_{i=1}^n a_iX_i \r\|_{\Psi_{2}}\leq \l(\sum_{i=1}^n a_i^{2}\|X_i\|_{\Psi_{\alpha}}^{2}\r)^{\frac{1}{2}}$ using the standard arguments as in \cite{vershynin2018high}.

    \textit{Case 2: $1\leq \alpha\leq 2$}. In this scenario, we have that $\beta>2$, $\sqrt{p}\leq p^{\frac{1}{\alpha}}$, and 
    $$\l(\sum_{i=1}^n |a_i|^{\beta}\|X_i\|_{\Psi_{\alpha}}^{\beta}\r)^{\frac{1}{\beta}}\leq \l(\sum_{i=1}^n a_i^{2}\|X_i\|_{\Psi_{\alpha}}^{2}\r)^{\frac{1}{2}}.$$ 
    By invoking Theorem~\ref{cor:moment_univariate}, it holds that 
    \begin{align*}
        \l\| \sum_{i=1}^n a_iX_i\r\|_p&\leq C_1^{\frac{1}{2}} p^{\frac{1}{2}}\cdot \l(\sum_{i=1}^n a_i^{2}\|X_i\|_{\Psi_{\alpha}}^{2}\r)^{\frac{1}{2}}+C_1^{\frac{1}{\alpha}} p^{\frac{1}{\alpha}}\cdot \l(\sum_{i=1}^n |a_i|^{\beta}\|X_i\|_{\Psi_{\alpha}}^{\beta}\r)^{\frac{1}{\beta}}\cdot\exp(-cn/p).
    \end{align*}
    Thus, by setting $K:=C_2 \l(\sum_{i=1}^n a_i^{2}\|X_i\|_{\Psi_{\alpha}}^{2}\r)^{\frac{1}{2}}$ with $C_2>0$ some sufficiently large constant, we can obtain that 
    \begin{align*}
        & \EE\l\{\exp\l( \l|\frac{\sum_{i=1}^n a_iX_i }{K}\r|^\alpha\r)\r\}=1+\sum_{p=1}^{\infty}\frac{1}{p!}\frac{1}{K^{\alpha p}}\l\| \sum_{i=1}^n a_iX_i\r\|_{p\alpha}^{p\alpha}
        \leq 1+\sum_{p=1}^{\infty}\frac{C_1^{p\alpha}}{p!}\frac{1}{K^{\alpha p}}  \\
        &\quad\times\l\{p^p \alpha^p \cdot \l(\sum_{i=1}^n |a_i|^{\beta}\|X_i\|_{\Psi_{\alpha}}^{\beta}\r)^{\frac{p\alpha}{\beta}}\cdot\exp(-n)+ \alpha^{\frac{\alpha p}{2}}p^{\frac{\alpha p}{2}}\cdot\l(\sum_{i=1}^n a_i^{2}\|X_i\|_{\Psi_{\alpha}}^{2}\r)^{\frac{\alpha p}{2}}\r\}\\
        & \leq 2.
    \end{align*}
This completes the proof of Corollary~\ref{cor:norm_univariate}.

\section{Proofs for Section~\ref{sec:univariate_var}}
\label{sec:proof:univariate_var}

This section contains the proofs of Examples~\ref{exp:1}--\ref{exp:2}, Lemma~\ref{lem:momGeneratingFunct_SigmaL}, and Theorem~\ref{thm:conc_univariate_SigmaL} in Section~\ref{sec:univariate_var}.

\subsection{Proof of Example~\ref{exp:1}}

The lemma below proves Example~\ref{exp:1}. We emphasize that the truncation technique used in the proof 
exploits the intuition from \cite{ahlswede2002strong}, \cite{recht2011simpler}, \cite{gross2010quantum}, \cite{gross2011recovering}, and \cite{koltchinskii2011neumann}.

\begin{lemma} \label{new.lem4}
    Assume that $X$ is a mean zero random variable and satisfies that $\|X\|_{\Psi_\alpha}<\infty$ 
    for some $\alpha\geq1$. Denote by $\sigma_X^2:=\operatorname{var}(X)$. Then it holds that for some universal constant $C > 0$ and all $k\geq 1$,
    \begin{align*}
        \EE|X|^k\leq C^{k}k^{\frac{k}{\alpha}}\cdot\sigma_X^2\|X\|_{\Psi_\alpha}^{k-2}\log^{\frac{k-2}{\alpha}}\l(\frac{\|X\|_{\Psi_\alpha}}{\sigma_X}\r).
    \end{align*}
\end{lemma}

\textit{Proof}. 
    Note that for any $\tau>0$, we have that 
    \begin{align*}
        \EE |X|^k=\EE |X|^k\cdot\II\{|X|>\tau\}+\EE |X|^k\cdot\II\{|X|\leq \tau\}.
    \end{align*}
    We will bound each of the terms on the right-hand side of the expression above. The first term admits that 
    \begin{align*}
        \EE |X|^k\cdot\II\{|X|>\tau\}&\leq \l(\EE|X|^{2k}\r)^{\frac{1}{2}}\l(\PP\l(|X|\geq \tau\r)\r)^{\frac{1}{2}}\\
        &\leq C^k k^{\frac{k}{\alpha}}\|X\|_{\Psi_\alpha}^k\cdot\exp\l(-c^\alpha\frac{\tau^\alpha}{\|X\|_{\Psi_\alpha}^\alpha}\r),
    \end{align*}
    where the last step above has utilized Lemma~\ref{lem:moment_bound} and the properties of sub-Weilbull random variables \citep{vershynin2018high}, and $C > 0$ is some universal constant. For the second term, it holds that 
    \begin{align*}
        \EE |X|^k\cdot\II\{|X|\leq \tau\}\leq \tau^{k-2}\EE X^2=\tau^{k-2}\sigma_X^2.
    \end{align*}
Hence, combining the above results leads to 
    \begin{align*}
        \EE |X|^k\leq C^k k^{\frac{k}{\alpha}}\|X\|_{\Psi_\alpha}^k\cdot\exp\l(-c^\alpha\frac{\tau^\alpha}{\|X\|_{\Psi_\alpha}^\alpha}\r)+\tau^{k-2}\sigma_X^2.
    \end{align*}
    Therefore, with the choice of  $\tau=C\|X\|_{\Psi_\alpha}\log^{\frac{1}{\alpha}}\l(\frac{2\|X\|_{\Psi_\alpha}}{\sigma}\r)$, we can obtain that 
   \begin{align*}
       \EE |X|^k\leq C^kk^{\frac{k}{\alpha}}\|X\|_{\Psi_\alpha}^{k-2}\sigma_X^2\log^{\frac{k-2}{\alpha}}\l(\frac{2\|X\|_{\Psi_\alpha}}{\sigma}\r).
   \end{align*}
This concludes the proof of Lemma \ref{new.lem4}.

\subsection{Proof of Example~\ref{exp:2}}

The lemma below proves Example~\ref{exp:2}. To the best of our knowledge, we are \textit{not aware} of a similar analysis or result elsewhere. Even though $\EE|X|^k\leq C^kk^{\frac{k}{\alpha}}\|X\|_{\Psi_\alpha}^{k}$ was proved to be equivalent to the definition of the Orlicz norm \citep{vershynin2018high}, the following lemma suggests a \textit{sharper} bound $ \EE|X|^k\leq C^{k}k^{\frac{k}{\alpha}}\cdot\sigma_X\|X\|_{\Psi_\alpha}^{k-1}$. For each fixed $\alpha$ and $k$, our bound is \textit{far smaller} than $C^kk^{\frac{k}{\alpha}}\|X\|_{\Psi_\alpha}^{k} $ when $\sigma_X\ll \|X\|_{\Psi_\alpha}$.

\begin{lemma} \label{new.lem5}
    Assume that $X$ is a mean zero random variable and satisfies that $\|X\|_{\Psi_\alpha}<\infty$ 
    for some $\alpha\geq1$. Denote by $\sigma_X^2:=\operatorname{var}(X)$. Then it holds that for some universal constant $C > 0$ and all $k\geq 1$,
    \begin{align*}
        \EE|X|^k\leq C^{k}k^{\frac{k}{\alpha}}\cdot\sigma_X\|X\|_{\Psi_\alpha}^{k-1}.
    \end{align*}
\end{lemma}

\textit{Proof}. 
By Hölder's inequality, we have that 
    \begin{align*}
        \EE|X|^k\leq \l(\EE X^2\r)^{\frac{1}{2}}\cdot \l(\EE|X|^{2k-2}\r)^{\frac{1}{2}}\leq \sigma_X\cdot \|X\|_{2k-2}^{k-1}.
    \end{align*}
    Note that in light of Lemma~\ref{lem:moment_bound}, it holds that for all $k\geq 1$, 
    $$\|X\|_k\leq Ck^{\frac{1}{\alpha}}\|X\|_{\Psi_\alpha}.$$ 
    Consequently, inserting this bound into the expression above yields that 
\begin{align*}
   \EE|X|^k & \leq \sigma_X\cdot C^{k-1}2^{\frac{k-1}{\alpha}}(k-1)^{\frac{k-1}{\alpha}}\|X\|_{\Psi_\alpha}^{k-1} \\
   & \leq C_1^{k}\cdot\sigma_X k^{\frac{k}{\alpha}}\|X\|_{\Psi_\alpha}^{k-1},
\end{align*}
where $C_1 > 0$ is some universal constant. This completes the proof of Lemma \ref{new.lem5}.

\subsection{Proof of Lemma~\ref{lem:momGeneratingFunct_SigmaL}}

Essentially, Lemma~\ref{lem:momGeneratingFunct_SigmaL} extends Lemma~\ref{lem:momGeneratingFunct}. We will consider the Taylor expansion of the moment generating function. Here, the term is $\lambda^k\sigma^2L^{k-2}$, whose $k$th root is $\lambda L$ as $k$ goes to infinity and this intuitively explains why there is a phase transition at $\lambda=\frac{1}{L}$ instead of $1/\sigma$ for a fixed $\alpha$. The following proof establishes the bound first for the case of $\alpha\geq 2$ and then for the case of $1\leq \alpha\leq 2$. 

    With the aid of the Taylor expansion and Definition~\ref{assm:sigma_L}, it holds that 
    \begin{align*}
        \EE\exp(\lambda X)\leq 1+\sum_{k=2}^{\infty}\frac{k^{\frac{k}{\alpha}}}{k!}\lambda^k\sigma^2L^{k-2}.
    \end{align*}
    In view of Stirling's inequality, the above expression can be further bounded as 
    \begin{align*}
        \EE\exp(\lambda X)\leq 1+\sum_{k=2}^{\infty}\frac{1}{\l[\frac{k}{\beta}\r]!}C^k\lambda^k\sigma^2L^{k-2};
    \end{align*}
    see the proof of Lemma~\ref{lem:momGeneratingFunct} in Section~\ref{sec:proof:univariate} for more details. We will bound the above expression for the cases of $\alpha\geq 2$ and $\alpha\in[1,2]$ separately.

    \textit{Case 1: $\alpha\geq 2$}. Note that 
    \begin{align*}
        \EE\exp(\lambda X)\leq 1+\frac{\sigma^2}{L^2} \sum_{k=2}^{\infty}\frac{1}{\l[\frac{k}{\beta}\r]!}C^k\lambda^kL^{k}=1+\frac{\sigma^2}{L^2}\l( \l(\sum_{k=2}^{\infty}\frac{1}{\l[\frac{k}{\beta}\r]!}C^k\lambda^kL^{k}+1 \r)-1\r).
    \end{align*}
    Then by the proof of Lemma~\ref{lem:momGeneratingFunct} in Section~\ref{sec:proof:univariate}, we have that when $\alpha\geq 2$,
    \begin{align*}
         \sum_{k=2}^{\infty}\frac{1}{\l[\frac{k}{\beta}\r]!}C^k\lambda^kL^{k}+1\leq \exp\l(C\min\l\{\lambda^2L^2,\lambda^\beta L^{\beta}\r\}\r).
    \end{align*}
    Hence, it follows that 
    \begin{align*}
        \EE\exp(\lambda X)\leq 1+\frac{\sigma^2}{L^2}\l(\exp\l(C\min\l\{\lambda^2L^2,\lambda^\beta L^{\beta}\r\}\r)-1 \r).
    \end{align*}
When $\lambda\leq\frac{1}{L}$, it holds that $\lambda L\leq 1$ and thus $\lambda^2L^2\leq \lambda^\beta L^\beta$. Consequently, when $\lambda\leq \frac{1}{L}$, we have that $$\exp(\lambda^2 L^2)-1\leq C_1\lambda^2 L^2,$$ which entails that 
    \begin{align*}
        \EE\exp(\lambda X)\leq 1+C_2\lambda^2\sigma^2\leq \exp(C_2\lambda^2 \sigma^2).
    \end{align*}
    
    On the other hand, it holds that for all $\lambda>0$,
    \begin{align*}
        \EE\exp(\lambda X)\leq 1+\frac{\sigma^2}{L^2} \l(\exp\l(C\lambda^\beta L^\beta\r)-1 \r).
    \end{align*}
    Then given $\sigma\leq L$, it follows from $\exp\l(C\lambda^\beta L^\beta\r)-1\geq0 $ that 
    $$\EE\exp(\lambda X)\leq 1+1\cdot  \l(\exp\l(C\lambda^\beta L^\beta\r)-1 \r)=\exp\l(C\lambda^\beta L^\beta\r).$$ Thus, when $\alpha\geq 2$, we can obtain that 
    \begin{align*}
        &\EE\exp(\lambda X)\leq\exp\l(C\lambda^2 \sigma^2\r) \ \ \text{ for } \lambda\leq\frac{1}{L},\\
        &\EE\exp(\lambda X)\leq\exp\l(C\lambda^\beta L^\beta\r) \ \ \text{ for all } \lambda\geq0.
    \end{align*}

    \textit{Case 2: $1\leq\alpha\leq 2$}. 
    By the proof of Lemma~\ref{lem:momGeneratingFunct}, when $\alpha\in[1,2)$, it holds that 
    \begin{align*}
        \sum_{k=2}^{\infty}\frac{1}{\l[\frac{k}{\beta}\r]!}C^k\lambda^kL^{k}+1\leq \exp\l(C\frac{1}{1-\tau}\lambda^2L^2\r)
    \end{align*}
for $\lambda\leq\frac{C\tau}{L}$, where $\tau$ is any constant in $(0,1)$. Setting $\tau=\frac{1}{2}$, we have that when $\lambda\leq\frac{C}{2L}$, $$\EE\exp(\lambda X)\leq 1+\frac{\sigma^2}{L^2}\l(\exp(C\lambda^2L^2)-1\r)\leq 1+C\lambda^2\sigma^2\leq\exp(C\lambda^2\sigma^2).$$
Further, for $\lambda\geq \frac{C\tau}{L}$ and $\alpha\in(1,2]$, it follows from Lemma~\ref{lem:momGeneratingFunct} that
    \begin{align*}
        \sum_{k=2}^{\infty}\frac{1}{\l[\frac{k}{\beta}\r]!}C^k\lambda^kL^{k}+1\leq \exp\l(C_1^\beta\frac{\tau^{-[\beta]-1}}{1-\tau}\lambda^\beta L^\beta\r).
    \end{align*}
    Hence, for $\lambda\geq \frac{C\tau}{L}$, we can obtain that 
    \begin{align*}
        \EE\exp(\lambda X)\leq 1+\frac{\sigma^2}{L^2}\l(\exp\l(C_1^\beta\frac{\tau^{-[\beta]-1}}{1-\tau}\lambda^\beta L^\beta\r)-1 \r).
    \end{align*}
    Therefore, given $\sigma\leq L$, we have that  $\frac{\sigma^2}{L^2}\leq1$ and thus  $$\EE\exp(\lambda X)\leq \exp\l(C_1^\beta\frac{\tau^{-[\beta]-1}}{1-\tau}\lambda^\beta L^\beta\r).$$ 
This concludes the proof of Lemma~\ref{lem:momGeneratingFunct_SigmaL}.

\subsection{Proof of Theorem~\ref{thm:conc_univariate_SigmaL}} 

The following proof is built upon Lemma~\ref{lem:momGeneratingFunct_SigmaL}. 
By the Markov inequality, it holds that for all $\lambda>0$, 
    \begin{align*}
        \PP\l(\sum_{i=1}^n a_iX_i \geq t\r)\leq \exp(-\lambda t)\prod_{i=1}^n\EE\exp\l(a_i\lambda X_i\r).
    \end{align*}
    Then for $\lambda\leq\frac{1}{\max_{i=1,\ldots,n} |a_i|L_i}$, an application of Lemma~\ref{lem:momGeneratingFunct} gives that
    \begin{align*}
        \PP\l(\l|\sum_{i=1}^n a_iX_i \r|\geq t\r)\leq 2\exp(-\lambda t)\exp\l(\lambda^2\sum_{i=1}^n a_i^2\sigma_i^2\r).
    \end{align*}
    Inserting $\lambda=\min\l\{\frac{ t}{\sum_{i=1}^n a_i^2\sigma_i^2},\frac{1}{\max_{i=1,\ldots,n}|a_i|L_i} \r\}$ into the above expression, we can deduce that for all $t\geq 0$ and $\alpha\geq 1$,
    \begin{align}
        \PP\l(\l|\sum_{i=1}^n a_iX_i \r|\geq t\r)\leq 2\exp\l( -C\min\l\{ \frac{t^2}{\sum_{i=1}^n a_i^2\sigma_i^2}, \frac{t}{\max_{i=1,\ldots,n}|a_i|L_i}\r\}\r).
        \label{eq1:proof:thm:conc_univariate_SigmaL}
    \end{align}
    
    Let us first investigate the case of $\alpha\geq 2$. 
    In this case, we have that $\EE\exp(\lambda X)\leq \exp(C\lambda^\beta L^\beta)$, which implies that 
    \begin{align*}
        \PP\l(\l|\sum_{i=1}^n a_iX_i \r|\geq t\r)\leq 2\exp(-\lambda t)\exp\l(C\lambda^\beta \sum_{i=1}^n |a_i|^\beta L_i^\beta\r).
    \end{align*}
    Inserting $\lambda=\l(\frac{t}{\beta\sum_{i=1}^n |a_i|^\beta L_i^{\beta}}\r)^{\frac{1}{\beta-1}}$ into the above expression leads to 
    \begin{align*}
        \PP\l(\l|\sum_{i=1}^n a_iX_i \r|\geq t\r)\leq 2\exp\l(-C\frac{t^\alpha}{\l(\sum_{i=1}^n |a_i|^\beta L_i^\beta\r)^{\frac{\alpha}{\beta}} }\r).
    \end{align*}
    Moreover, due to $\EE\exp(\lambda X)\leq \exp(C\lambda^2L^2)$, it holds that $$\PP\l(\l|\sum_{i=1}^n a_iX_i \r|\geq t\r)\leq 2\exp\l(-C\frac{t^2}{\sum_{i=1}^n a_i^2 L_i^2 }\r).$$ 
    Hence, we have that for $\alpha\geq 2$, 
    \begin{align*}
        &\PP\l(\l|\sum_{i=1}^n a_iX_i \r|\geq t\r)\leq 2\exp\l(-C\max\l\{\frac{t^\alpha}{\l(\sum_{i=1}^n |a_i|^\beta L_i^\beta\r)^{\frac{\alpha}{\beta}} } , \frac{t^2}{\sum_{i=1}^n a_i^2L_i^2},\r.\r.\\ 
        &\quad \l.\l.\min\l\{ \frac{t^2}{\sum_{i=1}^n a_i^2\sigma_i^2}, \frac{t}{\max_{i=1,\ldots,n}|a_i|L_i}\r\} \r\}\r).
    \end{align*}
    
    It remains to bound the tail probability for the case of $\alpha\in[1,2]$. Recall that Lemma~\ref{lem:momGeneratingFunct_SigmaL} establishes that for all $\lambda\geq 0$, 
    \begin{align*}
        \EE\exp(\lambda X)\leq \exp(C2^{\beta}\lambda^{\beta}L^\beta+C\lambda^2\sigma^2).
    \end{align*}
    Then it follows that 
    \begin{align*}
        \PP\l(\l|\sum_{i=1}^n a_iX_i \r|\geq t\r)\leq  2\exp\l(C\lambda^2\sum_{i=1}^n a_i^2 \sigma_i^2+C2^\beta \lambda^\beta \sum_{i=1}^n |a_i|^\beta L_i^{\beta}-\lambda t\r).
    \end{align*}
    By taking $\lambda$ to minimize the above tail probability, we can obtain that 
    \begin{align*}
        \PP\l(\l|\sum_{i=1}^n a_iX_i \r|\geq t\r)\leq  2\exp\l(-C\min\l\{\frac{t^2}{\sum_{i=1}^n a_i^2\sigma_i^2}, \frac{t^\alpha}{\l( \sum_{i=1}^n |a_i|^\beta L_i^\beta\r)^{\frac{\alpha}{\beta}} }\r\} \r).
    \end{align*}
    Thus, combining the above expression with 
    \eqref{eq1:proof:thm:conc_univariate_SigmaL} yields that 
    \begin{align*}
        & \PP\l(\l|\sum_{i=1}^n a_iX_i \r|\geq t\r)\\
        & \leq  2\exp\l(-C\min\l\{\frac{t^2}{\sum_{i=1}^n a_i^2\sigma_i^2}, \max\l\{\frac{t}{\max_{i=1,\ldots,n}|a_i|L_i},\frac{t^\alpha}{\l( \sum_{i=1}^n |a_i|^\beta L_i^\beta\r)^{\frac{\alpha}{\beta}} }\r\}\r\} \r).
    \end{align*}
This completes the proof of Theorem~\ref{thm:conc_univariate_SigmaL}.

\section{Proofs for Section~\ref{sec:app}}

This section presents the proofs of Theorems~\ref{thm:dependent}--\ref{thm:cov_est} in Section~\ref{sec:app}.

\subsection{Proof of Theorem~\ref{thm:dependent}}

Theorem~\ref{thm:dependent} remains valid for a finite sequence $\{a_k,X_k\}_{k=1}^n$ with $a_k=0$ for $k\geq n+1$. The following proof first bounds the moment generating function, which is \textit{different} from Lemma~\ref{lem:momGeneratingFunct_SigmaL} for the scenario of independent variables.

To prove this theorem, we need only to establish the results for $M_n$, and since $\lim_{n\to\infty}M_n=M_{\infty}$ almost surely, the concentration inequality also holds for $M_{\infty}$. Indeed, by the Markov inequality, it holds that 
    \begin{align*}
        \PP\l(\sum_{k=1}^n a_kX_k \geq t\r)&\leq\exp(-\lambda t)\EE\exp\bigg\{ \lambda\sum_{k=1}^n a_kX_k   \bigg\}.
    \end{align*}
    Observe that 
    \begin{align*}
        \EE\exp\bigg\{\lambda \sum_{k=1}^n a_kX_k   \bigg\}
        =\EE \l\{\exp\bigg\{ \lambda\sum_{k=1}^{n-1} a_kX_k   \bigg\}\cdot \EE\bigg\{\exp\bigg\{ \lambda a_nX_n   \bigg\}\bigg|\mathcal{F}_{n-1}\bigg\} \r\}.
    \end{align*}
    Additionally, we have that 
    \begin{align*}
        \EE\bigg\{\exp\bigg\{ \lambda a_nX_n   \bigg\}\bigg|\mathcal{F}_{n-1}\bigg\}&\leq \EE\bigg\{1+\sum_{p=2}^{\infty} \frac{1}{p!}\lambda^p |a_n|^p|X_n|^p \bigg|\mathcal{F}_{n-1}\bigg\}\\
        &\leq 1+\sum_{p=2}^{\infty} \frac{p^{\frac{p}{\alpha}}}{p!}\lambda^p |a_n|^p \sigma_n^2 L_n^{p-2},
    \end{align*}
    where the last step above follows from Assumption~\ref{assm:dependent}. Then by the proof of Lemma~\ref{lem:momGeneratingFunct_SigmaL} in Section \ref{sec:proof:univariate_var}, it holds that for $\lambda\leq\frac{c}{a_nL_n}$, 
    \begin{align*}
        \EE\exp(\lambda a_n X_n)\leq \exp(C\lambda^2 a_n^2\sigma_n^2).
    \end{align*}
    Hence, by induction, we can deduce that for all $\lambda\leq\frac{c}{\max_{k=1,\ldots.n} |a_k|L_k}$, 
    \begin{align*}
        \EE\exp\bigg\{ \lambda\sum_{k=1}^n a_kX_k   \bigg\}\leq\exp\l(C\lambda^2\sum_{k=1}^n a_n^2\sigma_n^2\r).
    \end{align*}
    Further, setting $\lambda:=\min\l\{ \frac{t}{\sum_{k=1}^n a_k^2\sigma_k^2},\frac{1}{\max_{k=1,\ldots,n} |a_k|L_k}\r\}$ yields that 
    \begin{align*}
        \PP\l(\sum_{k=1}^n a_kX_k \geq t\r)\leq \exp\l(-C\min\l\{\frac{t^2}{\sum_{k=1}^n a_k^2\sigma_k^2},\frac{t}{\max_{k=1,\ldots,n}|a_k|L_k}\r\}\r).
    \end{align*}
    The remaining arguments are similar to those in the proof of Theorem~\ref{thm:conc_univariate_SigmaL} in Section \ref{sec:proof:univariate_var} and thus omitted. 
    
    When $\alpha\geq 2$, it holds that for any fixed $n\geq 1$, 
    \begin{align*}
        & \PP\l(|M_n|\geq t\r)\leq 2\exp\l(-C\max\l\{\frac{t^\alpha}{\l(\sum_{k=1}^{\infty} |a_i|^\beta L_i^{\beta} \r)^{\frac{\alpha}{\beta}} }, \frac{t^2}{\sum_{k=1}^{\infty} a_i^2L_i^2},\r.\r.\\ 
        &\quad \l.\l.\min\l\{\frac{t^2}{\sum_{k=1}^{\infty} a_k^2\sigma_k^2}, \frac{t}{\max_{k=1,\ldots} |a_k|L_k}\r\}\r\}\r).
    \end{align*}
    Under 
    \eqref{eq1:dependent}, $M_n$ converges to $M_\infty$ almost surely which satisfies that 
    \begin{align*}
        & \PP\l(|M_\infty|\geq t\r)\leq 2\exp\l(-C\max\l\{\frac{t^\alpha}{\l(\sum_{k=1}^{\infty} |a_i|^\beta L_i^{\beta} \r)^{\frac{\alpha}{\beta}} }, \frac{t^2}{\sum_{k=1}^{\infty} a_i^2L_i^2},\r.\r.\\ 
        &\quad \l.\l.\min\l\{\frac{t^2}{\sum_{k=1}^{\infty} a_k^2\sigma_k^2}, \frac{t}{\max_{k=1,\ldots} |a_k|L_k}\r\}\r\}\r).
    \end{align*}
    
    When $\alpha\in[1,2]$, we have that for any fixed $n$, 
    \begin{align*}
        & \PP\l(\l|M_n \r|\geq t\r)\\
        & \leq  2\exp\l(-C\min\l\{\frac{t^2}{\sum_{k=1}^{\infty} a_k^2\sigma_k^2}, \max\l\{\frac{t}{\max_{k=1,\ldots}|a_k|L_k},\frac{t^\alpha}{\l( \sum_{k=1}^\infty |a_k|^\beta L_k^\beta\r)^{\frac{\alpha}{\beta}} }\r\}\r\} \r).
    \end{align*}
    Under 
    \eqref{eq1:dependent}, $M_n$ converges to $M_\infty$ almost surely and it also holds that 
    \begin{align*}
        & \PP\l(\l|M_\infty \r|\geq t\r)\\
        & \leq  2\exp\l(-C\min\l\{\frac{t^2}{\sum_{k=1}^{\infty} a_k^2\sigma_k^2}, \max\l\{\frac{t}{\max_{k=1,\ldots}|a_k|L_k},\frac{t^\alpha}{\l( \sum_{k=1}^\infty |a_k|^\beta L_k^\beta\r)^{\frac{\alpha}{\beta}} }\r\}\r\} \r).
    \end{align*}
This concludes the proof of Theorem~\ref{thm:dependent}.

\subsection{Proof of Theorem~\ref{thm:vec_an}}

The following proof exploits Theorem~\ref{thm:conc_univariates}. 
By invoking Theorem~\ref{thm:conc_univariates}, we have that when $\alpha\geq 4$, 
    \begin{align*}
        \PP\l(\l|\|X\|^2-\sum_{i=1}^d \sigma_{X_i}^2 \r|\geq t \r)\leq 2\exp\l(-C\max\l\{\frac{t^2}{\sum_{i=1}^d K_i^4},\frac{t^{\frac{\alpha}{2}}}{\l(\sum_{i=1}^d K_i^{\frac{2\alpha}{\alpha-2}}\r)^{\frac{\alpha-2}{2}}}\r\}\r).
    \end{align*}
    Additionally, an application of Lemma~\ref{teclem:minus} in Section \ref{new.SecD} shows that for any $s$ satisfying
    \begin{align*}
        \l|\|X\| -\sqrt{\sum_{i=1}^n \sigma_{X_i}^2}\r| \geq s,
    \end{align*}
    it holds that 
    \begin{align*}
        \l|\|X\|^2-\sum_{i=1}^d \sigma_{X_i}^2 \r|\geq \max\l\{s^2, s \sqrt{\sum_{i=1}^n \sigma_{X_i}^2}\r\}.
    \end{align*}
    Hence, we can deduce that 
    \begin{align*}
        &\PP\l(\l|\|X\| -\sqrt{\sum_{i=1}^n \sigma_{X_i}^2}\r| \geq s\r)\leq \PP\l(\l|\|X\|^2-\sum_{i=1}^d \sigma_{X_i}^2 \r|\geq \max\l\{s^2, s \sqrt{\sum_{i=1}^n \sigma_{X_i}^2}\r\} \r)\\
        &\leq 2\exp \l(-C\max\l\{\frac{s^4}{\sum_{i=1}^d K_i^4}, \frac{\sum_{i=1}^d \sigma_{X_i}^2}{\sum_{i=1}^d K_i^4}s^2,\frac{s^\alpha}{\l(\sum_{i=1}^d K_i^{\frac{2\alpha}{\alpha-2}}\r)^{\frac{\alpha-2}{2}}}, \frac{s^{\frac{\alpha}{2}}\l(\sum_{i=1}^d \sigma_{X_i}^2\r)^{\frac{\alpha}{4}} }{\l(\sum_{i=1}^d K_i^{\frac{2\alpha}{\alpha-2}}\r)^{\frac{\alpha-2}{2}}}\r\}\r).
    \end{align*}
    When $s\leq \sqrt{\sum_{i=1}^d \sigma_{X_i}^2}$, the tail probability above is dominated by $\frac{\sum_{i=1}^d \sigma_{X_i}^2}{\sum_{i=1}^d K_i^4}s^2$. On the other hand, when $s\geq \l(\frac{\l(\sum_{i=1}^d K_i^{\frac{2\alpha}{\alpha-2}}\r)^{\frac{\alpha-2}{2}} }{\sum_{i=1}^d K_i^4}\r)^{\frac{1}{\alpha-4}}$, the tail probability above is dominated by $\frac{s^\alpha}{\l(\sum_{i=1}^d K_i^{\frac{2\alpha}{\alpha-2}}\r)^{\frac{\alpha-2}{2}}}$.
    
    It remains to bound the tail probability for the case of $\alpha\in[2,4]$. In a similar fashion, we can obtain that 
    \begin{align*}
        \PP\l(\l|\|X\| -\sqrt{\sum_{i=1}^n \sigma_{X_i}^2}\r| \geq s\r) & \leq \PP\l(\l|\|X\|^2-\sum_{i=1}^d \sigma_{X_i}^2 \r|\geq \max\l\{s^2, s \sqrt{\sum_{i=1}^n \sigma_{X_i}^2}\r\} \r)\\
        & \leq 2\exp \l(-C\min\l\{\max\l\{\frac{s^4}{\sum_{i=1}^d K_i^4}, \frac{\sum_{i=1}^d \sigma_{X_i}^2}{\sum_{i=1}^d K_i^4}s^2\r\}, \r.\r.\\
        & \quad \l.\l.\max\l\{\frac{s^\alpha}{\l(\sum_{i=1}^d K_i^{\frac{2\alpha}{\alpha-2}}\r)^{\frac{\alpha-2}{2}}}, \frac{s^{\frac{\alpha}{2}}\l(\sum_{i=1}^d \sigma_{X_i}^2\r)^{\frac{\alpha}{4}} }{\l(\sum_{i=1}^d K_i^{\frac{2\alpha}{\alpha-2}}\r)^{\frac{\alpha-2}{2}}}\r\}\r\}\r).
    \end{align*}
    This completes the proof of Theorem~\ref{thm:vec_an}.

\subsection{Proof of Corollary~\ref{thm:vec}}
We emphasize that by definition, it holds that $\sigma_X\leq 2K$. Let us consider the squared Euclidean norm of $X$
    \begin{align*}
        \PP\l(\l|\|X\|^2-\EE\|X\|^2\r|\geq t \r)=\PP\l(\l|\sum_{i=1}^d \l(X_i^2-\EE X_i^2\r)\r|\geq t\r).
    \end{align*}
    Notice that $\EE\|X\|^2=d\sigma_X^2$. When $\alpha\geq 4$, an application of Theorem~\ref{thm:conc_univariates} gives that
    \begin{align}
        \PP\l(\l|\|X\|^2-d\sigma_X^2\r|\geq t \r)\leq 2\exp\l(-C\max\l\{\frac{t^2}{dK^4},\frac{t^{\frac{\alpha}{2}}}{d^{\frac{\alpha}{2}-1}K^{\alpha}}\r\}\r).
        \label{eq1:proof:vec}
    \end{align}
    Moreover, for any $s$ satisfying
    \begin{align*}
        \l|\|X\|- \sqrt{d}\sigma_X\r|\geq s,
    \end{align*}
    it follows from Lemma~\ref{teclem:minus} that 
    \begin{align*}
        \l|\|X\|^2-d\sigma_X^2\r|\geq \max\{s^2,\sqrt{d}\sigma_X s\}.
    \end{align*}
    Consequently, for any $s\geq 0$, it holds that 
    \begin{align*}
        \PP\l( \l|\|X\|-\sqrt{d}\sigma_X \r|\geq s\r)&\leq \PP\l( \l|\|X\|^2-d\sigma_X^2\r|\geq \max\{s^2,\sqrt{d}\sigma_X s\}\r)\\
        &\leq 2\exp\l(-C\max\l\{\frac{s^4}{dK^4},\frac{\sigma_X^2s^2}{K^4},\frac{s^\alpha}{d^{\frac{\alpha}{2}-1}K^\alpha}, \frac{s^{\frac{\alpha}{2}} \sigma_X^{\frac{\alpha}{2}}}{d^{\frac{\alpha}{4}-1}K^\alpha}\r\}\r).
    \end{align*}
    In other words, we have that 
    \begin{align*}
        \PP\l( \l|\frac{1}{\sqrt{d}}\|X\|-\sigma_X \r|\geq s\r)\leq 2\exp\l(-Cd\max\l\{\frac{s^4}{K^4},\frac{\sigma_X^2s^2}{K^4},\frac{s^\alpha}{K^\alpha}, \frac{s^{\frac{\alpha}{2}} \sigma_X^{\frac{\alpha}{2}}}{K^\alpha} \r\}\r).
    \end{align*}
    Obviously, when $s\leq \sigma_X$, the maximum term on the right-hand side of the expression above is $\frac{\sigma_X^2s^2}{K^4}$. When $s$ is between $\sigma_X$ and $K$, the maximum term is $\frac{s^4}{K^4}$. When $s\geq K$, the maximum term is $\frac{s^\alpha}{K^\alpha}$. Thus, we can simplify the above expression as 
    \begin{align*}
        \PP\l( \l|\frac{1}{\sqrt{d}}\|X\|-\sigma_X \r|\geq s\r)\leq 2\exp\l(-Cd\max\l\{\frac{s^4}{K^4},\frac{\sigma_X^2s^2}{K^4},\frac{s^\alpha}{K^\alpha} \r\}\r).
    \end{align*}
    
    It remains to analyze the norm for the scenario of $\alpha\in[2,4]$. In view of  Theorem~\ref{thm:conc_univariates}, it holds that 
    \begin{align*}
        \PP\l( \l|\|X\|^2-d\sigma_X^2\r|\geq t \r)\leq 2\exp\l(-C\min\l\{\frac{t^2}{dK^4},\frac{t^{\frac{\alpha}{2}}}{d^{\frac{\alpha}{2}-1} K^\alpha}\r\}\r).
    \end{align*}
    Further, an application of Lemma~\ref{teclem:minus} shows that for any $s\geq 0$, 
    \begin{align*}
        & \PP\l(\l|\|X\|-\sqrt{d}\sigma_X \r|\geq s \r)\leq \PP\l(\l|\|X\|^2-d\sigma_X^2 \r|\geq \max\l\{s^2,\sqrt{d}\sigma_X s \r\} \r)\\
        & \leq 2\exp\l(-C\min\l\{ \max\l\{\frac{s^4}{dK^4},\frac{\sigma_X^2 s^2}{K^4}\r\},\max\l\{\frac{s^\alpha}{d^{\frac{\alpha}{2}-1}K^\alpha},\frac{s^{\frac{\alpha}{2}}\sigma_X^{\frac{\alpha}{2}}}{d^{\frac{\alpha}{4}-1}K^\alpha}\r\} \r\}\r),
    \end{align*}
    which can be simplified as 
    \begin{align*}
         & \PP\l(\l|\frac{1}{\sqrt{d}}\|X\|-\sigma_X \r|\geq s \r)\\
         & \leq 2\exp\l(-Cd\min\l\{ \max\l\{\frac{s^4}{K^4},\frac{\sigma_X^2 s^2}{K^4}\r\},\max\l\{\frac{s^\alpha}{K^\alpha},\frac{s^{\frac{\alpha}{2}}\sigma_X^{\frac{\alpha}{2}}}{K^\alpha}\r\} \r\}\r).
    \end{align*}
    Indeed, when $s\leq \sigma_X$, the tail probability is dominated by term $\frac{\sigma_X^2 s^2}{K^4}$. When $s\in[\sigma_X,K]$, the tail probability is dominated by term $\exp\l(-Cd\frac{s^4}{K^4}\r)$. When $s\geq K$, the tail probability is dominated by term $\frac{s^\alpha}{K^\alpha}$. Therefore, we can further simplify the tail probability as 
    \begin{align*}
        \PP\l(\l|\frac{1}{\sqrt{d}}\|X\|-\sigma_X \r|\geq s \r)\leq 2\exp\l(-Cd\min\l\{ \max\l\{\frac{s^4}{K^4},\frac{\sigma_X^2 s^2}{K^4}\r\},\frac{s^\alpha}{K^\alpha} \r\}\r).
    \end{align*}
This concludes the proof of Corollary~\ref{thm:vec}.

\subsection{Proof of Theorem~\ref{thm:vec:Sigma_L}}

The following proof relies on Theorem~\ref{thm:conc_univariate_SigmaL}. 
From Theorem~\ref{thm:conc_univariate_SigmaL}, we have that when $\alpha\geq 2$, 
    \begin{multline}
        \PP\l(\l|\|X\|^2-\EE\|X\|^2 \r|\geq t \r)\\\leq 2\exp\l(-C\max\l\{\frac{t^\alpha}{\l(\sum_{i=1}^d L_i^\beta\r)^{\frac{\alpha}{\beta}}}, \frac{t^2}{\sum_{i=1}^d L_i^2},\min\l\{\frac{t^2}{\sum_{i=1}^d \sigma_i^2},\frac{t}{\max L_i}\r\}\r\}\r).
    \end{multline}
    Then an application of Lemma~\ref{teclem:minus} leads to 
    \begin{multline}
        \PP\l(\l|\|X\|-\sqrt{\EE\|X\|^2} \r|\geq s \r)\\
        \leq 2\exp\l(-C\max\l\{\frac{s^{2\alpha}}{\l(\sum_{i=1}^d L_i^\beta\r)^{\frac{\alpha}{\beta}}}, \frac{s^\alpha \l(\sum_{i=1}^{d}\sigma_{X_i}^2\r)^{\frac{\alpha}{2}}}{\l(\sum_{i=1}^d L_i^\beta\r)^{\frac{\alpha}{\beta}}},\frac{s^4}{\sum_{i=1}^d L_i^2},\frac{s^2\l(\sum_{i=1}^d \sigma_{X_i}^2\r)}{\sum_{i=1}^d L_i^2},\r.\r.\\\l.\l. \min\l\{ \max\l\{\frac{s^4}{\sum_{i=1}^d \sigma_i^2},\frac{\sum_{i=1}^d \sigma_{X_i}^2}{\sum_{i=1}^d \sigma_i^2}s^2\r\},\max\l\{\frac{s^2}{\max L_i},\frac{s\sqrt{\sum_{i=1}^d \sigma_{X_i}^2}}{\max L_i}\r\}\r\}\r\}\r).
    \end{multline}
    For the setting of common $(\sigma_i,L_i) = (\sigma,L)$, the above expression reduces to
    \begin{align*}
        \PP\l(\l|\|X\|-\sqrt{d}\sigma_X \r|\geq s \r) & \leq 2\exp\l(-C\max\l\{\frac{s^{2\alpha}}{d^{\alpha-1}L^\alpha},\frac{s^{\alpha}\sigma_X^{\alpha}}{d^{\frac{\alpha}{2}-1}L^{\alpha}}, \frac{s^4}{dL^2}, \frac{s^2\sigma_{X}^2}{L^2},\r.\r. \\
        &\quad \l.\l.\min\l\{ \max\l\{\frac{s^4}{d \sigma^2},\frac{\sigma_X^2}{\sigma^2}s^2\r\},\max\l\{\frac{s^2}{ L},\frac{s\sqrt{d} \sigma_X}{L}\r\}\r\}\r\}\r).
    \end{align*}
    With some calculations of the bound, we can further simplify the above expression as 
    \begin{align*}
        \PP\l(\l|\|X\|-\sqrt{\EE\|X\|^2} \r|\geq s \r) & \leq 2\exp\l(-C\max\l\{\frac{s^{2\alpha}}{d^{\alpha-1}L^\alpha}, \frac{s^4}{dL^2}, \frac{s^2\sigma_{X}^2}{L^2},\r.\r. \\ 
        & \quad \l.\l.\min\l\{ \max\l\{\frac{s^4}{d \sigma^2},\frac{\sigma_X^2}{\sigma^2}s^2\r\},\max\l\{\frac{s^2}{ L},\frac{s\sqrt{d} \sigma_X}{L}\r\}\r\}\r\}\r).
    \end{align*}

    It thus remains to examine the case of $\alpha\in[1,2]$. Similarly, in light of Theorem~\ref{thm:conc_univariate_SigmaL} it holds that 
    \begin{align*}
        & \PP\l(\l|\|X\|^2-\EE\|X\|^2 \r|\geq t \r)\\
        & \leq 2\exp\l(-C\min\l\{\frac{t^2}{\sum_{i=1}^d \sigma_i^2},\max\l\{\frac{t}{\max_i L_i},\frac{t^\alpha}{\l(\sum_{i=1}^n L_i^\beta\r)^{\frac{\alpha}{\beta}}}\r\}\r\}\r).
    \end{align*}
    Further, it follows from Lemma~\ref{teclem:minus} that 
    \begin{multline*}
        \PP\l(\l|\|X\|-\sqrt{\EE\|X\|^2} \r|\geq s \r)\leq 2\exp\l(-C\min\l\{\max\l\{\frac{s^4}{\sum_{i=1}^d \sigma_i^2}, \frac{s^2\l(\sum_{i=1}^d \sigma_{X_i}^2\r)}{\sum_{i=1}^d \sigma_i^2}\r\}
        ,\r.\r.\\\l.\l. \min\l\{ \max\l\{\frac{s^{2\alpha}}{\l(\sum_{i=1}^d L_i^\beta\r)^{\frac{\alpha}{\beta}} },\frac{\l(\sum_{i=1}^d \sigma_{X_i}^2\r)^{\frac{\alpha}{2}}}{\l(\sum_{i=1}^d L_i^\beta\r)^{\frac{\alpha}{\beta}}}s^\alpha\r\},\max\l\{\frac{s^2}{\max L_i},\frac{s\sqrt{\sum_{i=1}^d \sigma_{X_i}^2}}{\max L_i}\r\}\r\}\r\}\r).
    \end{multline*}
    Therefore, substituting the common $(\sigma_i,L_i) = (\sigma,L)$ into the above expression yields that 
    \begin{align*}
        & \PP\l(\l|\|X\|-\sqrt{d}\sigma_X \r|\geq s \r)\leq 2\exp\l(-C\min\l\{\max\l\{\frac{s^4}{d \sigma^2}, \frac{s^2\sigma_{X}^2}{ \sigma^2}\r\}
        ,\r.\r.\\
        &\quad \l.\l. \min\l\{ \max\l\{\frac{s^{2\alpha}}{d^{\alpha-1} L^\alpha},\frac{d^{1-\frac{\alpha}{2}} \sigma_{X}^\alpha}{ L^\alpha}s^\alpha\r\},\max\l\{\frac{s^2}{L},\frac{s\sqrt{d }\sigma_X}{ L}\r\}\r\}\r\}\r),
    \end{align*}
    which completes the proof of Theorem~\ref{thm:vec:Sigma_L}.

\subsection{Proof of Lemma~\ref{lem:matrix}} The following proof is based on the concentration inequality presented in Theorem~\ref{thm:conc_univariate_SigmaL} and a standard $\eps$-net argument. We refer interested readers to \cite{tao2023topics} and \cite{vershynin2018high} for the detailed discussion of $\eps$-net. From the definition of the operator norm, it holds that 
\begin{align*}
    \l\|\frac{1}{d_1}X^{\top}X -\Sigma\r\|=\sup_{u\in\SS^{d_2-1}}\l\{ \frac{1}{d_1}\l\|Xu \r\|^2-u^\top \Sigma u\r\}.
\end{align*}
In particular, vector $Xu$ consists of independent components
\begin{align*}
    Xu=\l(\begin{matrix}
        X_1^{\top} u\\
        X_2^{\top}u\\
        \vdots\\
        X_{d_1}^{\top}u
    \end{matrix}\r),
\end{align*}
which arises from the independent row vectors of $X$. Additionally, it should be noted that
\begin{align*}
    \EE\l(X_i^{\top}u\r)^2=u^{\top}\Sigma u.
\end{align*}
Hence, when $\alpha\geq 2$, an application of Theorem~\ref{thm:conc_univariate_SigmaL} yields for any $t>0$, 
\begin{align}
    \PP\l(\frac{1}{d_1}\l\|Xu \r\|^2-u^\top \Sigma u  \geq t \r)\leq \exp\l(-Cd_1\max\l\{\frac{t^\alpha}{L^\alpha},\min\l\{\frac{t^2}{\sigma^2},\frac{t}{L}\r\}\r\}\r).
    \label{eq1:proof:lem:matrix}
\end{align}

We emphasize that on the right-hand side of the above expression, the coefficient of $\exp$ is one instead of $2$, which is because the above expression bounds a mean-zero 
random variable rather than its absolute value. On the other hand, by the standard arguments of $\eps$-net (see details in the classical books \cite{tao2023topics} and \cite{vershynin2018high}), we have that 
\begin{align*}
    \l\|\frac{1}{d_1}X^{\top}X -\Sigma\r\|\leq 2\sup_{u\in\mathcal{N}}\l\{ \frac{1}{d_1}\l\|Xu \r\|^2-u^\top \Sigma u\r\},
\end{align*}
where $\mathcal{N}$ is a $\eps=\frac{1}{2}$-net of $\SS^{d_2-1}$ with cardinality $|\mathcal{N}|\leq 5^n$. Then after taking the union over $\mathcal{N}$ in \eqref{eq1:proof:lem:matrix}, we can obtain that 
\begin{multline*}
    \PP \l(\l\|\frac{1}{d_1}X^{\top}X -\Sigma\r\|\geq 2t \r)\leq \PP\l(\sup_{u\in\mathcal{N}}\l|\frac{1}{d_1}\l\|Xu \r\|^2-u^\top \Sigma u  \r|\geq t \r)\\
    \quad \leq  5^{d_2}\cdot\exp\l(-Cd_1\max\l\{\frac{t^\alpha}{L^\alpha},\min\l\{\frac{t^2}{\sigma^2},\frac{t}{L}\r\}\r\}\r),
\end{multline*}
which is in fact equivalent to
\begin{align*}
    \PP \l(\l\|\frac{1}{d_1}X^{\top}X -\Sigma\r\|\geq C\min\l\{L\l(\frac{t+d_2}{d_1}\r)^{\frac{1}{\alpha}},\sigma\l(\frac{t+d_2}{d_1}\r)^{\frac{1}{2}}+L\cdot\frac{t+d_2}{d_1}\r\}\r)
    \leq \exp(-t).
\end{align*}

It remains to consider the case when $\alpha\in[1,2]$. In light of Theorem~\ref{thm:conc_univariate_SigmaL}, we have that 
\begin{align}
    \PP\l(\frac{1}{d_1}\l\|Xu \r\|^2-u^\top \Sigma u  \geq t \r)\leq \exp\l(-Cd_1\min\l\{\frac{t^2}{\sigma^2},\max\l\{\frac{t^\alpha}{L^\alpha},\frac{t}{L}\r\}\r\}\r).
    \label{eq2:proof:lem:matrix}
\end{align}
Similar to the arguments for the case of $\alpha\geq 2$, \eqref{eq2:proof:lem:matrix} furthers implies that 
\begin{multline*}
    \PP \l(\l\|\frac{1}{d_1}X^{\top}X -\Sigma\r\|\geq 2t \r)\leq \PP\l(\sup_{u\in\mathcal{N}}\l|\frac{1}{d_1}\l\|Xu \r\|^2-u^\top \Sigma u  \r|\geq t \r)\\
    \quad \leq  5^{d_2}\cdot\exp\l(-Cd_1\min\l\{\frac{t^2}{\sigma^2},\max\l\{\frac{t^\alpha}{L^\alpha},\frac{t}{L}\r\}\r\}\r),
\end{multline*}
which is equivalent to
\begin{align*}
    \PP \l(\l\|\frac{1}{d_1}X^{\top}X -\Sigma\r\|\geq C\sigma\l(\frac{t+d_2}{d_1}\r)^{\frac{1}{2}}+CL\min\l\{\l(\frac{t+d_2}{d_1}\r)^{\frac{1}{\alpha}},\frac{t+d_2}{d_1}\r\}\r)
    \leq \exp(-t).
\end{align*}
This concludes the proof of Lemma~\ref{lem:matrix}.

\subsection{Proof of Theorem~\ref{thm:singular_value}} 
The proof of Theorem~\ref{thm:singular_value} relies on the results in Lemma~\ref{lem:matrix} and the well-known Weyl's Inequality \citep{weyl1912asymptotische}. Recall that Lemma~\ref{lem:matrix} proves that when $\alpha\geq2$, with probability over $1-\exp(-t)$,
\begin{align}
    \l\|\frac{1}{d_1}X^{\top}X -\Sigma\r\|\leq C\min\l\{L\l(\frac{t+d_2}{d_1}\r)^{\frac{1}{\alpha}},\sigma\l(\frac{t+d_2}{d_1}\r)^{\frac{1}{2}}+L\frac{t+d_2}{d_1}\r\}.
    \label{eq1:proof:thm:singular_value}
\end{align}
An application of Weyl's Inequality \citep{weyl1912asymptotische} gives that 
\begin{align*}
    \frac{1}{d_1}s_{\max}^2(X)\leq M_{\max}^2+\l\|\frac{1}{d_1}X^{\top}X-\Sigma \r\|.
\end{align*}
Consequently, the above result together with \eqref{eq1:proof:thm:singular_value} entails that for any $t>0$, 
\begin{align*}
    s_{\max}^2(X)\leq d_1M_{\max}^2+ C\min\l\{L(d_2+t)^{\frac{1}{\alpha}}d_1^{\frac{1}{\beta}},\,\sigma(d_2+t)^{\frac{1}{2}}d_1^{\frac{1}{2}}+L(d_2+t)\r\}
\end{align*}
holds with probability over $1-\exp(-t)$. 
Additionally, Weyl's inequality yields that
$$ \frac{1}{d_1}s_{\min}^2(X)\geq M_{\min}^2-\l\|\frac{1}{d_1}X^{\top}X-\Sigma \r\|,$$
and thus, we have the following lower bound on the smallest nonzero singular value under \eqref{eq1:proof:thm:singular_value}
\begin{align*}
    s_{\min}^2(X)\geq d_1M_{\min}^2- C\min\l\{L(d_2+t)^{\frac{1}{\alpha}}d_1^{\frac{1}{\beta}},\,\sigma(d_2+t)^{\frac{1}{2}}d_1^{\frac{1}{2}}+L(d_2+t)\r\}.
\end{align*}

It remains to investigate the case of $1\leq\alpha\leq 2$. In view of Lemma~\ref{lem:matrix}, we have that with probability over $1-\exp(-t)$, 
\begin{align}
    \l\|\frac{1}{d_1}X^{\top}X -\Sigma\r\|\leq C\sigma\l(\frac{t+d_2}{d_1}\r)^{\frac{1}{2}}+C\min\l\{L\l(\frac{t+d_2}{d_1}\r)^{\frac{1}{\alpha}},L\frac{t+d_2}{d_1}\r\}.
    \label{eq2:proof:thm:singular_value}
\end{align}
Similar to the arguments for the case of $\alpha\geq 2$, it holds with probability over $1-\exp(-t)$ that 
\begin{align*}
    s_{\max}^2(X)\leq d_1M_{\max}^2+ C\sigma(d_2+t)^{\frac{1}{2}}d_1^{\frac{1}{2}}+C\min\l\{L(d_2+t)^{\frac{1}{\alpha}}d_1^{\frac{1}{\beta}},\,L(d_2+t)\r\}
\end{align*}
and
\begin{align*}
    s_{\min}^2(X)\geq d_1M_{\min}^2- C\sigma(d_2+t)^{\frac{1}{2}}d_1^{\frac{1}{2}}-C\min\l\{L(d_2+t)^{\frac{1}{\alpha}}d_1^{\frac{1}{\beta}},\,L(d_2+t)\r\}.
\end{align*}
This completes the proof of Theorem~\ref{thm:singular_value}.

\subsection{Proof of Lemma~\ref{lem:mean_est}} 
The proof of Lemma~\ref{lem:mean_est} employs Theorem~\ref{thm:conc_univariate_SigmaL} and the standard $\eps$-net argument \citep{tao2023topics,vershynin2018high}. It follows from the definition of the Euclidean norm that 
\begin{align*}
    \l\| \wh{\mu}-\mu\r\|=\sup_{u\in\SS^{d-1}}\l(\wh\mu -\mu\r)^{\top}u.
\end{align*}
For any fixed $u\in\SS^{d-1}$, we see that 
\begin{align*}
    \l(\wh\mu -\mu\r)^{\top}u=\frac{1}{n}\sum_{i=1}^n (X_i-\mu)^{\top}u
\end{align*}
is the sum of $n$ i.i.d. random variables. Hence, for $\alpha\geq 2$, an application of Theorem~\ref{thm:conc_univariate_SigmaL} gives that for any $t>0$, 
\begin{align*}
    \PP\l(\frac{1}{n}\sum_{i=1}^n (X_i-\mu)^{\top}u\geq t \r)\leq \exp\l(-Cn\max\l\{\frac{t^\alpha}{L^\alpha},\min\l\{\frac{t^2}{\sigma^2},\frac{t}{L}\r\}\r\}\r).
\end{align*}
A standard $\eps$-net argument (more details can be found in \cite{tao2023topics} and \cite{vershynin2018high}) leads to
\begin{align*}
    \PP\l(\l\| \wh\mu-\mu\r\|\geq 2t \r)\leq 5^{d}\cdot \exp\l(-Cn\max\l\{\frac{t^\alpha}{L^\alpha},\min\l\{\frac{t^2}{\sigma^2},\frac{t}{L}\r\}\r\}\r),
\end{align*}
which is in fact equivalent to
\begin{align*}
    \PP\l(\l\| \wh\mu-\mu\r\|\geq C\min\l\{L\l(\frac{t+d}{n}\r)^{\frac{1}{\alpha}},\sigma\l(\frac{t+d}{n}\r)^{\frac{1}{2}}+L\cdot\frac{t+d}{n}\r\} \r)\leq \exp(-t).
\end{align*}

On the other hand, when $\alpha\in[1,2]$, by invoking Theorem~\ref{thm:conc_univariate_SigmaL}, we can obtain that 
\begin{align*}
    \PP\l(\frac{1}{n}\sum_{i=1}^n (X_i-\mu)^{\top}u\geq t \r)\leq \exp\l(-Cn\min\l\{\frac{t^2}{\sigma^2},\max\l\{\frac{t^\alpha}{L^\alpha},\frac{t}{L}\r\}\r\}\r).
\end{align*}
It then implies the following bound on the norm of $\wh\mu-\mu$
\begin{align*}
    \PP\l(\l\| \wh\mu-\mu\r\|\geq C\sigma\l(\frac{t+d}{n}\r)^{\frac{1}{2}}+ CL\min\l\{\l(\frac{t+d}{n}\r)^{\frac{1}{\alpha}},\frac{t+d}{n}\r\} \r)\leq \exp(-t),
\end{align*}
which concludes the proof of Lemma~\ref{lem:mean_est}.

\subsection{Proof of Theorem~\ref{thm:cov_est}} The estimation error of the sample covariance matrix can be decomposed into two terms. One term is the square of the mean estimation error, and another term is the quantity given in Lemma~\ref{lem:matrix}. With some calculations, $\l\|\wh\Sigma-\Sigma\r\|$ can be expressed as
\begin{align*}
    \l\|\wh\Sigma-\Sigma\r\|&=\l\|\frac{1}{n}\sum_{i=1}^n (X_i-\wh \mu)(X_i-\wh \mu)^\top -\Sigma \r\|\\
    &=\l\|\frac{1}{n}\sum_{i=1}^n (X_i- \mu)(X_i- \mu)^\top -\Sigma-(\wh \mu -\mu)(\wh \mu -\mu)^\top \r\|.
\end{align*}
Then by resorting to the triangle inequality, $\l\|\wh\Sigma-\Sigma\r\|$ can be upper bounded as
\begin{equation}
    \begin{split}
        \l\|\wh\Sigma-\Sigma\r\|&\leq \l\|\frac{1}{n}\sum_{i=1}^n (X_i- \mu)(X_i- \mu)^\top -\Sigma\r\|+\l\|(\wh \mu -\mu)(\wh \mu -\mu)^\top \r\|\\
    &=\l\|\frac{1}{n}\sum_{i=1}^n (X_i- \mu)(X_i- \mu)^\top -\Sigma\r\|+\big\|\wh \mu -\mu \big\|^2.
    \end{split}
    \label{eq2:proof:thm:cov_est}
\end{equation}
In particular, the first term on the right-hand side of \eqref{eq2:proof:thm:cov_est} above can be bounded by Lemma~\ref{lem:matrix}, while the second term above is the squared estimation error of the mean that is given by Lemma~\ref{lem:mean_est}. To this end, it remains to verify the conditions of Lemmas~\ref{lem:matrix} and \ref{lem:mean_est}.

An application of Jensen's inequality gives that 
\begin{multline}
    \EE\l|\l((X-\mu)^{\top}u\r)^2-\EE\l((X-\mu)^{\top}u\r)^2 \r|^k \\
    \leq 2^k\EE\l|\l((X-\mu)^{\top}u\r) \r|^{2k}+2^k\l( \EE\l((X-\mu)^{\top}u\r)^2\r)^k.
    \label{eq1:proof:thm:cov_est}
\end{multline}
Additionally, it follows from the conditions of Theorem~\ref{thm:cov_est} that 
\begin{align*}
    \EE\l|\l((X-\mu)^{\top}u\r) \r|^{2k}\leq 2^{\frac{k}{\alpha}}k^{\frac{k}{\alpha/2}}\cdot \sigma^2L^{2k-2},\quad \EE\l((X-\mu)^{\top}u\r)^2\leq 2^{\frac{2}{\alpha}}\sigma^2,
\end{align*}
which lead to the following upper bound for \eqref{eq1:proof:thm:cov_est} 
\begin{align*}
    \EE\l|\l((X-\mu)^{\top}u\r)^2-\EE\l((X-\mu)^{\top}u\r)^2 \r|^k\leq 2^{1+k/\alpha}k^{\frac{k}{\alpha/2}} \cdot (\sigma L)^2\cdot (L^2)^{k-2}.
\end{align*}
Hence, we have verified the condition of Lemma~\ref{lem:matrix}; that is, the condition of Lemma~\ref{lem:matrix} is satisfied with $(\frac{\alpha}{2},\sigma L,4L^2)$. Consequently, an application of Lemma~\ref{lem:matrix} yields that
\begin{multline*}
    \PP\l(\l\|\frac{1}{n}\sum_{i=1}^n (X_i- \mu)(X_i- \mu)^\top -\Sigma\r\|\geq C\min\l\{L^2\l(\frac{t+d}{n}\r)^{\frac{2}{\alpha}},\r.\r.\\ \l.\l.\sigma L\l(\frac{t+d}{n}\r)^{\frac{1}{2}}+L^2\cdot\frac{t+d}{n}\r\} \r)\leq \exp(-t) \ \  \text{ when } \alpha\geq 4,
\end{multline*}
and 
\begin{multline*}
    \PP\l(\l\|\frac{1}{n}\sum_{i=1}^n (X_i- \mu)(X_i- \mu)^\top -\Sigma\r\|\geq C\sigma L\l(\frac{t+d}{n}\r)^{\frac{1}{2}}\r.\\\l.+C\min\l\{L^2\l(\frac{t+d}{n}\r)^{\frac{2}{\alpha}},L^2\cdot\frac{t+d}{n}\r\} \r)\leq \exp(-t) \ \  \text{ when } 2\leq \alpha\leq 4.
\end{multline*}

So far, we have finished bounding the first term on the right-hand side of \eqref{eq2:proof:thm:cov_est} above. It remains to analyze the second term on the right-hand side of \eqref{eq2:proof:thm:cov_est}. Regarding the second term, Lemma~\ref{lem:mean_est} proves that
\begin{align*}
    \PP\l(\|\wh{\mu}-\mu\|\geq C\min\l\{L\l(\frac{t+d}{n}\r)^{\frac{1}{\alpha}},\sigma\l(\frac{t+d}{n}\r)^{\frac{1}{2}}+L\frac{t+d}{n}\r\} \r)\leq \exp(-t),
\end{align*}
which is equivalent to
\begin{align*}
    \PP\l(\|\wh{\mu}-\mu\|^2\geq C\min\l\{L^2\l(\frac{t+d}{n}\r)^{\frac{2}{\alpha}},\sigma^2\cdot\frac{t+d}{n}+L^2\l(\frac{t+d}{n}\r)^2\r\} \r)\leq \exp(-t).
\end{align*}
Therefore, putting the above results together, we can obtain that when $\alpha\geq 4$, 
\begin{align*}
    \PP\l(\l\|\wh\Sigma-\Sigma \r\|\geq C \min\l\{L^2\l(\frac{t+d}{n}\r)^{\frac{2}{\alpha}},\sigma L\l(\frac{t+d}{n}\r)^{\frac{1}{2}}+L^2 \frac{t+d}{n} \r\}\r)\leq 2\exp(-t),
\end{align*}
and when $2\leq \alpha\leq 4$, 
\begin{align*}
    \PP\l(\l\|\wh\Sigma-\Sigma \r\|\geq C \sigma L\l(\frac{t+d}{n}\r)^{\frac{1}{2}}+C\min\l\{L^2\l(\frac{t+d}{n}\r)^{\frac{2}{\alpha}},L^2 \frac{t+d}{n} \r\}\r)\leq 2\exp(-t).
\end{align*}
This completes the proof of Theorem~\ref{thm:cov_est}.

\section{Technical Lemmas} \label{new.SecD}

We provide in this section some technical lemmas and their proofs.

\begin{lemma} 
    Let $K,\tau>0$ be any given numbers. Then for any $q\leq 0$, it holds that 
    \begin{align*}
        \int_{\tau}^{\infty}t^{q} \exp\l(-\frac{t}{K}\r)\, dt\leq \tau^{q}K\exp\l(-\frac{\tau}{K}\r),
    \end{align*}
    and consequently, $\int_{\tau}^{\infty}t^{q} \exp\l(-\frac{t}{K}\r)\, dt\leq K^{q+1}\exp\l(-\frac{\tau}{K}\r) $ when $\tau/K\geq 1$. For any $q\geq 0$, it holds that 
    \begin{align*}
        \int_{\tau}^{\infty}t^{q} \exp\l(-\frac{t}{K}\r)\, dt\leq K^{q+1}\sqrt{\Gamma\l(2q+1\r)}\exp\l(-\frac{\tau}{2K}\r).
    \end{align*}
   \label{teclem:integral}
\end{lemma}

\textit{Proof}. 
    For the case of $q\leq 0$, it holds that 
    \begin{align*}
         \int_{\tau}^{\infty}t^{q} \exp\l(-\frac{t}{K}\r)\, dt\leq  \tau^{q}K\int_{\tau}^{\infty}\frac{1}{K} \exp\l(-\frac{t}{K}\r)\, dt= \tau^{q}K\exp\l(-\frac{\tau}{K}\r),
    \end{align*}
which entails that 
$$\int_{\tau}^{\infty}t^{q} \exp\l(-\frac{t}{K}\r)\, dt\leq K^{q+1}\exp\l(-\frac{\tau}{K}\r) $$
when $\tau/K\geq 1$. It remains to consider the case when $q\geq 0$. By resorting to Hölder's inequality, we can deduce that 
    \begin{align*}
        \int_{\tau}^{\infty}t^{q} \exp\l(-\frac{t}{K}\r)\, dt\leq \l(\int_{\tau}^{\infty}t^{2q} \exp\l(-\frac{t}{K}\r)\, dt \r)^{\frac{1}{2}}\l(\int_{\tau}^{\infty} \exp\l(-\frac{t}{K}\r)\, dt \r)^{\frac{1}{2}}.
    \end{align*}
    Note that 
    \begin{align*}
        \int_{\tau}^{\infty} \exp\l(-\frac{t}{K}\r)\, dt = K\exp\l(-\frac{\tau}{K}\r)
    \end{align*}
    and
    \begin{align*}
        \int_{\tau}^{\infty}t^{2q} \exp\l(-\frac{t}{K}\r)\, dt&\leq \int_{0}^{\infty}t^{2q} \exp\l(-\frac{t}{K}\r)\, dt=K^{2q+1}\Gamma\l(2q+1\r).
    \end{align*}
    Thus, when $q\geq 0$ we can obtain that 
    \begin{align*}
        \int_{\tau}^{\infty}t^{q} \exp\l(-\frac{t}{K}\r)\, dt\leq K^{q+1}\sqrt{\Gamma\l(2q+1\r)}\exp\l(-\frac{\tau}{2K}\r).
    \end{align*}
This concludes the proof of Lemma \ref{teclem:integral}.

\begin{lemma} 
    Assume that $X$ is a mean zero random variable and satisfies that $\|X\|_{\Psi_\alpha}<\infty$ for some $\alpha\geq1$. Denote by $\sigma_X^2:=\operatorname{var}(X)$. Then it holds that for some universal constant $C>0$,
    \begin{align*}
        \EE|X|^3\leq C\sigma_X^2\|X\|_{\Psi_\alpha}\l(\log\l(\frac{\|X\|_{\Psi_\alpha}}{\sigma_X}\r)\r)^{\frac{1}{\alpha}}.
    \end{align*}
    \label{teclem:cube}
\end{lemma}

\textit{Proof}. Observe that for any $\tau>0$, it holds that 
    \begin{align*}
    \EE|X|^3=\EE X^2\cdot|X|=\EE X^2\cdot|X|\cdot\II\{|X|\geq \tau\}+\EE X^2\cdot|X|\cdot\II\{|X|\leq \tau\}.
\end{align*}
The second term on the right-hand side of the expression above can be bounded as 
\begin{align*}
    \EE X^2\cdot|X|\cdot\II\{|X|\leq \tau\}\leq \tau\EE X^2=\tau\sigma_X^2.
\end{align*}
It remains to consider the first term above, which can be bounded as 
\begin{align*}
    \EE X^2\cdot|X|\cdot\II\{|X|\geq \tau\} & \leq \l(\EE X^6\r)^{\frac{1}{2}} \l(\PP\l(|X|\geq \tau \r)\r)^{\frac{1}{2}}\\
    & \leq C\|X\|_{\Psi_\alpha}^3\exp\l(-c\l(\frac{\tau}{\|X\|_{\Psi_{\alpha}}}\r)^{\alpha}\r).
\end{align*}
Hence, by taking $\tau=C\|X\|_{\Psi_\alpha} \l(\log\l(\frac{2\|X\|_{\Psi_\alpha}}{\sigma_X}\r)\r)^{\frac{1}{\alpha}}$, we can obtain that 
\begin{align*}
    \EE |X|^3\leq C\sigma_X^2\|X\|_{\Psi_\alpha} \l(\log\l(\frac{2\|X\|_{\Psi_\alpha}}{\sigma_X}\r)\r)^{\frac{1}{\alpha}}.
\end{align*}
This completes the proof of Lemma \ref{teclem:cube}.

\begin{lemma}
    For any nonnegative values $a,b\geq0$ and any $t>0$, if $|a-b|\geq t$, we have that $|a^2-b^2|\geq \max\{bt,t^2\}$.
    \label{teclem:minus}
\end{lemma}

\textit{Proof}. 
    We will bound $|a^2-b^2|$ for the cases when $a\geq b$ and $a\leq b$ separately. When $a\geq b$, it follows from $|a-b|\geq t$ that $a\geq t+b$, which along with $b \geq 0$ and $t > 0$ entails that 
    $$|a^2-b^2| = a^2-b^2\geq (t+b)^2-b^2=t^2+2bt\geq \max\{t^2,bt\}.$$ When $a\leq b$, it follows from $|a-b|\geq t$ that $a\leq b-t$, which together with $a \geq 0$ and $t > 0$ yields that $$|a^2-b^2| = b^2-a^2\geq 2bt-t^2\geq bt=\max\{bt,t^2\}.$$ 
This concludes the proof of Lemma \ref{teclem:minus}.

\begin{lemma}[\cite{zhang2022sharper,zajkowski2020norms}]
    If $\|X\|_{\Psi_{\alpha}}< \infty$ for some $\alpha\geq 1$, there exists a constant $C>0$ such that for all $k\geq 1$, 
    \begin{align*}
        \|X\|_{k}\leq Ck^{\frac{1}{\alpha}}\|X\|_{\Psi_{\alpha}},
    \end{align*}
    where constant $C$ does not depend on $\alpha$, $k$, and $X$.
    \label{lem:moment_bound}
\end{lemma}

We remark that \cite{zhang2022sharper} proved that $$\|X\|_{k}\leq C_{\alpha}k^{\frac{1}{\alpha}}\|X\|_{\Psi_{\alpha}}$$ with 
$$C_\alpha=(e^{11/12}\alpha)^{-1/\alpha}\cdot\max_{k\geq 1}\l(\frac{2\sqrt{2\pi}}{\alpha}\r)^{1/k}\l(\frac{k}{\alpha}\r)^{1/(2k)}.$$ 
Notice that function $f(x):=x^{\frac{1}{x}}$ is bounded on $[1,\infty]$. Consequently, for all $k\geq 1$ and $\alpha\geq 1$, there exists a universal constant $C>0$ that does \textit{not} depend on $\alpha$ and $k$ such that $C_{\alpha}\leq C$.

\end{document}